\documentclass[twoside]{article}
\usepackage[utf8]{inputenc}
\usepackage{amsmath}
\numberwithin{equation}{section}
\usepackage{amsmath,amsfonts,color,enumitem,geometry,latexsym,mathrsfs,upgreek}
\usepackage[colorlinks,linkcolor=blue,citecolor=red,urlcolor=blue]{hyperref}
\usepackage[amsmath,thmmarks]{ntheorem}
\usepackage[pagewise]{lineno}
\usepackage{ulem}  
\newtheorem{claim}{Claim}
\newcommand{\Title}{Symbolic dynamics for certain non-invertible $C^{1+\beta}$ maps}
\newcommand{\Titleabbr}{}
\newcommand{\Keyword}{$C^{1+\beta}$ map, zero Lyapunov exponents, Markov partition, symbolic dynamics}

\newcommand{\Authora}{Jing Xun${}^\text{ $\ast$}$}
\newcommand{\AuthorabbrA}{Xun}
\newcommand{\Authorb}{Yifan Zhang }
\newcommand{\AuthorabbrB}{ZHANG}
\newcommand{\Authorc}{Yujun Zhu}
\newcommand{\AuthorabbrC}{ZHU}

\usepackage{fancyhdr}
\title{\Title
\footnotetext{
\setlength{\parindent}{2em}
{2020 Mathematics Subject Classification:
37D05, 37B10, 37C50.}\\
\indent \emph{Keywords and phrases:} \Keyword.\\
\indent{} $^*$ The corresponding author.\\
\indent{}This research is supported by NSFC (No: 12571206, 12526102)}.}
\author{\Authora, \Authorb and \Authorc}
\date{}
{
\theorembodyfont{\rm}
\newtheorem{mdef}{Definition}[section]
\newtheorem{rmk}{Remark}

}
\newtheorem*{mainthm}{Main Theorem}
\newtheorem{thm}{Theorem}[section]

\newtheorem{prop}[thm]{Proposition}
\newtheorem{lem}[thm]{Lemma}
\newtheorem{cor}[thm]{Corollary}

\theoremstyle{nonumberchange}
\theorembodyfont{\rm}
\theoremsymbol{\mbox{$\Box$}}
\newtheorem{pf}{Proof}

\begin{document}
\vskip 5.1cm

\maketitle

\begin{center}
\small
 \textsl{School of Mathematical Sciences, Xiamen University, Xiamen, 361005, PR~China}

\end{center}
\begin{center}
\begin{minipage}{130mm}

{\bf Abstract}:
Let $f$ be a non-invertible $C^{1+\beta}(\beta>0)$ map with zero Lyapunov exponents and singularities on a closed Riemannian manifold $M$. We consider the symbolic dynamics of $f$. Combining the techniques in recent works of Sarig, Ovadia and Araujo-Lima-Poletti \cite{Sarig, Ovadia1, Ovadia2, Lima}, we construct a countable Markov partition for the invariant set consisting of summable points of the inverse limit space of $(M, f)$ and show that there exists a finite-to-one symbolic extension for $f$ on the corresponding subset of $M$.
\end{minipage}
\end{center}
\maketitle

\section{Introduction}
It is well known that symbolic systems are important dynamical systems which have rich topological and statistical properties. More relevantly, symbolic systems play important roles in the investigation of many classical systems, especially uniformly hyperbolic diffeomorphisms. Let $f$ be a diffeomorphism on a closed Riemannian manifold $M$. Once a Markov partition of $M$ (resp. an $f$-invariant subset $\Lambda$) is obtained, then we get a semi-conjugacy between a topological Markov shift and $f$ (resp. $f|_{\Lambda}$) which give a correspondence between $f$-invariant measures and shift-invariant measures. In 1960s, Adler and Weiss introduced finite Markov partitions for hyperbolic toral automorphisms \cite{Adler1, Adler2}. Later, Sina$\check{\text{\i}}$ and Bowen constructed finite Markov partitions for Anosov diffeomorphisms and basic sets of  Axiom A diffeomorphisms \cite{Sinai1,Sinai2,Bowen,Bowen2} respectively.

Recent years, researchers have made great efforts to apply symbolic dynamics to systems with weak hyperbolicity. One of the important progresses in this topic was made by Sarig \cite{Sarig}. He investigated symbolic dynamics for any $C^{1+\beta}$ surface diffeomorphism $f$ with positive topological entropy. In this case, $f$ is nonuniformly hyperbolic with respect to any ergodic measures of maximal entropy. Sarig constructed countable Markov partitions for the $\chi$--large subset (the subset is $\chi$--large if it is support on every ergodic invariant probability measure $\mu$ whose entropy is greater than $\chi>0$), and then proved two conjectures for surface diffeomorphisms, one due to Katok \cite{Katok3} and the other due to Buzzi \cite{Buzzi}, concerning the number of periodic points and the number of measures of maximal entropy. Later, Ovadia generalized Sarig's methods to higher dimension cases including $C^{1+\beta}$ diffeomorphisms which is non-uniformly hyperbolic or whose orbits may have zero Lyapunov exponents but still demonstrates some sensitivity to initial conditions \cite{Ovadia1, Ovadia2}. Araujo, Lima and Poletti considered this topic for the non-invertible and singular nonuniformly hyperbolic systems \cite{Lima}. There are many related works, which we will not list here.

Bowen's idea and Pesin theory are the fundamentals of Sarig's Method. It is well known that Bowen's construction of Markov partitions for uniformly hyperbolic diffeomorphisms uses Anosov's shadowing theory for pseudo-orbits. For the non-uniformly hyperbolic case, Pesin represented the system as a perturbation of a uniform hyperbolic system on a local chart by using the Lyapunov norm which is induced by the Lyapunov inner product $\langle \cdot, \cdot \rangle'$ (see Subsection \ref{pesin}). Nowadays, the local charts are called the Pesin charts, and the size of the Pesin charts depends on the hyperbolicity at the point. Sarig \cite{Sarig} developed a refined shadowing theory based on Pesin theory by introducing the concept of $\epsilon$-chains  which serve as the counterpart to Bowen’s pseudo-orbits. The $\epsilon$-chains are stronger than Bowen’s pseudo-orbits: they not only require the distance between points to be close, but also demand the hyperbolicity to be `close'.

Adapting Sarig's strategy, Ovadia investigated the symbolic dynamics in the higher dimensional settings \cite{Ovadia1, Ovadia2}. In \cite{Ovadia1}, Ovadia dealt with higher dimensional nonuniformly hyperbolic diffeomorphisms. Comparing with the 2-dimensional case,  the main difficulty in the higher dimensional is  how to compare the $\epsilon$-chains that are shadowing the same point. Ovadia introduced a linear transformation so that  the scaling functions can be compared (the definition of scaling function can be found in (\ref{scaling})).
In \cite{Ovadia2}, Ovadia considered the diffeomorphisms with deteriorating hyperbolicity, i.e., there exist orbits with zero Lyapunov exponents. The difference from the non-uniformly hyperbolic systems with non-zero Lyapunov exponents is that there exists no uniform estimate for the Pesin charts. To overcome this obstacle, Ovadia found a $I$--ladder function that can be seen as the bound of the Pesin charts (the bound changes with the points) and required a more restrictive “temperability” property of the orbits, where temperability refers to the asymptotic growth-rate of the norm of the inverse change of coordinates. Then there is a  shadowing theory for orbits which may have 0 Lyapunov exponents. Hence, Ovadia constructed a countable Markov partition for a class of orbits that may have 0 Lyapunov exponents. This induced a finite-to-one symbolic coding, which lifts the geometric potential with summable variations.

The investigation of the symbolic dynamics for non-invertible endomorphisms has also attracted a lot of attention. In 1990s, Qian and Zhang studied the ergodic theory for Axiom A endomorphisms via the Markov partitions \cite{Qian1} of the natural extension. Natural extension is a basic tool in the study of non-invertible systems, we will give the information of natural extension in Subsection \ref{extension}. For the systematical investigation of the ergodic theory for non-invertible $C^2$ endomorphisms we refer to the  monograph of Qian, Xie and Zhu \cite{Qian}. Via the natural extension, Araujo, Lima and Poletti systematically investigated the symbolic dynamics of $C^{1+\beta}$  non-uniformly hyperbolic maps with singularities on the smooth, possibly disconnected, and/or with
boundary and/or not compact, Riemannian manifold $M$ with finite diameter \cite{Lima}. In this setting, the Pesin charts are smaller in size. Hence, they made some assumptions about the manifold and the map so that the size of the Pesin charts can be controlled.  It is worth pointing out that the generality of the setup in \cite{Lima} covers many classical examples, such as geodesic flows in closed manifolds, multidimensional billiard maps, and Viana maps, and includes most of the recent results of the literature.

The main purpose of this paper is to consider the symbolic dynamics for more general systems --- non-invertible $C^{1+\beta}$ maps with zero Lyapunov exponents and singularities. The simultaneous occurrence of zero Lyapunov exponents and singularities require us to combine and finely adapt the techniques of Sarig, Ovadia and Araujo-Lima-Poletti \cite{Sarig, Ovadia1, Ovadia2, Lima} to our setting. The paper is organized as follows. In Section \ref{Preliminaries}, we introduce the basic properties of the natural extension. In Section \ref{set}, by Pesin theory, we prove the shadowing theorem for the natural extension, then we establish a countable Markov partition on the set 0-summ. It induces a finite-to-one symbolic coding by the method of Bowen and Sina$\check{\text{\i}}$ (the definition of 0-summ see Definition \ref{summ}). In Section \ref{example}, we introduce the canonical application and give a example satisfying our conditions.

\subsection{\texorpdfstring{Statement of the Main results}{Statement of the Main results}}
Let $M$ be a  closed (i.e., compact and  boundaryless) Riemannian manifold with dimension  greater than or equal to 2, and $f:M\rightarrow M$ be a $C^{1+\beta}$ endomorphism with singularities. Let $\Gamma_{\infty}:=\{x\in M: d_xf \text{ is degenerate}\}$.
For any $x\in M\backslash \Gamma_{\infty}$, denote $f_x^{-1}$ as the inverse branch of $f$ that sends $f(x)$ to $x$. Assume that the diameter of $M$ is smaller than 1, and $f$ satisfies the following properties.
For constants $a,K > 1$ such that for all $x \in M$ with $x, f(x) \notin \Gamma_{\infty}$, there are sets $D,D'\subset M$  and $\tau(x)$ with
$\min\{d(x, \Gamma_{\infty} )^a, d(f(x), \Gamma_{\infty} )^a\} < \tau(x) < 1$
s.t. for $D_x := B(x, 2\tau(x))$ and $\mathfrak{E}_x := B(f(x), 2\tau(x))$ the following holds:
\begin{enumerate}[label=(A\arabic*)]
    \item $f|_{D_x}$ is a diffeomorphism onto its image; $f_{x}^{-1}|_{\mathfrak{E}_x}$ is a well-defined diffeomorphism onto its image. \label{A1}
    \item $d(x, \Gamma_{\infty} )^a \leq \|d_yf\| \leq d(x, \Gamma_{\infty} )^{-a}\ \forall y \in D_x$ and $d(x, \Gamma_{\infty} )^a \leq \|d_zf^{-1}_{x}\| \leq d(x, \Gamma_{\infty} )^{-a}\ \forall z \in \mathfrak{E}_x$.\label{A2}
    \item $\| \Theta_D \circ d_yf\circ \nu_{y}- \Theta_D \circ d_zf\circ \nu_{z} \| \leq Kd(y, z)^\beta\ \forall y, z \in D_x$ and $\left\| \Theta_{D'} \circ d_{y'}f^{-1}_{x}\circ \nu_{y'} - \Theta_{D'} \circ d_{z'}f^{-1}_{x}\circ\nu_{z'}\right\| \leq Kd(y', z')^\beta\ \forall y, z \in \mathfrak{E}_x$ where $y,z\in D, y',z'\in D'$ (the definitions of $ \Theta_{D/D'}$ and $\nu_{y/y'}$ can be found in Section \ref{Preliminaries}).\label{A3}
\end{enumerate}
The assumption \ref{A2} guaranties that the orbits cannot approach the singularity set $\Gamma_\infty$ exponentially fast, since the size of the Pesin charts decreases as $x$ approaches $\Gamma_{\infty}$. For simplicity, we assume the multiplicative constants is 1.

Let $\mathcal{G}$ be a countable directed graph with vertices $\mathcal{V}$ and edges $\mathcal{E}$. For any vertex, there is at least one ingoing and one ongoing edge. The \textit{topological Markov shift} associated to $\mathcal{G}$ is the pair $(\Sigma,\sigma)$, where the set is
$$\Sigma=\Sigma(\mathcal{G}):=\{(v_i)_{i\in\mathbb{Z}}\in\mathcal{V}^{\mathbb{Z}}:v_i\rightarrow v_{i+1},\forall i\in\mathbb{Z}\},$$
and the map $\sigma$ is the left-shift. For simplicity, we denote $(v_i)_{i\in\mathbb{Z}}$ as $\uline{v}$, and endow the pair $(\Sigma,\sigma)$ with the natural metric $d(\uline{u},\uline{v}):=\exp(-\min\{|i|:u_i\neq v_i\})$ where $\uline{u},\uline{v}\in \Sigma$.  As in \cite{Sarig}, let
\begin{align*}
     \Sigma^{\#}:=&\{(v_i)_{i\in \mathbb{Z}}\in \Sigma: \exists v,w\in \mathcal{V}\text{ s.t. } v_i=v \text{ for infinitely many } i>0, \\
    &\text{ and } v_i=w \text{ for infinitely many } i<0\}.
\end{align*}
\begin{mainthm}[Theorem \ref{main}]\label{main theorem}
    Assume the system $(M,f)$ satisfies the above assumptions, and $(M^f,\hat{f})$ is the natural extension of $(M,f)$. For small enough $\epsilon>0$, there is a subset $RST\subset M^f$,  a locally compact topological Markov shift $(\tilde{\Sigma},\tilde{\sigma})$ and a uniformly continuous map $\tilde{\pi}:\tilde{\Sigma}\rightarrow M^f$ such that:

    (1) $\tilde{\pi}(\tilde{\Sigma}^{\#})=RST$.

    (2) $\tilde{\pi}$ is a finite-to-one map.

    (3) $\tilde{\pi}\circ \sigma=\hat{f}\circ \tilde{\pi}$.

    (4) $\tilde{\pi}$ is uniformly continuous on $\tilde{\Sigma}$.
\end{mainthm}

By the above theorem and the properties of natural extension, the following corollary holds.
\begin{cor}
    Assume the system $(M,f)$ satisfies the above assumptions, there exists a finite-to-one  symbolic extension for $(M,f)$, which lifts the geometric potential with summable variations.
\end{cor}
\section{Preliminaries}\label{Preliminaries}
Let $\exp_x:T_xM\rightarrow M$ be the exponential map. Due to the compactness of the manifold $M$, there are numbers $r=r(M)$ such that
$$\exp_x: \{\xi\in T_xM:|\xi|<r\}\rightarrow B(x,r)\subset M$$
is a $C^\infty$ diffeomorphism for every $x\in M$, where $B(x,r)=\{y\in M:d(x,y)<r\}$. Let $B_r(0):=\{\xi\in T_xM:|\xi|<r\}$.  Take $r$ small enough so that the maps $\exp_x:B_{r}(0)\rightarrow M$ and $(x,y)\mapsto \exp_x^{-1}(y)$ on $B(z,r)\times B(z,r)$ for any $z\in M$ have a uniform Lipschitz constant of 2, $\|d_v\exp_x\|\leq 2,\ \forall v\in B_{r}(0)$ and $\|d_y\exp_x^{-1}\|\leq 2,\ \forall\ y\in B(x,r)$.

By the standard theory of differential geometry, we can find for any $x \in M$, an open neighborhood $D:=B(x,r')(r'<r)$ and a smooth map $\Theta_D : TD \to \mathbb{R}^d$ such that:
\begin{enumerate}[label=(\arabic*)]
    \item $\Theta_D : T_yM \to \mathbb{R}^d$ is a linear isometry  $\forall y \in D$.
    \item Define $\nu_y := (\Theta_D|_{T_yM})^{-1} : \mathbb{R}^d \to T_yM$, then the map $(y, u) \mapsto (\exp_y \circ \nu_y)(u)$ is smooth and Lipschitz on $D \times B_r(0)$ w.r.t the metric $d(y, y') + |u - u'|$. Set $L_1$ be the uniform Lipschitz constant for the map $(y,u)\mapsto(\exp_y \circ \nu_y)(u)$ on $D\times B_r(0)$.
    \item $y \mapsto \nu_y^{-1} \circ \exp_y^{-1}$ is Lipschitzian on $D$. Set $L_2$ be the uniform Lipschitz constant for the map $y \mapsto \nu_y^{-1} \circ \exp_y^{-1}$ on $D$.
\end{enumerate}
Let $\mathcal{D}$ be a countable cover of $M\backslash \Gamma_\infty$, consisting of elements which satisfy the above properties. Denote $\omega(D)< r$ to be the Lebesgue number of $\mathcal\mathcal{D}$. Since $M$ is compact, $\exists$ constants $L_3,L_4,L_5$ and $D\in\mathcal{D}$ such that $\forall y_1,y_2\in B(x,r)\subset D$,
\begin{gather}
    \|d_{v_1}(\Theta_D\circ \exp_{y_1}\circ \nu_{y_1})-d_{v_2}(\Theta_D\circ \exp_{y_2}\circ \nu_{y_2})\|\leq L_3(|v_1-v_2|+d(y_1,y_2))\ \forall |v_1|\leq r,|v_2|\leq r,\label{exp1}\\
    \|d_{z_1}(\Theta_D\circ \exp_{y_1}^{-1}\circ \nu(z_1))-d_{z_2}(\Theta_D\circ \exp_{y_2}^{-1}\circ \nu(z_2))\|\leq L_4(d(y_1,y_2)+d(z_1,z_2)), \ \forall z_1,z_2\in B(x,r),\label{exp2}
\end{gather}
and
\begin{equation}\label{exp5}
    \begin{aligned}
        &\quad \|[d_{z_1}(\Theta_D\circ \exp_{y_1}^{-1}\circ \nu(z_1))-d_{z_1}(\Theta_D\circ \exp_{y_2}^{-1}\circ \nu(z_1))]\\
        &-[d_{z_2}(\Theta_D\circ \exp_{y_1}^{-1}\circ \nu(z_2))-d_{z_2}(\Theta_D\circ \exp_{y_2}^{-1}\circ \nu(z_2))]\|\leq L_5d(y_1,y_2)d(z_1,z_2).
    \end{aligned}
\end{equation}
Without loss of generality, we assume $L_1,L_2,L_3,L_4,L_5$ are larger than one.

The norm $|\cdot|$ on $\mathbb{R}^d$ is the canonical Euclidean norm, and the norm $|\cdot|$ on the tangent space $T_xM\ \forall x\in M$ is the norm induced by the Riemannian inner product. For a linear transformation $A:\mathbb{R}^n\rightarrow\mathbb{R}^m$, $\|A\|:=\sup_{v\in\mathbb{R}^n, v\neq 0}\frac{|Av|}{|v|}$. For an open bounded set $U\subset\mathbb{R}^n$ and $H:U\rightarrow\mathbb{R}^m$, the $C_0$-norm is $\|H\|_{C^0}:=\sup_{v\in U}|H(v)|$ and the $C^{1+\frac{\beta}{2}}$-norm is
\begin{align*}
        \|H\|_{C^{1+\frac{\beta}{2}}}:=\|H\|_{C^0}+\|d_{\cdot}H\|_{C^0}+H\ddot{o}l_{\frac{\beta}{2}}(d_{\cdot}H)=\sup_{v \in U} | H(v) | + \sup_{v \in U} \| d_vH \| + \sup_{\substack{v_1,v_2 \in U, \\ v_1 \neq v_2}} \frac{\| d_{v_1}H - d_{v_2}H \|}{| v_1 - v_2 |^{\beta/2}}.
\end{align*}

\subsection{\texorpdfstring{Natural extension}{Natural extension}}\label{extension}
In this part, we will review some necessary concepts and results for natural extension, especially for smooth maps. The details can be found in the book \cite{Qian}. Let $M^f$ be the subset of $M^{\mathbb{Z}}$ consisting of all full orbits, i.e.
$$M^f=\{\tilde{x}=(x_i)_{i\in\mathbb{Z}}|x_i\in M,\  f(x_i)=x_{i+1},\  \forall i\in\mathbb{Z}\}.$$
Endow $M^f$ with the product topology and the distance $d(\tilde{x},\tilde{y}):=\sup\{2^{i}d(x_i,y_i):i\leq 0\}$, then $M^f$ is a closed subset of $M^{\mathbb{Z}}$. Let $p: M^f\rightarrow M$ denote the natural projection:
$$p(\tilde{x})=x_0,\ \forall\tilde{x}\in M^f,$$
and $\hat{f}:M^f\rightarrow M^f$ as the shift homeomorphism. $M^f$ is called \textit{the inverse limit space} (or \textit{orbit space}) of system $(M,f)$, and $\hat{f}$ is called the \textit{natural extension} of system $f$.

It is known that $p*$ induced by the projection $p$ is a bijection between $M_{\hat{f}}(M^f)$ and $M_f(M)$, where $M_{\hat{f}}(M^f)$   (resp. $M_f(M)$) denotes the set of all $\hat{f}$ (resp. $f$)-invariant probability measures on $M^f$ (resp. $M$) with the weak* topology. For any $f$-invariant Borel probability measure $\mu$ on $M$, there exists a unique $\hat{f}$-invariant Borel probability measure $\tilde{\mu}$ on $M^f$ such that $p\tilde{\mu}=\mu$.

Let $E=\bigcup_{\tilde{x}\in M^f}E_{\tilde{x}}$ be the pullback bundle of the tangent bundle $TM=\bigcup_{x\in M}T_xM$ by the projection $p$, and there is a natural isomorphism between fibers $E_{\tilde{x}}$ and $T_{p\tilde{x}}M$. Hence, for $\tilde{x}\in M^f$, define  $d_{\tilde{x}}\hat{f}:E_{\tilde{x}}\rightarrow E_{\hat{f}\tilde{x}}$ by $d_{\tilde{x}}\hat{f}:= d_{p\tilde{x}}f$, and for $\tilde{x}=(x_i)_{i\in\mathbb{Z}}\in M^f\backslash\bigcup_{n\in\mathbb{Z}}\hat{f}^n(p^{-1}(\Gamma_{\infty}))$, define
\[
d_{\tilde{x}}\hat{f}^n :=
\begin{cases}
d_xf^n & \text{if } n \geq 0, \\
(d_{x_{-n}}f)^{-1} \cdots (d_{x_{-2}}f)^{-1}(d_{x_{-1}}f)^{-1}, & \text{if } n < 0.
\end{cases}
\]
Then $d_{\tilde{x}}\hat{f}^{m+n}=d_{\hat{f}^n\tilde{x}}\hat{f}^m\circ d_{\tilde{x}}\hat{f}^n$, for any $n,m\in\mathbb{Z}.$ This implies $d\hat{f}$ is an invertible linear cocycle on a fiber bundle over $M^f \backslash \bigcup_{n\in\mathbb{Z}}\hat{f}^n(p^{-1}(\Gamma_{\infty}))$. Furthermore, the Lyapunov spectrum of $d\hat{f}$ is the same as $df$.
\section{The subset with Markovian symbolic dynamics}\label{set}
Let $M$ be a closed Riemannian manifold with dimension $d\, (d\geq 2)$, and $f:M\rightarrow M$ be a $C^{1+\beta}$ endomorphism with a singularity set $\Gamma'$ and it satisfies the assumptions \ref{A1}-\ref{A3}. In this section, we will consider the natural extension $(M^f,\hat{f})$ of the system $(M,f)$, and construct a locally finite Markov partition on certain subset by proving the shadowing theorem in our setting. Since $\hat{f}, d\hat{f}$ are invertible, for $x\in M$ and $x,f(x)\notin \Gamma_{\infty}$, we introduce the following definitions which are adapted from \cite{Ovadia2}.
\begin{mdef}\label{summ}
A point $\tilde{x}\in M^f\backslash\bigcup_{n\in\mathbb{Z}}\hat{f}^n(p^{-1}(\Gamma_{\infty}))$ is called \textit{summable} if there is a unique decomposition $E_{\tilde{x}}=E_{\tilde{x}}^s \oplus E_{\tilde{x}}^u$ such that
$$\sup_{\substack{\xi^s \in E^s_{\tilde{x}},\ |\xi^s|=1}} \sum_{m \geq 0} |d\hat{f}^m \xi_s|^2 < \infty, \quad
\sup_{\substack{\xi^u \in E^u_{\tilde{x}},\ |\xi^u|=1}} \sum_{m \geq 0} |d\hat{f}^{-m} \xi_u|^2<\infty.$$
\end{mdef}
The set of summable points is called \textit{0-summ}. Clearly, the set 0-summ is $\hat{f}$ -invariant, and it may contain the points with zero Lyapunov exponent.

For any $\tilde{x}\in M^f\backslash\bigcup_{n\in\mathbb{Z}}\hat{f}^n(p^{-1}(\Gamma_{\infty}))$, the scaling parameters defined by
$$s(\tilde{x})=\sup_{\substack{\xi^s \in E^s_{\tilde{x}},\ |\xi^s|=1}}S(\tilde{x},\xi^s),\ u(\tilde{x})=\sup_{\substack{\xi^u \in E^u_{\tilde{x}},\ |\xi^u|=1}}U(\tilde{x},\xi^u),$$
where
\begin{equation}\label{scaling}S^2(\tilde{x},\xi^s)=2\sum_{m=0}^\infty|d_{\tilde{x}}\hat{f}^m\xi^s|^2\text{ and } U^2(\tilde{x},\xi^u)=2\sum_{m=0}^\infty|d_{\tilde{x}}\hat{f}^{-m}\xi^u|^2.
\end{equation}
Note that the parameters $s(\tilde{x})$ and $u(\tilde{x})$ are finite when $\tilde{x}$ is summable.

For handling the points with zero Lyapunov exponent, Ovadia introduced the ladder function $I(t)$ \cite{Ovadia2}. The following proposition lists the properties of the function $I(t)$.
\begin{prop}\label{I}
Given $\Gamma>0$ and $\gamma>\frac{20}{\beta}$, the ladder function $I(t):=te^{\Gamma t^{\frac{1}{\gamma}}},\ I:(0,1)\rightarrow(0,\infty)$
satisfies the following properties:
\begin{enumerate}[label=(\arabic*)]
    \item $I$ and $I^{-1}$ are strictly increasing,  $I(t)>t$ on $(0,1)$ and $I^{-1}(t)<t$ on $(0,e^{\Gamma})$;
    \item $I^{-1}(t)\geq te^{-\Gamma t^{\frac{1}{\gamma}}}$;
    \item $ \sum_{n \geq 0} \left( I^{-n}(t) \right)^{\frac{1}{\gamma}} = \infty \  \text{for all } t > 0$;
    \item $\forall \tau > 0, \forall t > 0,\ \sum_{n \geq 0} \left( I^{-n}(t) \right)^{\frac{1+\tau}{\gamma}} < \infty$;
    \item $I^{\frac{1}{4}}:(0,1)\rightarrow(0,\infty)$ is well-defined, and it is a strictly increasing $C^1$ positive function such that $I^{\frac{1}{4}}(t)>t$ for all $t\in(0,1)$;
    \item $I^{-\frac{1}{4}}$ is strictly increasing on $(0,e^{\Gamma})$, $I^{-\frac{1}{4}}(t)<t$, and $I^{-\frac{n}{4}}(t)\xrightarrow{n \to \infty}0\ \forall t\in(0,1)$.
\end{enumerate}
\end{prop}

Define $\rho:M^f\rightarrow\mathbb{R}$ by $\rho(\tilde{x}):=d(\{x_{-1},x_{0},x_{1}\},\Gamma_{\infty})$, for any $\tilde{x}=(x_i)_{i\in\mathbb{Z}}\in M^f$. Given $\epsilon>0$, let $\mathcal{I}:= \{I^{-\frac{\ell}{4}}(1)\}_{\ell \in \mathbb{N}}$. For summable point $\tilde{x}$, define parameter $Q_{\epsilon}(\tilde{x}) := \max\{q \in \mathcal{I} \,|\, q \leq \tilde{Q}_{\epsilon}(\tilde{x})\}$, where
$$\tilde{Q}_{\epsilon}(\tilde{x}) := \epsilon^{\frac{20}{\beta}}\rho(\tilde{x})^{\frac{8a}{\beta}}\cdot\|C_0(\tilde{x})^{-1} \|^{-2\gamma}.$$
\begin{rmk}
    Notice that the greater the constant $\gamma$, the smaller the parameter $Q(\tilde{x})$ becomes. And the lower bound of $\gamma$ is stronger than the lower bound in \cite{Ovadia2} since $\|df\|,\|df^{-1}_{x}\|$ can not be bounded by a constant for all $x,f(x)\notin \Gamma_{\infty}$. Furthermore, we assume that the diameter of $M$ is smaller than one, then $\rho(\tilde{x})\leq 1$. This means the shadowing theorem in this context is more refined than the previous versions in \cite{Ovadia2, Sarig}.
\end{rmk}
\begin{mdef}
    A summable point $\tilde{x}$ is called \textit{strongly temperable} if there is a function $q:\{\hat{f}^n(\tilde{x})\}_{n\in\mathbb{Z}}\rightarrow\mathcal{I}$ such that
\begin{enumerate}[label=(\arabic*)]
    \item $q\circ f = I^{\pm1}(q)$,
    \item $\forall n\in\mathbb{Z},\ q(\hat{f}^{n}(\tilde{x}))\leq Q_\epsilon(\hat{f}^{n}(\tilde{x}))$.
\end{enumerate}
\end{mdef}
\begin{mdef}
   A strongly temperable point $\tilde{x}$ is called \textit{recurrently strongly temperable} if there is a function $q:\{\hat{f}^n(\tilde{x})\}_{n\in\mathbb{Z}}\rightarrow\mathcal{I}$ such that
\begin{enumerate}[label=(3)]
    \item $\limsup_{n \to \pm\infty} q(\hat{f}^n(\tilde{x})) > 0$.
\end{enumerate}
Denote \textit{ST} as the set of strongly temperable points, and \textit{RST} as the set of recurrently strongly temperable points.
\end{mdef}

\subsection{\texorpdfstring{Pesin Theory}{Pesin Theory}}\label{pesin}
For $\tilde{x}\in$ 0-summ, there is a unique decomposition $E_{\tilde{x}}=E_{\tilde{x}}^s \oplus E_{\tilde{x}}^u$. Let $d_s(\tilde{x}),\ d_u(\tilde{x})$ be the dimensions of $E_{\tilde{x}}^s$ and $E_{\tilde{x}}^u$, respectively. And the functions $d_s,d_u$ are $\hat{f}$-invariant since the decomposition $E^s\oplus E^u$ is $d\hat{f}$-invariant. Denote the Riemannian inner product on $E_{\tilde{x}}$ by $\langle \cdot, \cdot \rangle$, and define the Lyapunov inner product $\langle \cdot, \cdot \rangle'$ on 0-summ by the following identities:
$$\langle \xi^s, \eta^s \rangle' := 2 \sum_{m=0}^{\infty} \langle d_{\tilde{x}}\hat{f}^m \xi^s, d_{\tilde{x}}\hat{f}^m \eta^s \rangle,\ \text{for } \xi^s, \eta^s \in E^s(\tilde{x});$$
$$\langle \xi^u, \eta^u \rangle' := 2 \sum_{m=0}^{\infty} \langle d_{\tilde{x}}\hat{f}^{-m} \xi^u, d_{\tilde{x}}\hat{f}^{-m} \eta^u \rangle,\ \text{for } \xi^u, \eta^u \in E^u(\tilde{x});$$
and $\langle \xi,\eta \rangle'=0$, if $\xi\in E^s(\tilde{x}),\  \eta \in E^u(\tilde{x})$.

For $\tilde{x}\in$ 0-summ, we choose $C_0(\tilde{x}):\mathbb{R}^n\rightarrow E_{\tilde{x}}$ that satisfies the following conditions:
\begin{enumerate}[label=(\arabic*)]
    \item $C_0[\mathbb{R}^{d_s(\tilde{x})}\times\{0\}]=E^s(\tilde{x})$, and $C_0[\{0\}\times\mathbb{R}^{d_u(\tilde{x})}]=E^u(\tilde{x})$;
    \item $C_0(\tilde{x})$ is an isometry between $(\mathbb{R}^n,\langle \cdot, \cdot \rangle_{\mathbb{R}^n})$ and $(E(\tilde{x}),\langle \cdot, \cdot \rangle')$.
\end{enumerate}
\begin{lem}\label{C}
For $\tilde{x}\in$ 0-summ, the following holds:
\begin{enumerate}[label=(\arabic*)]
    \item $C_0(\tilde{x})$ is  contractive, and
$$\|C_0(\tilde{x})^{-1}\|^2= \sup_{\substack{\xi^s \in E^s(\tilde{x}),\  \xi^u \in E^u(\tilde{x}),\\  \xi^s + \xi^u\neq 0}} \frac{S^2(\tilde{x}, \xi^s) + U^2(\tilde{x}, \xi^u)}{|\xi^s + \xi^u|^2}.$$
    \item The map $D_0(\tilde{x})= C_0(\hat{f}(\tilde{x}))^{-1}\circ d_{\tilde{x}}\hat{f}\circ C_0(\tilde{x})$ is a block matrix
$$\begin{pmatrix}
  D_s(\tilde{x}) & 0 \\
  0 & D_u(\tilde{x})
\end{pmatrix},$$
where $D_{s\backslash u}(\tilde{x})$ is a $d_{s\backslash u}(\tilde{x})\times d_{s\backslash u}(\tilde{x})$ matrix. In addition, for all non-zero vectors $v_s \in \mathbb{R}^{d_s(\tilde{x})}\text{ and } v_u \in \mathbb{R}^{d_u(\hat{f}(\tilde{x}))}$,
\begin{align}\label{D1}
    \frac{\sqrt{2}}{2}d(x_0,\Gamma_{\infty})^{a}\leq \frac{|D_sv_s|}{|v_s|} \leq e^{-\frac{1}{s^2(\tilde{x})}},\quad \frac{\sqrt{2}}{2}d(x_0,\Gamma_{\infty})^{a} \leq \frac{|D^{-1}_u v_u|}{|v_u|} \leq e^{-\frac{1}{u^2(\hat{f}(\tilde{x}))}}.
\end{align}\label{f(x)}
    \item $\frac{\sqrt{2}}{2}d(p(\tilde{x}),\Gamma_{\infty})^{2a}\leq \frac{\|C_0(\hat{f}(\tilde{x}))^{-1}\|}{\|C_0(\tilde{x})^{-1}\|}\leq \sqrt{2}d(p(\tilde{x}),\Gamma_{\infty})^{-2a}$.\label{C(f)}
\end{enumerate}
\end{lem}
\begin{pf}
  (1) See \cite{Ovadia1} for the proof of part (1).

  (2) By the construction of $C_0(\tilde{x})$, $D_0(\tilde{x})$ has the block form. Next, we estimate $|D_s(\cdot)|$ and $|D_u(\cdot)|$. For $\xi^s \in E^s(\tilde{x})$, we have
 \begin{equation}\label{Ly1}
\left\langle d_{\tilde{x}} \hat{f} \xi^s, d_{\tilde{x}} \hat{f} \xi^s \right\rangle' = 2 \sum_{m=0}^{\infty} |d_{\tilde{x}} \hat{f}^{m+1} \xi^s|^2 = \left\langle \xi^s, \xi^s \right\rangle' - 2|\xi^s|^2,
\end{equation}
this implies
\begin{align*}
    \frac{\left\langle d_{\tilde{x}} \hat{f} \xi^s, d_{\tilde{x}} \hat{f} \xi^s \right\rangle'}{\left\langle \xi^s, \xi^s \right\rangle'} = 1 - \frac{2|\xi^s|^2}{\left\langle \xi^s, \xi^s \right\rangle'}.
\end{align*}
By assumption $\ref{A2}$, $\left\langle \xi^s, \xi^s \right\rangle' > 2\left(|\xi^s|^2 + |d_{\tilde{x}} \hat{f} \xi^s|^2\right) \geq 2\left(1+d(x_0,\Gamma_{\infty})^{2a}\right)|\xi^s|^2$, where $x_0=p\tilde{x}$. Hence,
\begin{align*}
    \frac{d(x_0,\Gamma_{\infty})^{2a}}{1+d(x_0,\Gamma_{\infty})^{2a}} \leq \frac{\langle d_{\tilde{x}} \hat{f} \xi^s, d_{\tilde{x}} \hat{f} \xi^s \rangle'}{\langle \xi^s, \xi^s \rangle'} \leq e^{\frac{-2|\xi^s|^2}{\langle \xi^s, \xi^s \rangle'}}\leq e^{\frac{-2}{S^2(\tilde{x},\frac{\xi^s}{|\xi^s|})}}\leq e^{\frac{-2}{s^2(\tilde{x})}}.
\end{align*}
Let $v_s\in\mathbb{R}^{d_s(\tilde{x})}$ s.t. $\xi^s=C_0(\tilde{x})v_s$, then
\begin{align*}
    \frac{|D_s(\tilde{x})v^s|^2}{|v^s|^2}
    = \frac{|C_0(\hat{f}(\tilde{x}))^{-1} d_{\tilde{x}} \hat{f} \xi^s|^2}{|C_0(\tilde{x})^{-1} \xi^s|^2}
    = \frac{|d_{\tilde{x}} \hat{f} \xi^s|'^2}{|\xi^s|'^2}
    \in \bigg[\frac{d(x_0,\Gamma_{\infty})^{2a}}{1+d(x_0,\Gamma_{\infty})^{2a}}, e^{\frac{-2}{s^2(\tilde{x})}}\bigg],
\end{align*}
where the norm $|\cdot|'$ is induced by the inner product $\langle \cdot,\cdot\rangle'$.

Similarly, we can estimate $|D_u(\cdot)|$: for $\eta^u\in E^u(\hat{f}(\tilde{x}))$
 \begin{align*}
\left\langle d_{\hat{f}(\tilde{x})} \hat{f}^{-1} \eta^u, d_{\hat{f}(\tilde{x})} \hat{f}^{-1} \eta^u \right\rangle'
= 2 \sum_{m=0}^{\infty} |d_{\hat{f}(\tilde{x})} \hat{f}^{-m-1} \eta^u|^2
= 2 \sum_{m=0}^{\infty} |d_{\hat{f}(\tilde{x})} \hat{f}^{-m} \eta^u|^2 - 2|\eta^u|^2=\left\langle \eta^u, \eta^u \right\rangle' - 2|\eta^u|^2,
\end{align*}
and $\left\langle \eta^u, \eta^u \right\rangle'
> 2\left(|\eta^u|^2 + |d_{\hat{f}(\tilde{x})} \hat{f}^{-1} \eta^u|^2\right)
\geq 2\left(1+d(x_0,\Gamma_{\infty})^{2a}\right)|\eta^u|^2$. Let $v_u\in\mathbb{R}^{d_u(\hat{f}(\tilde{x}))}$, and $\eta^u:=C_0(\hat{f}(\tilde{x}))v_u$, then
\begin{align*}
    \frac{|D_u(\tilde{x})^{-1}v^u|^2}{|v^u|^2}
    = \frac{|C_0(\tilde{x})^{-1} d_{\hat{f}(\tilde{x}))} \hat{f}^{-1} \eta^u|^2}{|C_0(\hat{f}(\tilde{x}))^{-1} \eta^u|^2}
    = \frac{|d_{\hat{f}(\tilde{x})} \hat{f}^{-1} \eta^u|'^2}{|\eta^u|'^2}
    \in \bigg[\frac{d(x_0,\Gamma_{\infty})^{2a}}{1+d(x_0,\Gamma_{\infty})^{2a}}, e^{\frac{-2}{u^2(\hat{f}(\tilde{x}))}}\bigg].
 \end{align*}
Since the diameter of $M$  is less than  one and $a>1$, we have $\frac{d(x_0,\Gamma_{\infty})^{2a}}{1+d(x_0,\Gamma_{\infty})^{2a}}> \frac{1}{2} d(x_0,\Gamma_{\infty})^{2a}$, thus part (2) is proved.

(3) According to part (1) and $d\hat{f}$ is invertible,
\begin{equation}\label{Ly2}
    \begin{aligned}
\|C_0(\hat{f}(\tilde{x}))^{-1}\|^2
& = \sup_{\substack{\xi^s \in E^s(\tilde{x}),\  \xi^u \in E^u(\tilde{x}),\\  \xi^s + \xi^u\neq 0}} \frac{S^2(\tilde{x}, d_{\tilde{x}}\hat{f}(\xi^s)) + U^2(\tilde{x}, d_{\tilde{x}}\hat{f}(\xi^u))}{|d_{\tilde{x}}\hat{f}(\xi^s) + d_{\tilde{x}}\hat{f}(\xi^u)|^2}\\
& =\sup_{\substack{\xi^s \in E^s(\tilde{x}),\  \xi^u \in E^u(\tilde{x}),\\  \xi^s + \xi^u\neq 0}} \frac{|d_{\tilde{x}}\hat{f}(\xi^s)|'^2+|d_{\tilde{x}}\hat{f}(\xi^u)|'^2}{|d_{\tilde{x}}\hat{f}(\xi^s+\xi^u)|^2}.
    \end{aligned}
\end{equation}
By equation (\ref{Ly1}),
\begin{equation}\label{Ly3}
    |d_{\tilde{x}}\hat{f}(\xi^s)|'^2\leq |\xi^s|'^2\leq 2d(p(\tilde{x}),\Gamma_\infty)^{-2a}|\xi^s|'^2.
\end{equation}By definition of the inner product $\langle \cdot,\cdot\rangle'$ and assumption \ref{A2}, $|d_{\tilde{x}}\hat{f}(\xi^u)|'^2=|\xi^u|'^2+2|d_{\tilde{x}}\hat{f}(\xi^u)|^2$ and $|d_{\tilde{x}}\hat{f}(\xi^u)|'^2\geq 2|d_{\tilde{x}}\hat{f}(\xi^u)|^2+2|\xi^u|^2\geq 2|d_{\tilde{x}}\hat{f}(\xi^u)|^2(1+d(p(\tilde{x}),\Gamma_{\infty})^{2a})$,  then $\frac{|\xi^u|'^2}{|d_{\tilde{x}}\hat{f}(\xi^u)|'^2}=1-\frac{2|d_{\tilde{x}}\hat{f}(\xi^u)|^2}{|d_{\tilde{x}}\hat{f}(\xi^u)|'^2}\geq \frac{d(x_0,\Gamma_{\infty})^{2a}}{1+d(x_0,\Gamma_{\infty})^{2a}}$. This implies
\begin{equation}\label{Ly4}
    |d_{\tilde{x}}\hat{f}(\xi^u)|'^2\leq 2d(p(\tilde{x}),\Gamma_\infty)^{-2a}|\xi^u|'^2.
\end{equation}
Hence, applying inequalities (\ref{Ly3}), (\ref{Ly4}) and assumption \ref{A2} to equation (\ref{Ly2}), we get $\frac{\|C_0(\hat{f}(\tilde{x}))^{-1}\|}{\|C_0(\tilde{x})^{-1}\|}\leq \sqrt{2}d(p(\tilde{x}),\Gamma_{\infty})^{-2a}$. Similarly, we have $\frac{\|C_0(\hat{f}(\tilde{x}))^{-1}\|}{\|C_0(\tilde{x})^{-1}\|}\geq \frac{\sqrt{2}}{2}d(p(\tilde{x}),\Gamma_{\infty})^{2a}$.
\end{pf}
\begin{rmk}
In the situation of diffeomorphisms, there exists a constant $F_0$ s.t. $\frac{\|C_0(\hat{f}(\tilde{x}))^{-1}\|}{\|C_0(\tilde{x})^{-1}\|}\in[F_0^{-1},F_0]$ in \cite{Ovadia2}. In our setting, there exist singularities, the   corresponding result does not hold. By part \ref{C(f)} of the above lemma, there is not a uniform control on $\frac{\|C_0(\hat{f}(\tilde{x}))^{-1}\|}{\|C_0(\tilde{x})^{-1}\|}$. Hence, there is a finer control for the size of the Pesin chart.
\end{rmk}
\begin{mdef}
    (Pesin charts) The \textit{Pesin chart} at $\tilde{x}\in$ 0-summ is
    $$\psi_{\tilde{x}}:=\exp_{p(\tilde{x})}\circ C_0(\tilde{x}): B_r(0)\rightarrow M,$$
    where $B_r(0)$ is an open ball  of the origin with radius $r$ in $\mathbb{R}^d$ w.r.t. Euclidean norm.
\end{mdef}

For $\tilde{x}=(x_i)_{i\in\mathbb{R}}\in$ 0-summ, let
$$\iota(\tilde{x}):=\min\bigg\{\frac{1}{2}rd(x_0,\Gamma_{\infty})^a,\tau(\tilde{x})\bigg\}.$$
Without loss of generality, assume $r<1$, then $\iota(\tilde{x})=\frac{1}{2}rd(x_0,\Gamma_{\infty})^a$ since $\tau(\tilde{x})\geq d(x_0,\Gamma_{\infty})^a$. Write $f_{\tilde{x}}^{-1}:=f_{x_0}^{-1}$ and $\iota(\tilde{x}):=\iota(x_0)$. Define maps $F_{\tilde{x}}:=\psi_{\hat{f}(\tilde{x})}^{-1}\circ f\circ \psi_{\tilde{x}}$ on $B_{\iota(\tilde{x})}(0)$, and $F_{\tilde{x}}^{-1}:=\psi_{\tilde{x}}^{-1}\circ f_{\tilde{x}}^{-1}\circ \psi_{\hat{f}(\tilde{x})}$ on $B_{\iota(\tilde{x})}(0)$. The two maps are well-defined on $B_{\iota(\tilde{x})}(0)$. Indeed, $\psi_{\tilde{x}}(B_{\iota(\tilde{x})}(0))\subset B(x_0,2\iota(\tilde{x}))$, $f(B(x_0,2\iota(\tilde{x})))\subset B(f(x_0),2\iota(\tilde{x})d(x_0,\Gamma_{\infty})^{-a})$ and $2\iota(\tilde{x})d(x_0,\Gamma_{\infty})^{-a}<r$, then $F_{\tilde{x}}$ is well-defined. Similarly, we have
$\psi_{\hat{f}(\tilde{x})}(B_{\iota(\tilde{x})}(0) )\subset B(f(x_0),2\iota(\tilde{x}))$, and $f_{\tilde{x}}^{-1}(B(f(x_0),2\iota(\tilde{x})))\subset B(x_0,2\iota(\tilde{x})d(x_0,\Gamma_{\infty})^{-a})$
since $\iota(\tilde{x})\leq \tau(\tilde{x})$. And $2\iota(\tilde{x})d(x_0,\Gamma_{\infty})^{-a})<r$, then $F_{\tilde{x}}^{-1}$ is well-defined on $B_{\iota(\tilde{x})}(0)$. In addition, the maps $F_{\tilde{x}},F_{\tilde{x}}^{-1}:B_{\iota(\tilde{x})}(0)\rightarrow\mathbb{R}^d$ are diffeomorphisms onto their images, and they are the inverse of each other.

\begin{thm}\label{F_{x}}
    For all small enough $\epsilon>0$, $\tilde{x}=(x_i)_{i\in\mathbb{R}}\in$ 0-summ, the maps $F_{\tilde{x}}:B_{Q(\tilde{x})}(0)\rightarrow\mathbb{R}^d\text{ and } F_{\tilde{x}}^{-1}:B_{Q(\hat{f}(\tilde{x}))}(0)\rightarrow\mathbb{R}^d$ satisfies:
    \begin{enumerate}[label=(\arabic*)]
      \item $F_{\tilde{x}}\circ \ F_{\tilde{x}}^{-1}=Id$ and $F_{\tilde{x}}^{-1}\circ F_{\tilde{x}}=Id$, where $Id$ is the identity map.
      \item $F_{\tilde{x}}^{\pm1}(0)=0$, and $d_{0}F_{\tilde{x}}^{\pm1}=D_0(\tilde{x})^{\pm1}$.\label{D0}
     \item There exist two maps $H^+,H^-$ s.t. $F_{\tilde{x}}=D_0(\tilde{x})+H^+,F_{\tilde{x}}^{-1}=D_0(\tilde{x})^{-1}+H^-$; the maps $H^{\pm}$ satisfy  $H^{\pm1}(0)=0,\ d_0H^{\pm1}=0$, and $\|H^{\pm1}\|_{C^{1+\frac{\beta}{2}}}<\epsilon$.\label{F}
     \item $\|dF_{\tilde{x}}^{\pm1}\|_{C^0}\leq 2\sqrt{2} d(x_0,\Gamma_{\infty})^{-a}$.
\end{enumerate}
\end{thm}
\begin{pf}
Notice that $Q(\tilde{x}),Q(\hat{f}(\tilde{x}))\leq \frac{1}{2}rd(x_0,\Gamma_\infty)^a\leq \iota(\tilde{x})$ for small enough $\epsilon$. Thus, the maps $F_{\tilde{x}},F_{\tilde{x}}^{-1}$ are well-defined on $B_{Q(\tilde{x}}(0)),B_{Q(\hat{f}(\tilde{x})}$, respectively.

    (1) It is equivalent to prove that the maps $F_{\tilde{x}}\circ \ F_{\tilde{x}}^{-1}, F_{\tilde{x}}^{-1}\circ F_{\tilde{x}}$ are well-defined, i.e. $F_{\tilde{x}}(B_{Q(\tilde{x})}(0))\subset B_{\iota(\tilde{x})}(0)$ and $F_{\tilde{x}}^{-1}(B_{Q(\hat{f}(\tilde{x}))}(0))\subset B_{\iota(\tilde{x})}(0)$. According to the assumption \ref{A2}, we get
    $$ F_{\tilde{x}}^{-1}(B_{Q(\hat{f}(\tilde{x}))}(0))\subset B_{4d(x_0,\Gamma_{\infty})^{-a}\|C_0(\tilde{x})^{-1}\|Q(\hat{f}(\tilde{x}))}(0).$$
    And
\begin{equation*}
    \begin{aligned}
4d(x_0,\Gamma_{\infty})^{-a}\|C_0(\tilde{x})^{-1}\|Q(\hat{f}(\tilde{x}))
&\leq 4\sqrt{2}d(x_0,\Gamma_{\infty})^{-3a}\|C_0(\hat{f}(\tilde{x}))^{-1}\|Q(\hat{f}(\tilde{x}))\text{ (by Lemma \ref{C}(3))}\\
&\leq \epsilon^{\frac{20}{\beta}}d(x_0,\Gamma_{\infty})^{a}\|C_0(\hat{f}(\tilde{x}))^{-1}\|^{1-2\gamma}\\
&\leq \epsilon^{\frac{20}{\beta}}d(x_0,\Gamma_{\infty})^{a},
    \end{aligned}
\end{equation*}
    the last inequality holds since $\gamma>\frac{20}{\beta}, \text{ and } C_0(\tilde{x}) \text{ is  contractive for all }\tilde{x}\in$0-summ. Hence,
    $$4d(x_0,\Gamma_{\infty})^{-a}\|C_0(\tilde{x})^{-1}\|Q(\hat{f}(\tilde{x}))
    \leq \iota(\tilde{x})$$
   for small enough $\epsilon$. Similarly, we get
   $$ F_{\tilde{x}}(B_{Q(\tilde{x})}(0))\subset B_{4d(x_0,\Gamma_{\infty})^{-a}\|C_0(\hat{f}(\tilde{x}))^{-1}\|Q(\tilde{x})}(0),$$
   and
   $$4d(x_0,\Gamma_{\infty})^{-a}\|C_0(\hat{f}(\tilde{x}))^{-1}\|Q(\tilde{x})\leq 4\sqrt{2}d(x_0,\Gamma_{\infty})^{-3a}\|C_0(\tilde{x})^{-1}\|Q(\tilde{x})\leq \epsilon^{\frac{20}{\beta}}d(x_0,\Gamma_{\infty})^{a}\leq \iota(\tilde{x})$$
   for small enough $\epsilon$.

   (2) For any $x\in M$, $\exp_x(0)=x$ and $d_0\exp_x=Id$, then part (2) holds.

   (3) Define $H^{\pm1}=F_{\tilde{x}}^{\pm1}-D_0(\tilde{x})^{\pm1}$, then $dH^{\pm1}=dF_{\tilde{x}}^{\pm1}$.
For any $u,v \in  B_{Q(\tilde{x})}(0)$,
\begin{equation*}
    \begin{aligned}
    \| d_uH^{+} - d_vH^+ \|&=\|d_uF_{\tilde{x}}-d_vF_{\tilde{x}}\|\\
    &\leq\|C_0(\hat{f}(\tilde{x}))^{-1}\|\cdot\|d_u(\exp_{f(x_0)}^{-1}\circ f\circ \exp_{x_0})-d_v(\exp_{f(x_0)}^{-1}\circ f\circ \exp_{x_0})\|\cdot\|C_0(\tilde{x})\|.
    \end{aligned}
\end{equation*}
Assume $z_1:=\exp_{x_0}\circ C_0(\tilde{x})(u), z_2:=\exp_{x_0}\circ C_0(\tilde{x})(v)\in D$ and $f(z_1),f(z_2)\in D'$. where $D,D'\in \mathcal{D}$. Using assumption \ref{A2} and inequality (\ref{exp1}), we have
\begin{equation}\label{H1}
    \begin{aligned}
        &\quad \|d_{z_1}(\exp_{f(x_0)}^{-1}\circ f)\|\cdot\|\Theta_D d_{C_0(\tilde{x})(u)}\exp_{x_0}-\Theta_D d_{C_0(\tilde{x})(v)}\exp_{x_0}\|\\
        &\leq 2L_3d(x_0,\Gamma_{\infty})^{-a}|C_0(\tilde{x})(u)- C_0(\tilde{x})(v)|\leq 2L_3d(x_0,\Gamma_{\infty})^{-a}|u-v|.
    \end{aligned}
\end{equation}
 Applying assumption \ref{A3}, we have
\begin{equation}\label{H2}
    \begin{aligned}
        &\quad \|d_{f(z_1)}\exp_{f(x_0)}^{-1}\|\cdot\|\Theta_{D'}\circ d_{z_1}f\circ\nu_{z_1}-\Theta_{D'}\circ d_{z_2}f\circ\nu_{z_2}\|\cdot\|d_{C_0(\tilde{x})(v)}\exp_{x_0}\|\\
        &\leq 4Kd(z_1,z_2)^{\beta}\leq 8K|u-v|^\beta.
    \end{aligned}
\end{equation}
Using assumption \ref{A2} and inequality (\ref{exp2}), we get
\begin{equation}\label{H3}
    \begin{aligned}
        &\quad \|d_{f(z_1)}\exp_{f(x_0)}^{-1}\circ \nu_{f(z_1)}-d_{f(z_2)}\exp_{f(x_0)}^{-1}\circ \nu_{f(z_2)}\|\cdot\|d_{z_2}(f\circ\exp_{x_0})\|\\
        &\leq 2L_4d(x_0,\Gamma_{\infty})^{-a}d(f(z_1),f(z_2))\leq 4L_4d(x_0,\Gamma_{\infty})^{-2a}|u-v|.
    \end{aligned}
\end{equation}
 For any $u,v \in  B_{Q(\tilde{x})}(0)$, $|u-v|\leq |u-v|^\beta$ since $u,v\in Q(\tilde{x})$ and $Q(\tilde{x})\ll1$. According to the inequalities (\ref{H1}), (\ref{H2}) and (\ref{H3}),
\begin{equation}\label{F0}
\begin{aligned}
    \| d_uH^+ - d_vH^+  \|&\leq \|C_0(\hat{f}(\tilde{x})^{-1}\|\big(2L_3d(x_0,\Gamma_{\infty})^{-a}+8K+4L_4d(x_0,\Gamma_{\infty})^{-2a}\big)|u-v|^\beta\\
    &\leq 14K'd(x_0,\Gamma_{\infty})^{-2a}\|C_0(\hat{f}(\tilde{x})^{-1}\|\cdot|u-v|^\beta\\
    &\leq 28K'd(x_0,\Gamma_{\infty})^{-2a}\|C_0(\hat{f}(\tilde{x})^{-1}\|Q(\tilde{x})^{\frac{\beta}{2}}\cdot|u-v|^{\frac{\beta}{2}}\\
    &\leq 28K'\epsilon^{10}|u-v|^{\frac{\beta}{2}}\leq \frac{\epsilon^2}{3}|u-v|^{\frac{\beta}{2}},
\end{aligned}
\end{equation}
where $K':=\max\{K,L_3,L_4\}$. Let $v=0$, then $\|d_{u}H^+\|\leq \frac{\epsilon^2}{3}|u|^{\frac{\beta}{2}}$ for all $u\in B_{Q(\tilde{x})}(0)$, and so  $\|d_{u}H^+\|_{C^0}\leq \frac{\epsilon^2}{3}$ on $B_{Q(\tilde{x})}(0)$. Furthermore, $|H^+(u)|=|H^+(u)-H^+(0)|\leq \frac{\epsilon^2}{3}|u|$ by the mean value inequality, then $\|H^+\|_{C^0}\leq\frac{\epsilon^2}{3}$ on $B_{Q(\tilde{x})}(0)$. Therefore, $ \|H^+\|_{1+\frac{\beta}{2}}\leq \epsilon^2$ on $B_{Q(\tilde{x})}(0)$.

Similarly, we can estimate for $H^-$  on $B_{Q(\hat{f}(\tilde{x}))}(0)$.

(4) For any $u\in B_{Q(\tilde{x})}(0),v\in B_{Q(\hat{f}(\tilde{x}))}(0)$, by the proof of Lemma \ref{C}(2) and the foregoing conclusions, we get
\begin{equation*}
    \|d_uF_{\tilde{x}}\|\leq\|d_uF_{\tilde{x}}-d_0F_{\tilde{x}}\|+\|d_0F_{\tilde{x}}\|\leq \epsilon^2+\|D_u(\tilde{x})\|\leq \epsilon^2+\sqrt{2} d(x_0,\Gamma_{\infty})^{-a},
\end{equation*}
and
\begin{equation*}
\|d_vF_{\tilde{x}}^{-1}\|\leq\|d_vF_{\tilde{x}}^{-1}-d_0F_{\tilde{x}}^{-1}\|+\|d_0F_{\tilde{x}}^{-1}\|\leq \epsilon^2+\|D_s(\tilde{x})^{-1}\|\leq \epsilon^2+\sqrt{2} d(x_0,\Gamma_{\infty})^{-a}.
\end{equation*}
For small enough $\epsilon<1$, then $\|d_uF_{\tilde{x}}\|,\|d_vF_{\tilde{x}}^{-1}\|< 2\sqrt{2}d(x_0,\Gamma_{\infty})^{-a}$.
\end{pf}

Denote $\psi_{\tilde{x}}^{\eta}$ as the restriction of the Pesin chart $\psi_{\tilde{x}}$ to $B_\eta(0)$, where $0<\eta\leq Q(\tilde{x})$.
 \begin{mdef}
     ($I$-overlap) Two Pesin charts $\psi_{\tilde{x}_1}^{\eta_1},\ \psi_{\tilde{x}_2}^{\eta_2}$ are called \textit{$I$-overlap} if $\eta_1 = I^{\pm1} (\eta_2)$, and for some $D\in \mathcal{D}$ such that $p(\tilde{x}_1),p(\tilde{x}_2)\in D$, $d(p(\tilde{x}_1),p(\tilde{x}_2))+\|\Theta_D \circ C_0(\tilde{x}_1)- \Theta_D \circ C_0(\tilde{x}_2)\|< \eta_1^4\eta_2^4.$
 \end{mdef}
\begin{rmk}
Assume $\psi_{\tilde{x}_1}^{\eta_1},\ \psi_{\tilde{x}_2}^{\eta_2}$ are $I$-overlap, then
\begin{enumerate}[label=(\arabic*)]
    \item $\eta_2 = I^{\pm1} (\eta_1)$, i.e. the overlap condition is symmetric.
    \item $\psi_{\tilde{x}_1}^{\xi_1},\ \psi_{\tilde{x}_2}^{\xi_2}$ are $I$-overlap if $\eta_i\leq \xi_i \leq Q(\tilde{x_i})$ and $\xi_1=I^{\pm1}(\xi_2), i=1,2$.
\end{enumerate}
\end{rmk}

 Define the map $p_i: M^f\rightarrow M$ by $p_i(\tilde{x})=x_i,\ \forall \tilde{x}=(x_i)_{i\in \mathbb{Z}}$. For some $D\in\mathcal{D}$, let $C_i:=\Theta_D \circ C_0(\tilde{x}_i)$ for all $\tilde{x}_i\in $ 0-summ.
 \begin{prop}\label{overlap1}
For $\epsilon>0$ small enough, if Pesin charts $\psi_{\tilde{x}_1}^{\eta_1},\ \psi_{\tilde{x}_2}^{\eta_2}$ are $I$-overlap, then
\begin{enumerate}[label=(\arabic*)]
\item $\frac{d(p_i(\tilde{x}_1),\Gamma_\infty)}{d(p_i(\tilde{x}_2),\Gamma_\infty)}=\epsilon^{\pm\epsilon}$. In particular, $\frac{\rho(\tilde{x}_1)}{\rho(\tilde{x}_2)}=\epsilon^{\pm\epsilon}$.
\item $\|C_1^{-1}-C_2^{-1}\|< (\eta_1\eta_2)^3$, and $\frac{\|C_1^{-1}\|}{\|C_2^{-1}\|}=e^{\pm(\eta_1\eta_2)^3}$.
\item $\frac{\tilde{Q}(\tilde{x}_1)}{\tilde{Q}(\tilde{x}_2)}=e^{\pm(\frac{8a}{\beta}+2\gamma)\epsilon}$ and $\frac{Q(\tilde{x}_1)}{Q(\tilde{x}_2)}\in\big[I^{-\frac{1}{4}}(1)e^{-(\frac{8a}{\beta}+2\gamma)\epsilon},I^{\frac{1}{4}}(1)e^{(\frac{8a}{\beta}+2\gamma)\epsilon}\big].$
\item $\psi_{\tilde{x}_1}[B_{I^{-2}(\eta_1)}(0)] \subset \psi_{\tilde{x}_2}[B_{\eta_2}(0)]$ and $\psi_{\tilde{x}_2}[B_{I^{-2}(\eta_2)}(0)] \subset \psi_{\tilde{x}_1}[B_{\eta_1}(0)]$.
\item $\|\psi_{\tilde{x}_i}^{-1}\circ\psi_{\tilde{x}_j}-Id\|_{C^{1+\frac{\beta}{2}}}<\epsilon(\eta_1\eta_2)^3$ for $i,\ j =1,\ 2$ where the $C^{1+\frac{\beta}{2}}$ is calculated on $B_{I^{-1}(r)}(0)$.
\end{enumerate}
 \end{prop}
 \begin{pf}
Fix some $D\in\mathcal{D}$ such that $p(\tilde{x}_1),\ p(\tilde{x}_2)\in D$ and $d(p(\tilde{x}_1), p(\tilde{x}_2))+\|\ C_1-C_2\|<(\eta_1\eta_2)^4$. The proof of part (1) and (2) are the same as the proposition 4.1 in \cite{Lima}. And then part (3) holds by the definition of $Q(\cdot)$.

(4) Let $v\in B_{I^{-2}(\eta_1)}(0)$, by the $I$-overlap condition, we have
$$|\psi_{\tilde{x}_1}v-\psi_{\tilde{x}_2}v|=\big|\exp_{p(\tilde{x}_1)}\circ \nu_{\tilde{x}_1}(C_1v)-\exp_{p(\tilde{x}_2)}\circ \nu_{\tilde{x}_2}(C_2v)\big|<L_1(\eta_1\eta_2)^4,$$
where $L_1$ is the uniform Lipschitz constant for the map $(y,u)\mapsto(\exp_y \circ \nu_y)(u)$ on $D\times B_r(0)$. Then
\begin{equation*}
\begin{aligned}
\psi_{\tilde{x}_1}v&\in B_{L_1\eta_1^4\eta_2^4}(\psi_{\tilde{x}_2}v)=\psi_{\tilde{x}_2}\big(C(\tilde{x}_2)^{-1}\exp_{p(\tilde{x}_2)}^{-1}[B_{L_1\eta_1^4\eta_2^4}(\psi_{\tilde{x}_2}v)]\big)\\
&\subset \psi_{\tilde{x}_2}\big(C(\tilde{x}_2)^{-1}B_{2L_1\eta_1^4\eta_2^4}(\nu_{\tilde{x}_2}(C_1v))\big)\subset\psi_{\tilde{x}_{2}}\big[B_{2L_1\eta_1^4\eta_2^4\|C_2^{-1}\|}(C_2^{-1}(C_1v))\big].
\end{aligned}
\end{equation*}
Since $\eta_i\ll Q(\tilde{x}_i)< \epsilon^{\frac{20}{\beta}}$ for $i=1,\ 2$, $\exp^{-1}_{p(\tilde{x}_2)}$ is well-defined on $B_{L_1\eta_1^4\eta_2^4}(\psi_{\tilde{x}_2}v)$, so the '$=$' holds for small enough $\epsilon$. Thus, we only need to prove that $|w|<\eta_2$ if $w\in B_{2L_1\eta_1^4\eta_2^4\|C_2^{-1}\|}(C_2^{-1}(C_1v)).$
\begin{equation*}
\begin{aligned}
|w| &\leq|C_2^{-1}(C_1v)| + 2L_1\eta_1^4\eta_2^4\|C_2^{-1}\|\\
&\leq |(C_2^{-1}\circ C_1-Id)v| +|v| + 2L_1\eta_1^4\eta_2^4\|C_2^{-1}\|\\
&\leq |v| + \|C_2^{-1}\| \cdot \|C_1-C_2\|\cdot |v|  +2L_1\eta_1^4\eta_2^4\|C_2^{-1}\|\\
&\leq I^{-2}(\eta_1)+ \|C_0(\tilde{x}_2)^{-1}\|\eta_1^4\eta_2^4(I^{-2}(\eta_1)+ 2L_1)\\
&\leq I^{-2}(\eta_1)+\eta_1^4\text{ \Big(since } \eta_2^4\ll Q(\tilde{x}_2)^4 \leq (\epsilon^{\frac{20}{\beta}}\|C_0(\tilde{x}_2)\|^{-2\gamma})^4 \text{ and } \gamma>\frac{20}{\beta}.\Big)\\
&= (I^{-2}(\eta_1)-I^{-1}(\eta_1)) +I^{-1}(\eta_1)+ \eta_1^4\\
&\leq I^{-1}(\eta_1)-\frac{\Gamma}{2}\eta_1^{1+\frac{1}{\gamma}}+ \eta_1^4\\
&\leq I^{-1}(\eta_1)\leq \eta_1.
\end{aligned}
\end{equation*}

(5) Similar to the proof in part (1), for small enough $\epsilon<r-I^{-1}(r)$, we can show that $\psi_{\tilde{x}_j}[B_{I^{-1}(r)}(0)]\subset \psi_{\tilde{x}_i}[B_r(0)]$, $i,\ j=1,\ 2$. Hence, $\psi_{\tilde{x}_i}^{-1}\circ\psi_{\tilde{x}_j}$ is well defined on $B_{I^{-1}(r)}(0)$. Notice that
\begin{equation*}
\begin{aligned}
&\quad \psi_{\tilde{x}_i}^{-1}\circ \psi_{\tilde{x}_j}- Id\\
&= C_i^{-1}\circ \nu_{\tilde{x}_i}^{-1}\circ \exp_{p(\tilde{x}_i)}^{-1}\circ \exp_{p(\tilde{x}_j)} \circ \nu_{\tilde{x}_j}\circ C_j-Id\\
&= C_i^{-1}\circ \big(\nu_{\tilde{x}_i}^{-1}\circ \exp_{p(\tilde{x}_i)}^{-1}+\nu_{\tilde{x}_j}^{-1}\circ \exp_{p(\tilde{x}_j)}^{-1}-\nu_{\tilde{x}_j}^{-1}\circ \exp_{p(\tilde{x}_j)}^{-1}\big)\circ \exp_{p(\tilde{x}_j)} \circ \nu_{\tilde{x}_j}\circ C_j-C_i^{-1}\circ C_i\\
&= C_i^{-1}\circ \big(\nu_{\tilde{x}_i}^{-1}\circ \exp_{p(\tilde{x}_i)}^{-1}-\nu_{\tilde{x}_j}^{-1}\circ \exp_{p(\tilde{x}_j)}^{-1}\big)\circ \exp_{p(\tilde{x}_j)} \circ \nu_{\tilde{x}_j}\circ C_j+C_i^{-1}\circ \big(C_j-C_i\big).
\end{aligned}
\end{equation*}
Set $\Theta:= \nu_{\tilde{x}_i}^{-1}\circ \exp_{p(\tilde{x}_i)}^{-1}-\nu_{\tilde{x}_j}^{-1}\circ \exp_{p(\tilde{x}_j)}^{-1}$, then $\big\|\psi_{\tilde{x}_i}^{-1}\circ\psi_{\tilde{x}_j}-Id\big\|_{C^{1+\frac{\beta}{2}}}\leq \|C_i^{-1}\|\cdot \big\|\Theta\circ \psi_{\tilde{x}_j}\big\|_{C^{1+\frac{\beta}{2}}}+ \|C_{\tilde{x}_i^{-1}}\|\eta_1^4\eta_2^4$. Hence, we only need to calculate
\begin{equation*}
\begin{aligned}
\big\|\Theta \circ \psi_{\tilde{x}_j}\big\|_{C^{1+\frac{\beta}{2}}}=\big\|\Theta \circ \psi_{\tilde{x}_j}\big\|_{C^0}+ \big\|d(\Theta \circ \psi_{\tilde{x}_j})\big\|_{C^0}+ H\ddot{o}l_{\frac{\beta}{2}}\big(d(\Theta \circ \psi_{\tilde{x}_j})\big).
\end{aligned}
\end{equation*}  For $z_1\in B(p(\tilde{x}_i),r), z_2\in B(p(\tilde{x}_j),r)$, by inequality (\ref{exp5}),
\begin{equation}\label{exp3}
\big\|d_{z_1}\Theta-d_{z_2}\Theta\big\|\leq L_5d(p(\tilde{x}_i),p(\tilde{x}_j))d(z_1,z_2) \leq L_5\eta_1^4\eta_2^4d(z_1,z_2).
\end{equation}
(i) $\big\|\Theta \circ \psi_{\tilde{x}_j}\big\|_{C^0}=\big\|\big(\nu_{\tilde{x}_i}^{-1}\circ \exp_{p(\tilde{x}_i)}^{-1}-\nu_{\tilde{x}_j}^{-1}\circ \exp_{p(\tilde{x}_j)}^{-1}\big)\circ \psi_{\tilde{x}_j}\big\| \leq 2L_2 d(p(\tilde{x}_i),p(\tilde{x}_j)) \leq 2L_2\eta_1^4\eta_2^4,$ where $L_2$ is the uniform Lipschitz constant for the map $x \mapsto \nu_x^{-1} \circ \exp_x^{-1}$ on some $D\in\mathcal{D}$. \\
(ii) $\big\|d(\Theta \circ \psi_{\tilde{x}_j})\big\|_{C^0}\leq \big\|d\Theta\big\|_{C^0} \cdot \big\|d\psi_{\tilde{x}_j}\big\|_{C^0}\leq 2L_4d(p(\tilde{x}_i),p(\tilde{x}_j))\leq 2L_4 \eta_1^4\eta_2^4$ by inequality (\ref{exp2}).\\
(iii) For $v_1,v_2\in B_{I^{-1}(r)}(0)\subset \mathbb{R}^d$, let $z_1:=\psi_{\tilde{x}_j}(v_1),z_2:=\psi_{\tilde{x}_j}(v_2)$, we have
\begin{equation*}
    \begin{aligned}
\big\|d_{v_1}(\Theta\circ \psi_{\tilde{x}_j})-d_{v_2}(\Theta\circ \psi_{\tilde{x}_j})\big\|
&\leq \big\|d_{z_1}\Theta\cdot d_{v_1}\psi_{\tilde{x}_j}
       -d_{z_1}\Theta\cdot d_{v_2}\psi_{\tilde{x}_j}\big\|+\big\|d_{z_1}\Theta\cdot d_{v_2}\psi_{\tilde{x}_j}-d_{z_2}\Theta\cdot d_{v_2}\psi_{\tilde{x}_j}\big\|\\
&\leq \big\|d\Theta\big\|\cdot \big\|d_{v_1}\psi_{\tilde{x}_j}- d_{v_2}\psi_{\tilde{x}_j}\big\|
       + L_5\eta_1^4\eta_2^4d(\psi_{\tilde{x}_j}(v_1),\psi_{\tilde{x}_j}(v_2))\cdot \big\|d_{v_2}\psi_{\tilde{x}_j}\big\|\\
&\leq \big\|d\Theta\big\|\cdot H\ddot{o}l_{\frac{\beta}{2}}(d\psi_{\tilde{x}_j}) |v_1-v_2|^{\frac{\beta}{2}} +L_5\eta_1^4\eta_2^4\big\|d\psi_{\tilde{x}_j}\big\|^2|v_1-v_2|^{1-\frac{\beta}{2}}|v_1-v_2|^{\frac{\beta}{2}}\\
       &\leq \big(\big\|d\Theta\big\|\cdot H\ddot{o}l_{\frac{\beta}{2}}(d\psi_{\tilde{x}_j}) +L_5\eta_1^4\eta_2^4\big\|d\psi_{\tilde{x}_j}\big\|^2(2I^{-1}(r))^{1-\frac{\beta}{2}}\big)|v_1-v_2|^{\frac{\beta}{2}}.
    \end{aligned}
\end{equation*}
And $\big\|d_{v_1}(\nu_{z_1}^{-1}\circ \psi_{\tilde{x}_j})-d_{v_2}(\nu_{z_2}^{-1}\circ \psi_{\tilde{x}_j})\big\|\leq L_3|v_1-v_2|$ by inequality (\ref{exp2}),
then $H\ddot{o}l_{\frac{\beta}{2}}(d\psi_{\tilde{x}_j})\leq L_3(2I^{-1}(r))^{1-\frac{\beta}{2}}$. Thus,
\begin{equation*}
    \begin{aligned}
        H\ddot{o}l_{\frac{\beta}{2}}(d(\Theta \circ \psi_{\tilde{x}_j}))
        &\leq \big\|d\Theta\big\|\cdot H\ddot{o}l_{\frac{\beta}{2}}(d\psi_{\tilde{x}_j}) +L_5\eta_1^4\eta_2^4\big\|d\psi_{\tilde{x}_j}\big\|^2(2I^{-1}(r))^{1-\frac{\beta}{2}}\\
        &\leq L_4 \eta_1^4\eta_2^4\cdot  L_3(2I^{-1}(r))^{1-\frac{\beta}{2}}+4L_5\eta_1^4\eta_2^4 (2I^{-1}(r))^{1-\frac{\beta}{2}}.
    \end{aligned}
\end{equation*}
Hence, $\big\|\psi_{\tilde{x}_i}^{-1}\circ\psi_{\tilde{x}_j}-Id\big\|_{C^{1+\frac{\beta}{2}}}\leq \epsilon(\eta_1\eta_2)^3$ on $B_{I^{-1}(r)}(0)$.
\end{pf}

For $\tilde{x},\ \tilde{y}\in 0$-summ, if $\psi_{\hat{f}(\tilde{x})}^{\eta_1}$ and $\psi_{\tilde{y}}^{\eta_2}$ are $I$-overlap, define $F_{\tilde{x},\tilde{y}}:=\psi_{\tilde{y}}^{-1}\circ f \circ \psi_{\tilde{x}}$; and if $\psi_{\tilde{x}}^{\eta_1}$ and $\psi_{\hat{f}^{-1}(\tilde{y})}^{\eta_2}$ are $I$-overlap, define $F_{\tilde{x}, \tilde{y}}^{-1}:=\psi_{\tilde{x}}^{-1}\circ f^{-1}_{\tilde{x}} \circ \psi_{\tilde{y}}$. Notice that the inverse of the map $F_{\tilde{x},\tilde{y}}$ is $\psi_{\tilde{x}}^{-1}\circ f^{-1}_{\hat{f}^{-1}(\tilde{y})} \circ \psi_{\tilde{y}}$ if $F_{\tilde{x},\tilde{y}}$ is invertible. The following lemma shows that the maps $F_{\tilde{x},\tilde{y}},F_{\tilde{x},\tilde{y}}^{-1}$ are the inverse of each other at some small areas.
\begin{lem}
   For $\tilde{x},\ \tilde{y}\in 0$-summ, if $\psi_{\tilde{x}}^{\eta_1}$ and $\psi_{\hat{f}^{-1}(\tilde{y})}^{\eta_2}$ are $I$-overlap, then $f^{-1}_{\tilde{x}}$ and $ f^{-1}_{\hat{f}^{-1}(\tilde{y})}$ coincide on $\psi_{\tilde{y}}(B_{Q(\tilde{y})}(0))$.
\end{lem}
\begin{pf}
    Let $\tilde{x}=(x_n)_{n\in\mathbb{Z}}$ and $\tilde{y}=(y_n)_{n\in\mathbb{Z}}$, and $\mathfrak{E}_{\tilde{x}}=\mathfrak{E}_{x_0}=B(x_1,2\tau(x_0))$. Since $$\psi_{\tilde{y}}(B_{Q(\tilde{y})}(0))\subset B(y_0, 2Q(\tilde{y}))\subset B(x_1, 2Q(\tilde{y})+d(x_1,y_0)),$$
and
\begin{equation}\label{tao}
   \begin{aligned}
       2Q(\tilde{y})+d(x_1,y_0))&\leq 2\epsilon^{\frac{20}{\beta}}\rho(\tilde{y})^a+ d(x_0,\Gamma_\infty)^{-a}d(x_0,y_{-1})\text{ (by assumption \ref{A2}})\\
       &\leq 2\epsilon^{\frac{20}{\beta}}\min\{d(y_{-1},\Gamma_\infty)^a, d(y_0,\Gamma_\infty)^a\}+\rho(\tilde{x})^{-a}Q(\tilde{x}) \text{ (since }d(x_0,y_{-1})\leq \eta_1^4\eta_2^4\leq Q(\tilde{x}))\\
       &\leq 2\epsilon^{\frac{20}{\beta}}e^{a\epsilon}\min\{d(x_0,\Gamma_\infty)^a, d(x_{1},\Gamma_\infty)^a\}+\epsilon\rho(\tilde{x})^a\text{ (by Proposition \ref{overlap1}(1)})\\
       &\leq \epsilon\tau(x_0)+ \epsilon\tau(x_0)=2\epsilon\tau(x_0)\leq 2\tau(x_0),
    \end{aligned}
\end{equation}
then $\psi_{\tilde{y}}(B_{Q(\tilde{y})}(0))\subset \mathfrak{E}_{\tilde{x}}$. Next, we prove $f^{-1}_{\hat{f}^{-1}(\tilde{y})}\psi_{\tilde{y}}(B_{Q(\tilde{y})}(0))\subset f^{-1}_{\tilde{x}}(\mathfrak{E}_{\tilde{x}})$. By assumption \ref{A2}, we have
$$f^{-1}_{\hat{f}^{-1}(\tilde{y})}\psi_{\tilde{y}}(B_{Q(\tilde{y})}(0))\subset f^{-1}_{\hat{f}^{-1}(\tilde{y})}B(y_0, 2Q(\tilde{y}))\subset B(y_{-1}, 2d(y_{-1},\Gamma_\infty)^{-a}Q(\tilde{y})),$$
and
$$f^{-1}_{\tilde{x}}(\mathfrak{E}_{\tilde{x}})=f^{-1}_{\tilde{x}}B(x_1,2\tau(x_0))\supset B(x_0, 2d(x_0,\Gamma_\infty)^a \tau(x_0)).$$
Since $B(y_{-1}, 2d(y_{-1},\Gamma_\infty)^{-a}Q(\tilde{y}))\subset B(x_0, 2d(y_{-1},\Gamma_\infty)^{-a}Q(\tilde{y})+d(x_0,y_{-1}))$ and
\begin{align*}
        2d(y_{-1},\Gamma_\infty)^{-a}Q(\tilde{y})+d(x_0,y_{-1})&\leq 2\rho(\tilde{y})^{-a}Q(\tilde{y}) + Q(\tilde{x})
        \leq 2\epsilon^{\frac{20}{\beta}}\rho(\tilde{y})^{2a} + \rho(\tilde{x})^{2a}\\
        &\leq 2\epsilon^{\frac{20}{\beta}}\big(\min\{d(y_{-1},\Gamma_\infty),d(y_0,\Gamma_\infty)\}\big)^{2a}+d(x_0,\Gamma_\infty)^a\rho(\tilde{x})^a\\
        &\leq 2\epsilon^{\frac{20}{\beta}}e^{2a\epsilon}\big(\min\{d(x_0,\Gamma_\infty),d(x_1,\Gamma_\infty)\}\big)^{2a}+d(x_0,\Gamma_\infty)^a\tau(x_0)\\
        &\leq 2d(x_0,\Gamma_\infty)^a\tau(x_0).
\end{align*}
\end{pf}

 The following theorem shows that $F_{\tilde{x}, \tilde{y}}$ is a small perturbation of the map $F_{\tilde{x}}$ on $B_{Q(\tilde{x})}(0)$, and $F_{\tilde{x}, \tilde{y}}^{-1}$ is a small perturbation of the map $F_{\tilde{x}}^{-1}$ on $B_{Q(\tilde{y})}(0)$.
\begin{thm}\label{F1}
    For small enough $\epsilon$, assume $\tilde{x},\ \tilde{y}\in $ 0-summ and $d_s(\tilde{x})=d_s(\tilde{y})$. If $\psi_{\hat{f}(\tilde{x})}^{\eta_1}$ $I$- overlaps $\psi_{\tilde{y}}^{\eta_2}$, then $F_{\tilde{x},\tilde{y}}$ is well-defined in $B_{Q(\tilde{x})}(0)$ and it can be written as the form $F_{\tilde{x},\tilde{y}}=D_0(\tilde{x})+\bar{H}^+$, where $\bar{H}^+$ satisfies: $\forall t\in[\eta_1,Q(\tilde{x})]$,
\begin{enumerate}[label=(\arabic*)]
\item $|\bar{H}^+(0)|\leq \epsilon\eta_1^3$,
\item $\|d\bar{H}^+\|_{C^0}\leq \sqrt{\epsilon}t^{\frac{\beta}{2}}$ on $B_{t}(0)$,
\item $H\ddot{o}l_{\frac{\beta}{2}}(d\bar{H}^+)\leq \sqrt{\epsilon}$.
\end{enumerate}
 Similarly, if $\psi_{\tilde{x}}^{\eta_1}$ $I$-overlaps $\psi_{\hat{f}^{-1}(\tilde{y})}^{\eta_2}$, then $F_{\tilde{x},\tilde{y}}^{-1}$ is well-defined in $B_{Q(\tilde{y})}(0)$ and it can be written as the form $F_{\tilde{x},\tilde{y}}^{-1}=D_0(\hat{f}^{-1}(\tilde{y}))^{-1}+\bar{H}^-$, where $\bar{H}^-$ satisfies: $\forall t\in[\eta_1,Q(\tilde{y})]$,
\begin{enumerate}[label=(\arabic*')]
\item $|\bar{H}^-(0)|\leq \epsilon\eta_1^3$,
\item $\|d\bar{H}^-\|_{C^0} \leq \sqrt{\epsilon}t^{\frac{\beta}{2}}$ on $B_{t}(0)$,\label{h-}
\item $H\ddot{o}l_{\frac{\beta}{2}}(d\bar{H}^-)\leq \sqrt{\epsilon}$.
\end{enumerate}
\end{thm}
\begin{pf}
    Notice that $F_{\tilde{x},\tilde{y}}=\psi_{\tilde{y}}^{-1}\circ \psi_{\hat{f}(\tilde{x})}\circ \psi_{\hat{f}(\tilde{x})}^{-1}\circ f\circ \psi_{\tilde{x}}=\psi_{\tilde{y}}^{-1}\circ \psi_{\hat{f}(\tilde{x})}\circ F_{\tilde{x}}$. Since $2\sqrt{2}d(x_0,\Gamma_{\infty})^{-a}Q(\tilde{x})\leq 2\sqrt{2}\epsilon^{\frac{20}{\beta}}\rho(\tilde{x})^a\leq I^{-1}(r)$ for small enough $\epsilon$, then
    $$F_{\tilde{x}}(B_{Q(\tilde{x})}(0))\subset B_{2\sqrt{2} d(x_0,\Gamma_{\infty})^{-a}Q(\tilde{x})}(0)\subset B_{I^{-1}(r)}(0) $$
    by Theorem \ref{F_{x}}(4). Thus, $F_{\tilde{x},\tilde{y}}$ is well-defined in $B_{Q(\tilde{x})}(0)$ by Proposition \ref{overlap1}(5). Now we estimate $\bar{H}^+$:

    (1) $|\bar{H}^+(0)|=|F_{\tilde{x},\tilde{y}}(0)|=\big|\psi_{\tilde{y}}^{-1}\circ \psi_{\hat{f}(\tilde{x})}(0)\big| \leq \epsilon\eta_1^3$ by Proposition \ref{overlap1}(5).

    (2) For any $v\in B_t(0)$, we have
    \begin{equation*}
    \begin{aligned}
        \|d_v\bar{H}^+\|&=\|d_vF_{\tilde{x},\tilde{y}}-D_0(\tilde{x})\|\\
        &\leq \|d_{F_{\tilde{x}}(v)}\big(\psi_{\tilde{y}}^{-1}\circ \psi_{\hat{f}(\tilde{x})}\big)-Id\|\cdot\|d_vF_{\tilde{x}}\|+\|d_vF_{\tilde{x}}-d_0F_{\tilde{x}}\|\text{ (since }d_0F_{\tilde{x}}=D_0(\tilde{x}))\\
        &\leq 2\sqrt{2}\epsilon\eta_1^3\eta_2^3d(x_0,\Gamma_{\infty})^{-a} + \frac{\epsilon^2}{3}|v|^{\frac{\beta}{2}}\\
        &\leq \epsilon\eta_1^2\eta_2^3(\epsilon^2+\rho(\tilde{x})^{-a}) + \frac{\epsilon^2}{3}t^{\frac{\beta}{2}}\leq \sqrt{\epsilon}t^{\frac{\beta}{2}},
    \end{aligned}
\end{equation*}
for small enough $\epsilon$. The second inequality holds by Proposition \ref{overlap1}(5), Theorem \ref{F_{x}}(4) and inequality (\ref{F0}).

(3) For any $v_1,\ v_2 \in B_{Q(\tilde{x})}(0)$, we have
    \begin{equation*}
    \begin{aligned}
        &\quad \|d_{v_1}\bar{H}^+-d_{v_2}\bar{H}^+\|=\|d_{v_1}F_{\tilde{x},\tilde{y}}-d_{v_2}F_{\tilde{x},\tilde{y}}\|\\
        &\leq \big\|d_{F_{\tilde{x}}(v_1)}\big(\psi_{\tilde{y}}^{-1}\circ \psi_{\hat{f}(\tilde{x})}\big)-d_{F_{\tilde{x}}(v_2)}\big(\psi_{\tilde{y}}^{-1}\circ \psi_{\hat{f}(\tilde{x})}\big)\big\|\cdot\|d_{v_1}F_{\tilde{x}}\|+\|d_{F_{\tilde{x}}(v_2)}\big(\psi_{\tilde{y}}^{-1}\circ \psi_{\hat{f}(\tilde{x})}\big)\|\cdot\|d_{v_1}F_{\tilde{x}}-d_{v_2}F_{\tilde{x}}\|\\
         &\leq 2\sqrt{2}\epsilon\eta_1^3\eta_2^3d(x_0,\Gamma_{\infty}^{-2a})\|F_{\tilde{x}}(v_1)-F_{\tilde{x}}(v_2)\|^{\frac{\beta}{2}} + (1+\epsilon\eta_1^3\eta_2^3)\cdot\Big(\frac{\epsilon^2}{3}|v_1-v_2|^{\frac{\beta}{2}}\Big)\text{ (by inequality (\ref{F0}))}\\
        &\leq \Big[\epsilon\eta_1^3\eta_2^3(2\sqrt{2}d(x_0,\Gamma_{\infty})^{-2a})^{1+\frac{\beta}{2}} + \frac{\epsilon^3}{3}(1+\epsilon\eta_1^3\eta_2^3)\Big]|v_1-v_2|^{\frac{\beta}{2}} \text{ (by Theorem \ref{F_{x}}(4)})\\
        &\leq \sqrt{\epsilon}|v_1-v_2|^{\frac{\beta}{2}},
    \end{aligned}
\end{equation*}
for small enough $\epsilon$. Hence,  $H\ddot{o}l_{\frac{\beta}{2}}(d\bar{H}^+)\leq \sqrt{\epsilon}$.

By similar calculations, the above conclusions hold for $F^{-1}_{\tilde{x},\tilde{y}}$.
\end{pf}

\subsection{\texorpdfstring{Coarse Graining}{Coarse Graining}}
\begin{prop}\label{graph A}
    For small enough $\epsilon>0$, there exists a countable collection $\mathcal{A}$ of Pesin charts such that
\begin{enumerate}[label=(\arabic*)]
\item Discreteness: $ \{\psi^{\eta}_{\tilde{x}}\in\mathcal{A}: \eta> t \}$ is finite for every $t>0$.
\item  Sufficiency: $\forall \tilde{x}\in RST$ and for every sequence $\{\eta_n\in \mathcal{I}\}_{n\in\mathbb{Z}}$ with $0<\eta_n<I^{\frac{-1}{4}}(Q(\hat{f}^n(\tilde{x})))$. If $\eta_n=I^{\pm1}(\eta_{n+1})\forall n\in\mathbb{N}$, then there exists a sequence $\{\psi_{\tilde{x}_n}^{\eta_n}\in \mathcal{A}\}_{n\in\mathbb{Z}}$ such that $\forall n$:
\begin{enumerate}[label=(\alph*)]
\item $\psi_{\tilde{x}_n}^{\eta_n}$ $I$-overlaps $\psi_{\hat{f}^n(\tilde{x})}^{\eta_n}$, $Q(\hat{f}^n(\tilde{x}))=I^{\frac{\pm1}{4}}(Q(\tilde{x}_n))$ and $d_s(\tilde{x}_n)=d_s(\hat{f}^n(\tilde{x}))$;\label{a1}
\item $\psi_{\hat{f}(\tilde{x}_n)}^{\eta_{n+1}}$ $I$-overlaps $\psi_{\tilde{x}_{n+1}}^{\eta_{n+1}}$;\label{b1}
\item $\psi_{\hat{f}^{-1}(\tilde{x}_n)}^{\eta_{n-1}}$ $I$-overlaps $\psi_{\tilde{x}_{n-1}}^{\eta_{n-1}}$;
\item If $\eta_n'\in\mathcal{I}$ satisfies $\eta_n\leq \eta_n'\leq \min\{Q(\tilde{x}_n), I(\eta_n)\}$, then $\psi_{\tilde{x}_n}^{\eta_n'}\in\mathcal{A}$.
\end{enumerate}
\end{enumerate}
\end{prop}
\begin{pf}
    Denote $X:= M^3\times GL(d,\mathbb{R})^3\times (0,1]^3$ with the product topology. For $\tilde{x}\in RST$, there exist three elements $D_{-1},D_0,D_1\in\mathcal{D}$ with $p_{-1}(\tilde{x}),p(\tilde{x}),p_1(\tilde{x})\in D_{-1}\times D_0\times D_1$, and let
\begin{align*}
    Y:=\Big\{ (\uline{\tilde{x}}, \uline{C_0},\uline{Q(\tilde{x})} ):&\
   \uline{\tilde{x}}= \big(p_{-1}(\tilde{x}),p(\tilde{x}),p_1(\tilde{x}) \big), \uline{Q(\tilde{x})}=\Big(Q\big(\hat{f}^{-1}(\tilde{x})\big),Q(\tilde{x}),Q\big(\hat{f}(\tilde{x})\big)\Big)\\
   &\text{ and }\uline{C_0}= \Big(\Theta_{D_{-1}}\circ C_0\big(\hat{f}^{-1}(\tilde{x})\big),\Theta_{D_0}\circ C_0(\tilde{x}),\Theta_{D_1}\circ C_0\big(\hat{f}(\tilde{x})\big) \Big)\Big\}.
\end{align*}
  For $\uline{k}=(k^{(-1)},k^{(0)},k^{(1)})\in\mathbb{N}^3$ and $l\in\{0,1,\dots, d\}$, we denote
  $$Y_{\uline{k},l}:= \big\{ (\uline{\tilde{x}}, \uline{C_0},\uline{Q(\tilde{x})} ):e^{-k^{(i)}-1}\leq \uline{Q(\tilde{x})}\leq e^{-k^{(i)}}, d_s(\tilde{x})=l,i=-1,0,1\big\}.$$
  Now we prove that $Y_{\uline{k},l}$ is precompact in $X$:

  (i) Since $\{ (p_{-1}(\tilde{x}),p(\tilde{x}),p_1(\tilde{x}) ):\tilde{x}\in RST\}$ is a subset of $M^3$ and $M$ is compact, then $\{ (p_{-1}(\tilde{x}),p(\tilde{x}),p_1(\tilde{x}) ):\tilde{x}\in RST\}$ is precompact;

  (ii) For $i=-1, 0,1$, since $\Theta_{D_i}$ is a linear isometry, $\|C_0(\hat{f}^i(\tilde{x}))\|<1$ and $$\|C_0(\hat{f}^i(\tilde{x}))^{-1}\|\leq (\epsilon^{\frac{20}{\beta}})^{\frac{1}{2\gamma}}Q(\hat{f}^i(\tilde{x}))^{-\frac{1}{2\gamma}}\leq (\epsilon^{\frac{20}{\beta}}e^{k^{(i)}+1})^{\frac{1}{2\gamma}}$$ by the definition of $Q(\hat{f}^i(\tilde{x}))$, thus $ \uline{C_0}$ belongs to a compact set.

  (iii) Since $e^{-k^{(i)}-1}\leq Q(\hat{f}^{i}(\tilde{x}))\leq 1$ and $[e^{-k^{(i)}-1},1]$ is compact, then $\{\uline{Q(\tilde{x})}:\tilde{x} \in RST\}$ is precompact.

 With the product topology, we have $Y_{\uline{k},l}\subset X$ is a precompact set. Hence, we can find a finite set $Y_{\uline{k},l,m}$ that satisfies $\forall (\uline{\tilde{x}}, \uline{C_0},\uline{Q(\tilde{x})})\in Y_{\uline{k},l}$ there exists an element $(\uline{\tilde{y}}, \uline{C_0'},\uline{Q(\tilde{y})})\in Y_{\uline{k},l,m}$ s.t. $\forall i=-1,0,1$

  $\bullet \quad d(p_i(\tilde{x}),p_i(\tilde{y}))< \frac{1}{2}\omega(\mathcal{D})$;

   $\bullet \quad \exists D \in\mathcal{D}$ s.t. $p_i(\tilde{x})$ and $p_i(\tilde{y})$ belong to the same $D$, and $d(p_i(\tilde{x}),p_i(\tilde{y}))+\|\Theta_{D}\circ C_0(\hat{f}^i(\tilde{x}))-\Theta_{D}\circ C_0(\hat{f}^i(\tilde{y}))\|\leq e^{-8(m+2)}$;

  $\bullet \quad Q(\hat{f}^i(\tilde{x}))=I^{\pm\frac{1}{4}}Q(\hat{f}^i(\tilde{y}))$.

  Let
  \begin{equation}\label{A}
      \begin{aligned}
          \mathcal{A}:=\big\{\psi_{\tilde{x}}^\eta: &(\uline{\tilde{x}}, \uline{C_0},\uline{Q(\tilde{x})} )\in Y_{\uline{k},l,m} \text{ for some }\uline{k},m\in\mathbb{N},l\in\{1,2,\cdots,d\},\\
          &\text{ and }\eta\in\mathcal{I}, 0<\eta\leq Q(\tilde{x}),e^{-(m+2)}\leq \eta <e^{-(m-2)}\big\}.
      \end{aligned}
  \end{equation}
  Hence, $\mathcal{A}$ is a countable collection of Pesin charts.

(1) (Discreteness) For any $\psi_{\tilde{x}}^\eta\in\mathcal{A}$, we have $\eta <e^{-(m-2)}$ and $Q(\hat{f}^{i}(\tilde{x}))\leq e^{-k^{(i)}}$ (since $(\uline{\tilde{x}}, \uline{C_0},\uline{Q(\tilde{x})})\in Y_{\uline{k},l,m}\subset Y_{\uline{k},l})$. Then $m\leq |\log\eta|+2$ and $k^{(i)}\leq |\log Q(\hat{f}^i(\tilde{x}))|$. Note that $\eta\leq Q(\tilde{x})\leq \rho(\tilde{x})^{\frac{8a}{\beta}}\leq 1$, then $k^{(0)}\leq |\log\eta|$ and $\frac{\eta}{2^\gamma} d(x_0,\Gamma_\infty)^{4a\gamma}\leq \frac{Q(\hat{f}(\tilde{x}))}{Q(\tilde{x})}\leq \frac{2^\gamma}{\eta} d(x_0,\Gamma_\infty)^{-4a\gamma}$ by Proposition \ref{C}(3). Furthermore,
$$k^{(\pm1)}\leq |\log Q(\hat{f}^{\pm}(\tilde{x}))|\leq \Big|\log \Big(\frac{1}{2^\gamma}\eta^{1+\frac{\beta\gamma}{2}}Q(\tilde{x})\Big)\Big|\leq \Big|\log \Big(\frac{1}{2^\gamma}\eta^{2+\frac{\beta\gamma}{2}}\Big)\Big|\leq \Big|\log \Big(\frac{1}{2^\gamma}t^{2+\frac{\beta\gamma}{2}}\Big)\Big|,$$
by $\frac{\eta}{2^\gamma} d(x_0,\Gamma_\infty)^{4a\gamma}\geq \frac{\eta}{2^\gamma} \rho(\tilde{x})^{4a\gamma}\geq \frac{1}{2^\gamma}\eta^{1+\frac{\beta\gamma}{2}}$.
Hence,
$$|\{\psi^{\eta}_{\tilde{x}}\in\mathcal{A}: \eta> t \}|\leq \sum_{\substack{k^{(0)},m<|\log t|+2,k^{(\pm1)}\leq \big|\log (\frac{1}{2^\gamma}t^{2+\frac{\beta\gamma}{2}})\big|\\
l\in\{0,1,\dots, d\}}}|Y_{\uline{k},l,m}|\times |\{\eta\in\mathcal{I}: \eta>t\}|<\infty.$$

(2) (Sufficiency) Notice that $I^{\frac{-1}{4}}(Q(\hat{f}^n(\tilde{x})))\leq Q(\hat{f}^n(\tilde{x}))$ by Proposition \ref{I}(6). For any $\tilde{x}\in RST$ and $\eta_n\in\mathcal{I}$ satisfy $0<\eta_n<I^{\frac{-1}{4}}(Q(\hat{f}^n(\tilde{x})))$ and $\eta_n=I^{\pm1}(\eta_{n+1})\ \forall n\in\mathbb{N}$. Suppose $\uline{k_n},m_n$ s.t. $e^{-(m_n+1)}\leq \eta_n\leq e^{-m_n+1}$ and $e^{-k_n^{(i)}-1}\leq Q(\hat{f}^{(i)}(\hat{f}^n(\tilde{x})))\leq e^{k_n^{(i)}},i=-1,0,1$, and $l_n=d_s(\hat{f}^n(\tilde{x}))\ \forall n\in\mathbb{N}$. Then $(\uline{\hat{f}^n(\tilde{x})}, \uline{C_0},\uline{Q(\hat{f}^n(\tilde{x}))}) \in Y_{\uline{k}_n,l_n}$, and so there exist element $(\uline{\tilde{x}_n}, \uline{C_0'},\uline{Q(\tilde{x}_n)}) \in Y_{\uline{k}_n,l_n,m_n}$ such that $\forall n\in\mathbb{N}$:
\begin{enumerate}[label=(B\arabic*)]
\item $d(p_i(\hat{f}^n(\tilde{x})),p_i(\tilde{x}_n))< \frac{1}{2}\omega(\mathcal{D})$;\label{B1}
\item For $i=-1, 0,1$, $p_i(\hat{f}^n(\tilde{x}))$ and $p_i(\tilde{x}_n)$ belong to the same $D\in\mathcal{D}$, and $d(p_i(\hat{f}^n(\tilde{x})),p_i(\tilde{x}_n))+\|\Theta_{D}\circ C_0(\hat{f}^i(\hat{f}^n(\tilde{x})))-\Theta_{D}\circ C_0(\hat{f}^i(\tilde{x}_n))\|\leq e^{-8(m_n+2)}$;\label{B2}
\item $Q(\hat{f}^i(\hat{f}^n(\tilde{x})))=I^{\pm\frac{1}{4}}Q(\hat{f}^i(\tilde{x}_n))$;\label{B3}
\item $d_s(\tilde{x}_n)=d_s(\hat{f}^n(\tilde{x}))=l_n$.
\end{enumerate}
And the Pesin chart $\psi_{\tilde{x}_n}^{\eta_n}$ belongs to $\mathcal{A},\ \forall n\in\mathbb{N}$.

  (a) We just need to prove $\psi_{\tilde{x}_n}^{\eta_n}$ $I$-overlaps $\psi_{\hat{f}^n(\tilde{x})}^{\eta_n}$: According to \ref{B1} with $i=0$, there exists $D_0\in\mathcal{D}$ s.t. $p(\hat{f}^n(\tilde{x}))$ and $p(\tilde{x}_n)$ belong to $D_0$. And
  $$d\big(p(\hat{f}^n(\tilde{x})),p(\tilde{x}_n)\big)+\|\Theta_{D_0}\circ C_0(\hat{f}^n(\tilde{x}))-\Theta_{D_i}\circ C_0(\tilde{x}_n)\|\leq e^{-8(m_n+2)}\leq \eta_n^4\eta_n^4,$$
the last inequality holds since $\eta_n\geq e^{-m_n-1 }$.

(b) We need to prove $\psi_{\hat{f}(\tilde{x}_n)}^{\eta_{n+1}}$ $I$-overlaps $\psi_{\tilde{x}_{n+1}}^{\eta_{n+1}}$: In the case of $n$, by \ref{B1} with $i=1$, $d(p_1(\hat{f}^n(\tilde{x})),p_1(\tilde{x}_n))< \frac{1}{2}\omega(\mathcal{D})$; and in the case of $n+1$, by \ref{B1} with $i=0$, $ d(p(\hat{f}^{n+1}(\tilde{x})),p(\tilde{x}_{n+1})< \frac{1}{2}\omega(\mathcal{D})$. Thus we can find a set $D\in\mathcal{D}$ such that $\hat{f}(\tilde{x}_n),\ \hat{f}^{n+1}(\tilde{x})$ and $\tilde{x}_{n+1}$ belong to the same set $D$. Hence,
\begin{equation*}
    \begin{aligned}
        &\quad d(p_1(\tilde{x}_n),p(\tilde{x}_{n+1}))+\|\Theta_{D}\circ C_0(\hat{f}(\tilde{x}_n))-\Theta_{D}\circ C_0(\tilde{x}_{n+1})\|\\
        &\leq d(p_1(\tilde{x}_{n}),p_1(\hat{f}^n(\tilde{x})))+\|\Theta_{D}\circ C_0(\hat{f}(\tilde{x}_n))-\Theta_{D}\circ C_0(\hat{f}(\hat{f}^n(\tilde{x})))\|\\
        &+ d(p_1(\hat{f}^n(\tilde{x})),p(\tilde{x}_{n+1}))+\|\Theta_{D}\circ C_0(\hat{f}(\hat{f}^n(\tilde{x})))-\Theta_{D}\circ C_0(\tilde{x}_{n+1})\|\\
        &\leq e^{-8(m_n+2)}+e^{-8(m_{n+1}+2)}\text{ (by \ref{B2}})\\
        &\leq e^{-8}(\eta_n^8+\eta_{n+1}^8)\\
        &\leq e^{-8}\Big(e^{\Gamma\eta_{n+1}^{\frac{1}{\gamma}}}+1\Big)\eta_{n+1}^8 \text{ (since } \eta_n\leq I(\eta_{n+1}))\\
        &\leq \eta_{n+1}^4\eta_{n+1}^4 \text{ (since } \eta_{n+1}\ll 1).
    \end{aligned}
\end{equation*}
Similarly, we can prove (c).

(d) By $\psi_{\tilde{x}_n}^{\eta_n}\in \mathcal{A}$ and $e^{-m_n-2}\leq \eta_n\leq \eta_n'\leq I(\eta_n)= e^{\Gamma\eta_n^{\frac{1}{\gamma}}}\eta_n\leq e^{-m_n+2}$, part (d) holds.
\end{pf}
\begin{mdef}
    An \textit{$I$-double chart} is a pair of Pesin charts $\psi_{\tilde{x}}^{p^s,p^u}:=(\psi_{\tilde{x}}^{p^s},\psi_{\tilde{x}}^{p^u})$ where $p^s,p^u\in\mathcal{I}$ and $0<p^s,p^u\leq Q(\tilde{x})$. For simplicity, we just say $\psi_{\tilde{x}}^{p^s,p^u}$ \textit{double chart}.
\end{mdef}

Write $\psi_{\tilde{x}}^{p^s,p^u}\rightarrow\psi_{\tilde{y}}^{q^s,q^u}$ if
\begin{enumerate}[label=(\arabic*)]
\item $\psi_{\tilde{x}}^{p^s\land p^u}$ $I$-overlaps $\psi_{\hat{f}^{-1}(\tilde{y})}^{p^s\land p^u}$ and $\psi_{\hat{f}(\tilde{x})}^{q^s\land q^u}$ $I$-overlaps $\psi_{\tilde{y}}^{q^s\land q^u}$;
\item $p^s=\min\{I(q^s),Q(\tilde{x})\}$ and $q^u=\min\{I(p^u),Q(\tilde{y})\}$.
\end{enumerate}
\begin{mdef}
    An \textit{$I$-chain} is a sequence $\{\psi_{\tilde{x}_i}^{p_i^s,p_i^u}\}_{i\in\mathbb{Z}}$ of double charts such that $\psi_{\tilde{x}_i}^{p_i^s,p_i^u}\rightarrow\psi_{\tilde{x}_{i+1}}^{p_{i+1}^s,p_{i+1}^u}\ \forall i\in\mathbb{Z}$. And the sequence $\{\psi_{\tilde{x}_i}^{p_i^s,p_i^u}\}_{i\geq 0}$ (resp. $\{\psi_{\tilde{x}_i}^{p_i^s,p_i^u}\}_{i\leq 0}$) is called \textit{positive (resp. negative) $I$-chain} if $\psi_{\tilde{x}_i}^{p_i^s,p_i^u}\rightarrow\psi_{\tilde{x}_{i+1}}^{p_{i+1}^s,p_{i+1}^u}\ \forall i\geq 0$ (resp. $i\leq0$).
\end{mdef}

Let $\mathcal{G}$ be a countable directed graph with vertices $\mathcal{V}$ and edges $\mathcal{E}$ where

$$\mathcal{V}:=\big\{\psi_{\tilde{x}}^{p^s,p^u}:\psi_{\tilde{x}}^{p^s\land p^u}\in\mathcal{A}, p^s,p^u\in \mathcal{I}, 0<p^s,p^u\leq Q(\tilde{x})\big\}\text{ ($\mathcal{A}$ is constructed by (\ref{A}))}$$
and
$$\mathcal{E}:=\big\{(\psi_{\tilde{x}}^{p^s,p^u},\psi_{\tilde{y}}^{q^s,q^u})\in\mathcal{V}\times\mathcal{V}: \psi_{\tilde{x}}^{p^s,p^u}\rightarrow\psi_{\tilde{y}}^{q^s,q^u}\big\}.$$
By the above definitions, it is easy to prove that  $q^u\land q^s=I^{\pm1}(p^u\land p^s)$ if $\psi_{\tilde{x}}^{p^s,p^u}\rightarrow\psi_{\tilde{y}}^{q^s,q^u}$. Then for every $\psi_{\tilde{x}}^{p^s,p^u}\in\mathcal{V}$, there are only finitely many double charts $\psi_{\tilde{y}}^{q^s,q^u}$ in $\mathcal{V}$ such that $\psi_{\tilde{x}}^{p^s,p^u}\rightarrow\psi_{\tilde{y}}^{q^s,q^u}$ or $\psi_{\tilde{y}}^{q^s,q^u}\rightarrow \psi_{\tilde{x}}^{p^s,p^u}$ by Proposition \ref{graph A}(1). Hence, every vertex of $\mathcal{G}$ has finite degree. Let
$$\Sigma(\mathcal{G}):=\big\{\{\psi_{\tilde{x}_i}^{p_i^s,p_i^u}\}_{i\in\mathbb{Z}}:(\psi_{\tilde{x}}^{p^s,p^u},\psi_{\tilde{y}}^{q^s,q^u})\in\mathcal{E},\forall i\in \mathbb{Z\big\}}$$
 equipped with left shift map $\sigma$ which is \textit{the topological Markov shift associated to $\mathcal{G}$}.
\begin{mdef}
    Let $(Q_k)_{k\in\mathbb{Z}}$ be a sequence in $\mathcal{I}:= \{I^{-\frac{\ell}{4}}(1)\}_{\ell \in \mathbb{N}}$. A sequence of pairs $\{(p_k^s,p_k^u)\}_{k\in\mathbb{Z}}$ is called \textit{$I$-strongly subordinated to $(Q_k)_{k\in\mathbb{Z}}$} if $\forall k\in\mathbb{Z}$:
\begin{enumerate}[label=(\arabic*)]
\item $p_k^s,p_k^u\in\mathcal{I}$ and $0<p_k^s,p^u_k\leq Q_k$;
\item $p_{k+1}^u=\min\{I(p_k^u),Q_{k+1}\}$ and $p_{k-1}^s=\min\{I(p_k^s),Q_{k-1}\}$.
\end{enumerate}
\end{mdef}

The following two lemmas  are the counterparts of lemma 2.26 and 2.27 in \cite{Ovadia2}, we omit the proofs.
\begin{lem}\label{CG1}
  Let $(Q_k)_{k\in\mathbb{Z}}$ be a sequence in $\mathcal{I}:= \{I^{-\frac{\ell}{4}}(1)\}_{\ell \in \mathbb{N}}$, and a sequence $(q_k)_{k\in\mathbb{Z}}$ in $\mathcal{I}$ such that $0<q_k\leq Q_k$ and $q_k=I^{\pm1}(q_{k+1})\ \forall k\in\mathbb{Z}$. Then there exists a sequence of pairs $\{(p_k^s,p_k^u)\}_{k\in\mathbb{Z}}$ which is $I$-strongly subordinated to $(Q_k)_{k\in\mathbb{Z}}$ and $p_k^s\land p_k^u\geq q_k,\ \forall k\in\mathbb{Z}$.
\end{lem}

\begin{lem}\label{limsup}
 Suppose $\{(p_k^s,p_k^u)\}_{k\in\mathbb{Z}}$ $I$-strongly subordinated to $(Q_k)_{k\in\mathbb{Z}}$ and $p_k:=p_k^s\land p_k^u$. If $\limsup_{k\rightarrow\infty}p_k>0$ and $\limsup_{k\rightarrow -\infty}p_k>0$, then there exist infinitely many $k>0$ and $k<0$ s.t. $p_k^s= Q_k$ and $p_k^u=Q_k$.
\end{lem}

\begin{prop}\label{exist}
    $\forall \tilde{x}\in RST$, there exists a chain $\Big\{\psi_{\tilde{x}_k}^{p_k^s,p_k^u}\Big\}_{k\in\mathbb{Z}}\in\Sigma(\mathcal{G})$ such that $\psi_{\tilde{x}_k}^{p_k^s,p_k^u}$ $I$-overlaps $\psi_{\hat{f}^k(\tilde{x})}^{p_k^s,p_k^u}\ \forall k\in\mathbb{Z}$.
    \end{prop}
    \begin{pf}
The proof is the same as the proof of Proposition 2.28 in \cite{Ovadia2}. For the completeness, we give a brief sketch.

    Step 1: Choose $q_k\in\mathcal{I}\bigcap \big[I^{-\frac{1}{4}}(q(\hat{f}^k(\tilde{x}))),I^{\frac{1}{4}}(q(\hat{f}^k(\tilde{x})))\big]$, then there exists a sequence $\{(q_k^s,q_k^u)\}_{k\in\mathbb{Z}}$ which is  $I$-strongly subordinated to $I^{-\frac{1}{4}}(Q(\hat{f}^k(\tilde{x})))\}_{k\in\mathbb{Z}}$ and $q_k^s\land q_k^u\geq q_k$ by Lemma \ref{CG1} and $q_k\leq I^{\frac{1}{4}}(q(\hat{f}^k(\tilde{x})))\leq I^{-\frac{1}{4}}(Q(\hat{f}^k(\tilde{x})))$ since $I^{-\frac{1}{4}}$ is increasing and $I^{-\frac{1}{4}}(t)<t\ \forall t\in(0,e^\Gamma)$.

    Step 2: Let $\eta_k:=q_k^s\land q_k^u$, then there exists a sequence $\{\psi_{\tilde{x}_k}^{\eta_k}\}_{k\in\mathbb{Z}}$ satisfying the properties in Proposition \ref{graph A}(2).

    Step 3: Using Lemma \ref{CG1} for $(\eta_k)_{k\in\mathbb{Z}}$, then there exists a sequence $\{(p_k^s,p_k^u)\}_{k\in\mathbb{Z}}$ which is  $I$-strongly subordinated to $\{Q(\tilde{x}_k)\}_{k\in\mathbb{Z}}$.

    Step 4: Verify that $\{\psi_{\tilde{x}_k}^{p_k^s,p_k^u}\}_{k\in\mathbb{Z}}$ satisfies the conditions.
    \end{pf}

\subsection{\texorpdfstring{Admissible Manifold}{Admissible Manifold}}
  \begin{mdef}
      For every $\tilde{x}\in RST$, an \textit{$s$-admissible manifold} $V^s$ at double chart $\psi_{\tilde{x}}^{p^s,p^u}$ has the form of
      $$V^s=\psi_{\tilde{x}}\big\{(v_1,G(v_1)):v_1\in B^{d_s(\tilde{x})}_{p^s}(0)\big\},$$
      where $B^{d_s(\tilde{x})}_{p^s}(0)\subset\mathbb{R}^{d_s(\tilde{x})}$ is a ball at the origin with radius $p^s$ and $G: B^{d_s(\tilde{x})}_{p^s}(0)\rightarrow\mathbb{R}^{d_u(\tilde{x})}$ is a $C^{1+\frac{\beta}{2}}$ function such that
\begin{enumerate}[label=(AM\arabic*), labelwidth=3em, leftmargin=!]
\item $|G(0)|\leq 10^{-3}(p^s\land p^u)^2$;\label{AM1}
\item $\|d_0G\|\leq \frac{1}{2}(p^s\land p^u)^{\frac{\beta}{2}}$;\label{AM2}
\item $\|d_\cdot G\|_{\frac{\beta}{2}}= \|d_\cdot G\|_{C^0}+H\ddot{o}l_{\frac{\beta}{2}}(d_\cdot G)<\frac{1}{2}$.\label{AM3}
\end{enumerate}
A \textit{$u$-admissible manifold} $V^u$ at double chart $\psi_{\tilde{x}}^{p^s,p^u}$ has the form of
$$V^u=\psi_{\tilde{x}}\big\{(G(v_2),v_2):v_2\in B^{d_u(\tilde{x})}_{p^u }(0)\big\},$$
where $G: B^{d_u(\tilde{x})}_{p^u}(0)\rightarrow\mathbb{R}^{d_s(\tilde{x})}$ is a $C^{1+\frac{\beta}{2}}$ function such that \ref{AM1}-\ref{AM3}. The map $G$ is called  the \textit{representing function}.
  \end{mdef}

By assumptions \ref{A1}-\ref{A3}, the representing function $G$ satisfies: $\forall t\in B_{p^{s/u}}(0)\subset\mathbb{R}^{d_{s/u}(\tilde{x})}$
\begin{equation}\label{G1}
    \|d_tG\|\leq \|d_0G\|+H\ddot{o}l_{\frac{\beta}{2}}(d_\cdot G)|t|^{\frac{\beta}{2}}\leq \frac{1}{2}(p^s\land p^u)^{\frac{\beta}{2}}+\frac{1}{2}(p^{s/u})^{\frac{\beta}{2}}\leq (p^{s/u})^{\frac{\beta}{2}}\leq \epsilon\implies Lip(G)\leq \epsilon,
\end{equation}
and
\begin{equation}\label{G2}
    |G(t)|\leq |G(0)|+Lip(G)\cdot p^{s/u}\leq 10^{-3}(p^s\land p^u)^2+\epsilon p^{s/u}\leq 10^{-2}p^{s/u}.
\end{equation}

Let $v:=\psi_{\tilde{x}}^{p^s,p^u}$, denote $\mathcal{M}^s(v):=\{V^s: V^s \text{ is the $s$-admissible manifold  at $v$\}}$, endowing the metric $d_{C^0}(V_1,V_2):=\|G_1-G_2\|_{C^0}$ and $d_{C^1}(V_1,V_2):=\|d_{\cdot}G_1-d_{\cdot}G_2\|_{C^0}$ on $B_{p^{s}}(0)$, where $G_1,G_2$ are the representing functions of $V_1, V_2$, respectively. Similarly, we can define $\mathcal{M}^u(v)$.

\begin{prop}\label{ad}
    For $\epsilon>0$ small enough, let $v:=\psi_{\tilde{x}}^{p^s,p^u}$, then for every $V^s\in \mathcal{M}^s(v)$ and $V^u\in \mathcal{M}^u(v)$, we have
\begin{enumerate}[label=(\arabic*)]
\item $\exists !\ p=\psi_{\tilde{x}}(v_1)$ s.t. $V^s\bigcap V^u=\{p\}$, where $v_1\in\mathbb{R}^d,|v|\leq 10^{-2}(p^s\land p^u)^2$;
\item Define $P:\mathcal{M}^s(v)\times \mathcal{M}^u(v)\rightarrow M$ by $(V^s,V^u)\mapsto p$, where $p$ is the intersection of $V^s$ and $V^u$. Then $P$ is a Lipschitz map with Lipschitz constant 3 i.e. if $V^s_i\bigcap V^u_i=\{p_i\}$ for $i=1,2$, then $d(p_1,p_2)\leq 3[d_{C^0}(V^s_1,V^s_2)+d_{C^0}(V^u_1,V^u_2)]$, where $V_i^s\in\mathcal{M}^s(v),V_i^u\in\mathcal{M}^u(v)$.
\end{enumerate}
\end{prop}
\begin{pf}
Since the process of proof does not involve whether the mapping $f$ is invertible, this proof is the same as Ovadia's \cite{Ovadia2}.
\end{pf}

 The following theorem will show that the Graph Transform preserves admissibility.

 \begin{thm}\label{TF}
     For $\epsilon>0$ small enough, given two double charts $w_1=\psi_{\tilde{x}}^{p^s,p^u}, w_2=\psi_{\tilde{y}}^{q^s,q^u}$, and let $V^s$ be a $s$-admissible manifold in $w_2$. If $w_1\rightarrow w_2$, then
\begin{enumerate}[label=(\arabic*)]
\item $f_{\tilde{x}}^{-1}V^s$ contains a unique $s$-admissible manifold in $\psi_{\tilde{x}}^{p^s,p^u}$, denote the $s$-admissible manifold as $\mathcal{F}^s_{w_1,w_2}(V^s)$;\label{1}
\item For any $u$-admissible $V^u$ in $\psi_{\tilde{x}}^{p^s,p^u}$, $f_{\tilde{x}}^{-1}V^s$ intersects $V^u$ at a unique point;
\item If $p:=\psi_{\tilde{y}}(0,G(0))\in V^s$, then $f^{-1}_{\tilde{x}}(p)\in \mathcal{F}^s_{w_1,w_2}(V^s)$.
\end{enumerate}
 For the $u$-case, there are similar statements.
 \end{thm}
 \begin{pf}
      Suppose $G$ is the representing function of $V^s$, then $V^s=\psi_{\tilde{y}}\{(v,G(v)):v\in B_{q^s}^{d_s(\tilde{y})}(0)\}$. Notice that  $\psi_{\tilde{x}}^{-1}(f_{\tilde{x}}^{-1}(V^s))= \psi_{\tilde{x}}^{-1}(f_{\tilde{x}}^{-1}(\psi_{\tilde{y}}\text{Graph}(G)))=F^{-1}_{\tilde{x},\tilde{y}}\text{Graph}(G)$. Denote $\eta:=q^s\land q^u$, by Theorem \ref{F1}, $F^{-1}_{\tilde{x},\tilde{y}}\text{Graph}(G)$ can be represented as the form of
      \begin{align}\label{Graph2}
          \big\{(D_s^{-1}(v)+h_s^-(v,G(v)),D_u^{-1}(G(v))+h_u^{-}(v,G(v))):v\in  B_{q^s}^{d_s(\tilde{y})}(0)\big\},
      \end{align}
    where $\|D_s\|\leq e^{-\frac{1}{s^2(\hat{f}^{-1}(\tilde{y}))}}$, $\|D_u^{-1}\|\leq e^{-\frac{1}{u^2(\tilde{y})}}$, $|h_{s/u}^-(0,0)|\leq \epsilon \eta^3$, $\|dh_{s/u}^-\|\leq \sqrt{\epsilon}\eta^{\frac{\beta}{2}}$, and $H\ddot{o}l_{\frac{\beta}{2}}\bar{H}^-=H\ddot{o}l_{\frac{\beta}{2}}(h_s^-,h_u^-)\leq \sqrt{\epsilon}$.

    (1) Using graph transform, we will prove that (\ref{Graph2}) can be represented as the form of $\{(w,F(w)):w\in B_{p^s}^{d_s(\tilde{x})}(0)\}$ and $F$ satisfies \ref{AM1}-\ref{AM3}. As the ideas in \cite{Katok1, Sarig}, denote $\phi(v):=D_s^{-1}(v)+h_s^-(v,G(v))$ on $B_{q^s}^{d_s(\tilde{y})}(0)$, then $F=[D_u^{-1}(G(\cdot))+h_u^-(\cdot,G(\cdot)))]\circ \phi^{-1}$. The map $\phi^{-1}$ is well-defined on $\phi (B_{q^s}^{d_s(\tilde{y})}(0))$ since $\phi$ is expanding on $B_{q^s}^{d_s(\tilde{y})}(0)$. Indeed, $\forall v\in B_{q^s}^{d_s(\tilde{y})}(0)$
\begin{align*}
        \|d_v\phi\|&= \|D_s^{-1}+d_{(v,G(v))}h_s^-\begin{pmatrix}
            I\\
            d_vG
        \end{pmatrix}\|\geq \|D_s^{-1}\|-\sqrt{\epsilon}\eta^{\frac{\beta}{2}}(\|d_vG\|+1)\\
        &\geq \|D_s^{-1}\|-\sqrt{\epsilon}\eta^{\frac{\beta}{2}}(\epsilon+1) \text{ (by inequality (\ref{G1}))}\\
        &\geq \|D_s^{-1}\|\big(1-\sqrt{\epsilon}\eta^{\frac{\beta}{2}}(\epsilon+1)\big)\text{ (since }\|D_s^{-1}\|\geq \|D_s\|^{-1}\geq e^{\frac{1}{s^2(\hat{f}^{-1}(\tilde{y}))}}\geq1)\\
        &\geq e^{-\eta^{\frac{\beta}{2}}}\|D_s\|^{-1} \ (\epsilon \text{ is small enough s.t. }\sqrt{\epsilon}(\epsilon+1 )\ll1),
\end{align*}
where $I$ is a $d_s(\tilde{y})\times d_s(\tilde{y})$ identity matrix. By Lemma \ref{C}(1)(3), $\|C_0(\tilde{y})^{-1}\|^2\geq s^2(\tilde{y})$ and
$\|C_0(\tilde{y})^{-1}\|\geq \frac{\sqrt{2}}{2}\|C_0(\hat{f}^{-1}(\tilde{y}))^{-1}\|d(p(\hat{f}^{-1}(\tilde{y}),\Gamma_\infty)^{2a}\geq \frac{\sqrt{2}}{2}\|C_0(\hat{f}^{-1}(\tilde{y}))^{-1}\|\rho(\tilde{y})^{2a}.$
And so
\begin{equation}\label{>}
e^{-\eta^{\frac{\beta}{2}}}\|D_s\|^{-1}\geq e^{-Q(\tilde{y})^{\frac{\beta}{2}}}\cdot e^{\frac{1}{s^2(\hat{f}^{-1}(\tilde{y}))}}\geq e^{-Q(\tilde{y})^{\frac{\beta}{2}}+\|C_0(\hat{f}^{-1}(\tilde{y}))^{-1}\|^{-2}}\geq e^{-Q(\tilde{y})^{\frac{\beta}{2}}+\frac{1}{2}\rho(\tilde{y})^{4a}\|C_0(\tilde{y})^{-1}\|^{-2}}>1,\end{equation}
since $Q(\tilde{y})^{\frac{\beta}{2}}\leq (\epsilon^{\frac{20}{\beta}})^{\frac{\beta}{2}}\|C_0(\tilde{y})^{-1}\|^{-2}\rho(\tilde{y})^{4a}\leq \frac{1}{2}\rho(\tilde{y})^{4a}\|C_0(\tilde{y})^{-1}\|^{-2}$ (recall that $\gamma\geq \frac{20}{\beta}$ and $a>1$).
 \begin{claim}\label{phi}
     For small enough $\epsilon$, the image $\phi(B_{q^s}^{d_s(\tilde{y})}(0))$ contains $B_{e^{\frac{1}{s^2(\hat{f}^{-1}(\tilde{y}))}-\frac{1}{2}\eta^{\frac{\beta}{2}}}q^s}^{d_s(\tilde{y})}(0)$, then $\phi^{-1}$ is well defined on $B_{e^{\frac{1}{s^2(\hat{f}^{-1}(\tilde{y}))}-\frac{1}{2}\eta^{\frac{\beta}{2}}}q^s}^{d_s(\tilde{y})}(0)$ and it satisfies
\begin{enumerate}[label=(\alph*)]
\item $\|d_{\cdot}\phi^{-1}\|<e^{\eta^{\frac{\beta}{2}}-\frac{1}{s^2(\hat{f}^{-1}(\tilde{y}))}}$;\label{a}
\item $|\phi^{-1}(0)|< 2\sqrt{\epsilon}\eta^{2+\frac{\beta}{2}}$;\label{b}
\item $\|d_{\cdot}\phi^{-1}\|_{\frac{\beta}{2}}\leq e^{6\sqrt{\epsilon}\eta^{\frac{\beta}{2}}-\frac{1}{s^2(\hat{f}^{-1}(\tilde{y}))}}$.\label{c}
\end{enumerate}
 \end{claim}
\begin{pf}
 For $v\in B_{q^s}^{d_s(\tilde{y})}(0)$, $d_v\phi=D_s^{-1}+d_{(v,G(v))}h_s^-\begin{pmatrix}
            I\\
            d_vG
        \end{pmatrix}=:D_s^{-1}+A(v)$. Estimate for $\|A(v)\|$: $\forall v_1,v_2\in B_{q^s}^{d_s(\tilde{y})}(0)$
\begin{equation}\label{Av}
    \begin{aligned}
        &\quad \|A(v_1)-A(v_2)\|=\Big\|d_{(v_1,G(v_1))}h_s^-\begin{pmatrix}
            I\\
            d_{v_1}G
        \end{pmatrix}-d_{(v_2,G(v_2))}h_s^-\begin{pmatrix}
            I\\
            d_{v_2}G
        \end{pmatrix}\Big\|\\
        &\leq \Big\|d_{(v_1,G(v_1))}h_s^-\begin{pmatrix}
            I\\
            d_{v_1}G
        \end{pmatrix}-d_{(v_1,G(v_1))}h_s^-\begin{pmatrix}
            I\\
            d_{v_2}G
        \end{pmatrix}\Big\|+\Big\|d_{(v_1,G(v_1))}h_s^-\begin{pmatrix}
            I\\
            d_{v_2}G
        \end{pmatrix}-d_{(v_2,G(v_2))}h_s^-\begin{pmatrix}
            I\\
            d_{v_2}G
        \end{pmatrix}\Big\|\\
        &\leq \|d_{(v_1,G(v_1))}h_s^-\|\cdot\|d_{v_1}G-d_{v_2}G\|+\|d_{(v_1,G(v_1))}h_s^--d_{(v_2,G(v_2))}h_s^-\|\cdot\Big\|\begin{pmatrix}
            I\\
            d_{v_2}G
        \end{pmatrix}\Big\|\\
        &\leq \sqrt{\epsilon} \eta^{\frac{\beta}{2}}\cdot \frac{1}{2}|v_1-v_2|^{\frac{\beta}{2}}+\sqrt{\epsilon}|v_1-v_2|^{\frac{\beta}{2}}\cdot (1+ \epsilon)^{1+\frac{\beta}{2}}\leq 3\sqrt{\epsilon}|v_1-v_2|^{\frac{\beta}{2}}.
    \end{aligned}
\end{equation}
Hence,  $H\ddot{o}l_{\frac{\beta}{2}}A\leq 3\sqrt{\epsilon}$. For $v\in B_{q^s}^{d_s(\tilde{y})}(0)$,
\begin{equation}\label{Av1}
    \|A(v)\|\leq\|A(0)\|+H\ddot{o}l_\frac{\beta}{2}A\cdot|v|^{\frac{\beta}{2}}\leq \sqrt{\epsilon}\eta^{\frac{\beta}{2}}(1+\epsilon)+3\sqrt{\epsilon}(q^s)^{\frac{\beta}{2}} \leq 5\sqrt{\epsilon}(q^s)^{\frac{\beta}{2}}.
\end{equation}
Estimate for the error term $h_s^-$: $\forall v\in B_{q^s}^{d_s(\tilde{y})}(0)$,
$$|h_s^-(v,G(v))|\leq |h_s^-(0,0)|+\|d_{\cdot}h_s^-\|\cdot(|v|+|G(v)|) \leq \epsilon\eta^3+2\sqrt{\epsilon}(q^s)^{1+\frac{\beta}{2}}\leq 3\sqrt{\epsilon}(q^s)^{1+\frac{\beta}{2}},$$
by inequality (\ref{G2}).
Given $u\in B_{e^{\frac{1}{s^2(\hat{f}^{-1}(\tilde{y}))}-\frac{1}{2}\eta^{\frac{\beta}{2}}}q^s}^{d_s(\tilde{y})}(0)$, assume $\phi(v)=u$. Define the map $T$ on $B_{q^s}^{d_s(\tilde{x})}(0)$ by
$$T(v)=D_s[u-h_s^-(v,G(v))].$$
Then $\phi(v)=u$ iff $v$ is the fixed point of $T$. Next, we will prove that $T$ has a unique fixed point $v_0$ and $v_0\in B_{q^s}^{d_s(\tilde{y})}(0)$. Indeed,
\begin{align*}
    |T(v)|&\leq \|D_s\|(|\phi(v)|+|h_s^-(v,G(v))|)\\
    &\leq e^{-\frac{1}{s^2(\hat{f}^{-1}(\tilde{y}))}}(e^{\frac{1}{s^2(\hat{f}^{-1}(\tilde{y}))}-\frac{1}{2}\eta^{\frac{\beta}{2}}}q^s+3\sqrt{\epsilon}(q^s)^{1+\frac{\beta}{2}})\\
    &\leq (e^{-\frac{1}{2}\eta^{\frac{\beta}{2}}}+3\sqrt{\epsilon}(q^s)^{\frac{\beta}{2}})q^s\leq q^s.
\end{align*}
This implies that $T(B_{q^s}^{d_s(\tilde{y})}(0))\subset B_{q^s}^{d_s(\tilde{y})}(0)$. And $\|d_vT\|\leq \|D_s\|\cdot \|A(v)\|\leq e^{\frac{1}{s^2(\hat{f}^{-1}(\tilde{y}))}}\cdot 4\sqrt{\epsilon}\eta^{\frac{\beta}{2}}\leq 1$. By the fixed point theorem, $T$ has the unique fixed point $v_0\in B_{q^s}(0)$.

(a) $\|d_{\cdot}\phi^{-1}\|=\|(d_{\cdot}\phi)^{-1}\|\leq \|(d_{\cdot}\phi)\|^{-1}\leq e^{\eta^{\frac{\beta}{2}}}\|D_s\|<e^{\eta^{\frac{\beta}{2}}-\frac{1}{s^2(\hat{f}^{-1}(\tilde{y}))}}$.

(b) Since $|\phi(0)|=|h_s^-(0,G(0))|\leq |h_s^-(0,0)|+\|d_{\cdot}h_s^-\|\cdot|G(0)|\leq \epsilon\eta^3+\sqrt{\epsilon}\eta^{\frac{\beta}{2}}\cdot 10^{-3}\eta^2\leq 2\sqrt{\epsilon}\eta^{2+\frac{\beta}{2}}$, then $\phi(0)\in B_{e^{\frac{1}{s^2(\hat{f}^{-1}(\tilde{y}))}-\frac{1}{2}\eta^{\frac{\beta}{2}}}q^s}^{d_s(\tilde{x})}(0)$ for small enough $\epsilon$. Hence,
$$|\phi^{-1}(0)|=|\phi^{-1}(0)-\phi^{-1}(\phi(0))|\leq \|d(\phi^{-1})\|\cdot|\phi(0)|<e^{\eta^{\frac{\beta}{2}}-\frac{1}{s^2(\hat{f}^{-1}(\tilde{y}))}}\cdot 2\sqrt{\epsilon}\eta^{2+\frac{\beta}{2}}< 2\sqrt{\epsilon}\eta^{2+\frac{\beta}{2}}.$$

(c) Let $w:=\phi(v)$, then $d_w\phi^{-1}= (D_s^{-1}+A(v))^{-1}=(Id+D_sA(v))^{-1}D_s$. Hence
\begin{align}\label{norm}
\|d_{\cdot}\phi^{-1}\|_{\frac{\beta}{2}}\leq \|(Id+D_sA(\phi^{-1}(\cdot)))^{-1}\|_{\frac{\beta}{2}}\cdot\|D_s\|_{\frac{\beta}{2}}\leq \frac{1}{1-\|D_s A(\phi^{-1}(\cdot))\|_{\frac{\beta}{2}}}\cdot\|D_s\|_{\frac{\beta}{2}}.
\end{align}
Notice that $D_s$ and $\phi^{-1}$ are contractive since $\|d(\phi^{-1})\|<e^{\eta^{\frac{\beta}{2}}-\frac{1}{s^2(\hat{f}^{-1}(\tilde{y}))}}<1$. By inequalities (\ref{Av}) and (\ref{Av1}), $\|A(v)\|_{\frac{\beta}{2}}\leq 5\sqrt{\epsilon}(q^s)^{\frac{\beta}{2}}+ 3\sqrt{\epsilon}\leq 8\sqrt{\epsilon}$, and $\|D_s A(\phi^{-1}(\cdot))\|_{\frac{\beta}{2}}\leq e^{-\frac{1}{s^2(\hat{f}^{-1}(\tilde{y}))}}\cdot 8\sqrt{\epsilon}<1$. Hence,
$$\|d_{\cdot}\phi^{-1}\|_{\frac{\beta}{2}}\leq \frac{1}{1-8\sqrt{\epsilon}}\cdot e^{-\frac{1}{s^2(\hat{f}^{-1}(\tilde{y}))}}\leq e^{10\sqrt{\epsilon}-\frac{1}{s^2(\hat{f}^{-1}(\tilde{y}))}},$$
for small enough $\epsilon>0$.
\end{pf}
\begin{claim}\label{claim3}
    The map $F:\mathbb{R}^{d_s(\tilde{x})}\rightarrow\mathbb{R}^{d_u(\tilde{x})}$ satisfies \ref{AM1}-\ref{AM3}.
\end{claim}
\begin{pf}
    \ref{AM1}: Since $Lip(G)<\epsilon$ and  $|G(0)|\leq 10^{-3}(p^s\land p^u)^2$, we have
    \begin{equation}\label{phi1}
 \begin{aligned}
       |(\phi^{-1}(0),G(\phi^{-1}(0)))|&\leq |\phi^{-1}(0)|+|G(0)|+Lip(G)\cdot|\phi^{-1}(0)|\\
       &\leq (1+\epsilon)2\sqrt{\epsilon}\eta^{2+\frac{\beta}{2}}+10^{-3}\eta^2\leq 2\epsilon^2\eta^2.
    \end{aligned}
    \end{equation}
Then
   \begin{equation*}
    \begin{aligned}
        |F(0)|&=|D_u^{-1}(G(\phi^{-1}(0)))+h_u^-(\phi^{-1}(0),G(\phi^{-1}(0)))|\\
        &\leq \|D_u^{-1}\|(|G(0)|+Lip(G)|(\phi^{-1}(0)|)+(|h_u^-(0)|+\|d_{\cdot}h_u^-\|\cdot|(\phi^{-1}(0),G(\phi^{-1}(0)))|)\\
         &\leq e^{-\frac{1}{u^2(\tilde{y})}}(10^{-3}\eta^2+ \epsilon\cdot 2\sqrt{\epsilon}\eta^{2+\frac{\beta}{2}}) +\epsilon\eta^3+\sqrt{\epsilon}\eta^{\frac{\beta}{2}}\cdot 2\epsilon^2\eta^2\text{ (by inequality (\ref{phi1})})\\
         &\leq  e^{-\frac{1}{u^2(\tilde{y})}}[10^{-3}\eta^2+ \epsilon\cdot 2\sqrt{\epsilon}\eta^{2+\frac{\beta}{2}}+\sqrt{2}d(p(\tilde{y}),\Gamma_{\infty})^{-a}\cdot2\sqrt{\epsilon}\eta^{2+\frac{\beta}{2}}]\text{ (by Lemma \ref{C}}(2)) \\
         &\leq e^{-\frac{1}{u^2(\tilde{y})}}\cdot 10^{-3}\eta^2[1+2\cdot 10^3(\epsilon +\sqrt{2}d(p(\tilde{y}),\Gamma_{\infty})^{-a})\cdot \sqrt{\epsilon}\eta^{\frac{\beta}{2}}]\\
         &\leq e^{-\frac{1}{u^2(\tilde{y})}}\cdot 10^{-3}(p^s\land p^u)^2\cdot e^{\epsilon\eta^{\frac{\beta}{4}}}(\frac{q^s\land q^u}{p^s\land p^u})^2\text{ (since }\eta^{\frac{\beta}{4}}\cdot d(p(\tilde{y}),\Gamma_{\infty})^{-a})\ll\epsilon)\\
         &\leq 10^{-3}(p^s\land p^u)^2\cdot e^{-\frac{1}{u^2(\tilde{y})}}\cdot e^{\epsilon\eta^{\frac{\beta}{4}}}\cdot e^{2\Gamma\eta^{\frac{1}{\gamma}}}\text{ (since } p^s\land p^u\geq I^{-1}(q^s\land q^u))\\
         &\leq 10^{-3}(p^s\land p^u)^2\cdot e^{-\frac{1}{u^2(\tilde{y})}+3\Gamma \eta^{\frac{1}{\gamma}}} \text{ (since } \frac{1}{\gamma}<\frac{\beta}{20}< \frac{\beta}{4}).
         \end{aligned}
\end{equation*}
And $-\frac{1}{u^2(\tilde{y})}+3\Gamma \eta^{\frac{1}{\gamma}}\leq -\frac{1}{u^2(\tilde{y})}+3\Gamma Q(\tilde{y})^{\frac{1}{\gamma}}\leq -\frac{1}{u^2(\tilde{y})}+\|C(\tilde{y})^{-1}\|^{-2}<0$, then $|F(0)|\leq 10^{-3}(p^s\land p^u)^{2}$.

\ref{AM2} : According to $\|d_\cdot G\|_{\frac{\beta}{2}}= \|d_\cdot G\|_{C^0}+H\ddot{o}l_{\frac{\beta}{2}}(d_\cdot G)<\frac{1}{2}$ and claim \ref{phi}\ref{a},
\begin{equation}\label{dF1}
    \|d_{\phi^{-1}(0)}G\|\leq \|d_0G\|+H\ddot{o}l_{\frac{\beta}{2}}d_{\cdot}G\cdot|\phi^{-1}(0)|^{\frac{\beta}{2}}\leq \|d_0G\|+\frac{1}{2}\cdot (2\sqrt{\epsilon}\eta^{2+\frac{\beta}{2}})^{\frac{\beta}{2}}\leq \|d_0G\|+\epsilon^2\eta^\beta(<\epsilon).
\end{equation}
By the proof of Theorem \ref{F1} and inequality (\ref{phi1}),
 \begin{equation}\label{dF2}
     \|d_{(\phi^{-1}(0),G(\phi^{-1}(0))}h_u^-\|\leq \sqrt{\epsilon}|(\phi^{-1}(0),G(\phi^{-1}(0))|^{\frac{\beta}{2}}\leq \sqrt{\epsilon}\eta^\beta.
 \end{equation}
Next, we estimate $\|d_0F\|$:
\begin{equation*}
    \begin{aligned}
\|d_0F\|&= \Big\|D_u^{-1}\cdot d_{\phi^{-1}(0)}G\cdot d_0\phi^{-1}+d_{(\phi^{-1}(0),G(\phi^{-1}(0))}h_u^-\begin{pmatrix}
            I\\
            d_{\phi^{-1}(0)}G
        \end{pmatrix}d_0\phi^{-1}\Big\|\\
        &\leq \big[\|D_u^{-1}\|\cdot \|d_{\phi^{-1}(0)}G\|+\|d_{(\phi^{-1}(0),G(\phi^{-1}(0))}h_u^-\|\cdot (1+\|d_{\phi^{-1}(0)}G\|)\big]\cdot\|d_0\phi^{-1}\|\\
        & \leq \big[\|D_u^{-1}\|\cdot (\|d_0G\|+\epsilon^2\eta^\beta)+\sqrt{\epsilon}\eta^\beta \cdot 2\big]\cdot e^{(q^s)^{\frac{\beta}{2}}-\frac{1}{s^2(\hat{f}^{-1}(\tilde{y}))}}\text{ (by inequalities (\ref{dF1}), (\ref{dF2}), claim \ref{phi}\ref{a}})\\
        &\leq e^{(q^s)^{\frac{\beta}{2}}-\frac{1}{s^2(\hat{f}^{-1}(\tilde{y}))}-\frac{1}{u^2(\tilde{y})}}\cdot (\|d_0G\|+\epsilon^2\eta^\beta+2\sqrt{2}d(p(\tilde{y}),\Gamma_\infty)^{-a}\sqrt{\epsilon}\eta^\beta)\text{ (by Lemma \ref{C}(2)})\\
         &\leq e^{(q^s)^{\frac{\beta}{2}}-\frac{1}{s^2(\hat{f}^{-1}(\tilde{y}))}-\frac{1}{u^2(\tilde{y})}}\cdot \frac{1}{2}(p^s\land p^u)^{\frac{\beta}{2}} \Big[\frac{\|d_0G\|}{\frac{1}{2}(p^s\land p^u)^{\frac{\beta}{2}}}+(\frac{q^s\land q^u}{p^s\land p^u})^{\frac{\beta}{2}}\epsilon\eta^{\frac{\beta}{4}}\Big]\\
         &\leq e^{(q^s)^{\frac{\beta}{2}}-\frac{1}{s^2(\hat{f}^{-1}(\tilde{y}))}-\frac{1}{u^2(\tilde{y})}}\cdot \frac{1}{2}(p^s\land p^u)^{\frac{\beta}{2}}\cdot e^{\frac{\beta}{2}\Gamma(p^s\land p^u)^{\frac{1}{\gamma}}}(1+\epsilon\eta^{\frac{\beta}{4}}) \text{ (by  } q^s\land q^u\leq I(p^s\land p^u))\\
         &\leq \frac{1}{2}(p^s\land p^u)^{\frac{\beta}{2}} \cdot e^{(q^s)^{\frac{\beta}{2}}-\frac{1}{s^2(\hat{f}^{-1}(\tilde{y}))}-\frac{1}{u^2(\tilde{y})}+\frac{\beta}{2}\Gamma(p^s\land p^u)^{\frac{1}{\gamma}}+\epsilon\eta^{\frac{\beta}{4}}}.
    \end{aligned}
\end{equation*}
Similar to the calculation of inequality (\ref{>}), there exists a constant $k_0=k_0(\Gamma,\beta,\epsilon)$ s.t.
\begin{equation}\label{s2}
\frac{1}{s^2(\hat{f}^{-1}(\tilde{y}))}\geq k_0Q(\tilde{y})^{\frac{1}{\gamma}}\geq k_0Q(\tilde{y})^{\frac{\beta}{2}},
\end{equation}
and
\begin{equation}\label{u2}
\frac{1}{u^2(\tilde{y})}\geq k_0Q(\hat{f}^{-1}(\tilde{y}))^{\frac{1}{\gamma}}\geq k_0Q(\hat{f}^{-1}(\tilde{y}))^{\frac{\beta}{2}}.
\end{equation}
Then $(q^s)^{\frac{\beta}{2}}-\frac{1}{s^2(\hat{f}^{-1}(\tilde{y}))}-\frac{1}{u^2(\tilde{y})}+\frac{\beta}{2}\Gamma(p^s\land p^u)^{\frac{1}{\gamma}}+\epsilon\eta^{\frac{\beta}{4}}\leq Q(\tilde{y})^{\frac{\beta}{2}}-\frac{1}{s^2(\hat{f}^{-1}(\tilde{y}))}-\frac{1}{u^2(\tilde{y})}+ \frac{\beta}{2}\Gamma e^{\epsilon}\eta^{\frac{1}{\gamma}}+Q(\tilde{y})^{\frac{\beta}{4}}<0.$
Hence, $\|d_0F\|\leq \frac{1}{2}(p^s\land p^u)^{\frac{\beta}{2}}$.

\ref{AM3}: By Theorem \ref{F1}, $(1+\epsilon)q^s+10^{-3}\eta^2\leq 2q^s$. And  $\|d_\cdot h_u^-\|_{\frac{\beta}{2}}\leq \sqrt{\epsilon}\eta^{\frac{\beta}{2}}+\sqrt{\epsilon}\leq 3\sqrt{\epsilon}(q^s)^{\frac{\beta}{2}}$. Then
\begin{equation*}
    \begin{aligned}
        \|d_\cdot F\|_{\frac{\beta}{2}}&=\|D_u^{-1}\cdot d_{\phi^{-1}(\cdot)}G\cdot d_\cdot\phi^{-1}+d_{(\phi^{-1}(\cdot),G(\phi^{-1}(\cdot))}h_u^-\begin{pmatrix}
            I\\
            d_{\phi^{-1}(\cdot)}G
        \end{pmatrix}d_\cdot\phi^{-1}\|_{\frac{\beta}{2}}\\
        &\leq \big[\|D_u^{-1}\|\cdot \|d_{\cdot}G\|_{\frac{\beta}{2}}+\|d_{\cdot}h_u^-\|_{\frac{\beta}{2}}(1+\|d_{\cdot}G\|_{\frac{\beta}{2}})\big]\|d_\cdot\phi^{-1}\|_{\frac{\beta}{2}}\\
        &\leq e^{6\sqrt{\epsilon}(q^s)^{\frac{\beta}{2}}-\frac{1}{s^2(\hat{f}^{-1}(\tilde{y}))}}(\|d_{\cdot}G\|_{\frac{\beta}{2}}\|D_u^{-1}\|+4\sqrt{\epsilon})\\
        &\leq e^{6\sqrt{\epsilon}(q^s)^{\frac{\beta}{2}}-\frac{1}{s^2(\hat{f}^{-1}(\tilde{y}))}-\frac{1}{u^2(\tilde{y})}}(\|d_{\cdot}G\|_{\frac{\beta}{2}}+4e^{\frac{1}{2}}\sqrt{\epsilon})\\
        &\leq \frac{1}{2}e^{6\sqrt{\epsilon}(q^s)^{\frac{\beta}{2}}-\frac{1}{s^2(\hat{f}^{-1}(\tilde{y}))}-\frac{1}{u^2(\tilde{y})}}\cdot e^{8e^{\frac{1}{2}}\sqrt{\epsilon}}.
    \end{aligned}
\end{equation*}
And by inequality (\ref{s2}), $6\sqrt{\epsilon}(q^s)^{\frac{\beta}{2}}-\frac{1}{s^2(\hat{f}^{-1}(\tilde{y}))}-\frac{1}{u^2(\tilde{y})}+8e^{\frac{1}{2}}\sqrt{\epsilon}<0$, then $\|d_{\cdot}F\|_{\frac{\beta}{2}}<\frac{1}{2}$.
\end{pf}

 Since $\frac{1}{u^2(\tilde{y})}>\|C(\tilde{y})^{-1}\|^{-2}\geq (\frac{1}{2}+\Gamma)(q^s)^{\frac{1}{\gamma}}\geq \frac{1}{2}(q^s)^{\frac{\beta}{2}}+\Gamma(q^s)^{\frac{1}{\gamma}}$ and $p^s=\min\{I(q^s),Q(\tilde{x})\}$, for $v\in B_{q^s}^{d_s(\tilde{y})}(0)$,
 \begin{equation}\label{qs}
     e^{\frac{1}{u^2(\tilde{y})}-\frac{1}{2}(q^s)^{\frac{\beta}{2}}}q^s\geq e^{\Gamma(q^s)^{\frac{1}{\gamma}}}q^s=I(q^s)\geq p^s,
 \end{equation}
 for small enough $\epsilon$. Hence, $f^{-1}_{\tilde{x}}V^s$ contains   an $s$-admissible manifold $\mathcal{F}^s_{w_1,w_2}(V^s)=\psi_{\tilde{x}}\{(w,F(w)):w\in B_{p^s}^{d_s(\tilde{x})}(0)\}$. The proof for the uniqueness of the $s$-admissible manifold contained in $f_{\tilde{x}}^{-1}V^s$ is   the same as the proof in \cite{Ovadia2}.

 (2) For any $u$-admissible $V^u$ in $\psi_{\tilde{x}}^{p^s,p^u}$, $f_{\tilde{x}}^{-1}V^s$ intersects $V^u$ at least at one point by Proposition \ref{ad}(1). And by McShane's extension formula, we can prove $f_{\tilde{x}}^{-1}V^s$ intersects $V^u$ at most at one point. The detailed proof is shown in \cite{Ovadia2}.

 (3) By the Theorem \ref{F1} and $|G(0)|\leq 10^{-3}(p^s\land p^u)^2$, we have
$$|h_s^-(0,G(0))|\leq |h_s^-(0,0)|+\|d_{\cdot}h\|\cdot|G(0)|\leq \epsilon\eta^3+\sqrt{\epsilon}\eta^{\frac{\beta}{2}}\cdot 10^{-3}\eta^2\leq \sqrt{\epsilon}\eta^{2+\frac{\beta}{2}}\leq \sqrt{\epsilon}(I(p^s\land p^u))^{2+\frac{\beta}{2}}\ll p^s.$$
 \end{pf}
\begin{rmk} Since the map $f$ has singularities, there is no uniform lower bound on $\|D_s\|$ and $\|D_u^{-1}\|$. By Lemma \ref{C}(2), we have $\|D_s\|, \|D_u^{-1}\|\geq \frac{\sqrt{2}}{2}d(x,\Gamma_\infty)^a$. Hence, stronger control over the origin offset $|G(0)|$ is required when proving $F$ satisfies the admissibility parameters \ref{AM1}-\ref{AM3}.
\end{rmk}

\begin{mdef}
    Suppose $v_1\rightarrow v_2$, where $v_1=\psi_{\tilde{x}}^{p^s,p^u}, v_2=\psi_{\tilde{y}}^{q^s,q^u}$.
 \begin{enumerate}[label=(\arabic*)]
\item The \textit{Graph Transform} $\mathcal{F}^s_{v_1,v_2}$ maps  an $s$-admissible manifold $V^s$ in $\psi_{\tilde{y}}^{q^s,q^u}$ with representing function $F:B_{q^s}^{d_s(\tilde{y})}(0)\rightarrow \mathbb{R}^{d_u(\tilde{y})}$ to the unique $s$-admissible manifold in $\psi_{\tilde{x}}^{p^s,p^u}$ with representing function $\tilde{F}:B_{p^s}^{d_s(\tilde{x})}(0)\rightarrow \mathbb{R}^{d_u(\tilde{x})}$ such that it is contained in $f^{-1}_{\tilde{x}}(V^s)$ and
    $$\{(v,\tilde{F}(v)):v\in B_{p^s}^{d_s(\tilde{x})}(0)\}\subset F^{-1}_{\tilde{x},\tilde{y}}\{(v,F(v)):v\in B^{d_s(\tilde{y})}_{q^s}(0)\}.$$
\item  The \textit{Graph Transform} $\mathcal{F}^u_{v_1,v_2}$ maps a $u$-admissible manifold $V^u$ in $\psi_{\tilde{x}}^{p^s,p^u}$ with representing function $G:B_{p^u}^{d_u(\tilde{x})}(0)\rightarrow \mathbb{R}^{d_s(\tilde{x})}$ to the unique $u$-admissible manifold in $\psi_{\tilde{y}}^{q^s,q^u}$ with representing function $\tilde{G}:B_{q^u}^{d_u(\tilde{y})}(0)\rightarrow \mathbb{R}^{d_s(\tilde{y})}$ such that it is contained in $f(V^u)$ and
    $$\{(\tilde{G}(v),v):v\in B_{q^u}^{d_u(\tilde{y})}(0)\}\subset F_{\tilde{x},\tilde{y}}\{(G(v),v):v\in B^{d_u(\tilde{x})}_{p^u}(0)\}.$$
\end{enumerate}
\end{mdef}
 \begin{prop}\label{Fs}
       Suppose $w_1\rightarrow w_2$, where $w_1=\psi_{\tilde{x}}^{p^s,p^u}, w_2=\psi_{\tilde{y}}^{q^s,q^u}$. For small enough $\epsilon$, and $s$-admissible manifolds $V_1^{s}, V_2^{s}$ in $\psi_{\tilde{y}}^{q^s,q^u}$ with the representing functions $F_1,F_2$, respectively:
\begin{enumerate}[label=(\arabic*), series=myenum]
    \item $d_{C^0}(\mathcal{F}^{s}_{w_1,w_2}(V_1^{s}),\mathcal{F}^{s}_{w_1,w_2}(V_1^{s}))\leq e^{-\frac{1-\epsilon}{u^2(\tilde{y}))}}d_{C^0}(V_1^{s},V_2^{s})$.
    \item $d_{C^1}(\mathcal{F}^s_{w_1,w_2}(V_1^{s}),\mathcal{F}^s_{w_1,w_2}(V_1^{s}))\leq e^{-\frac{1-\epsilon}{s^2(\hat{f}^{-1}(\tilde{y}))}-\frac{1-\epsilon}{u^2(\tilde{y})}}\big((1+\epsilon(q^s\land q^u)^{\frac{\beta}{2}})\|d_{\cdot}F_1-d_{\cdot}F_2\|+10\epsilon\|F_1-F_2\|^{\frac{\beta}{2}}\big)$.
\end{enumerate}
For $u$-admissible manifolds $V_1^{u}, V_2^{u}$ in $\psi_{\tilde{x}}^{p^s,p^u}$ with the representing functions $G_1,G_2$, respectively:
\begin{enumerate}[label=(\arabic*), resume=myenum]
    \item $d_{C^0}(\mathcal{F}^{u}_{w_1,w_2}(V_1^{u}),\mathcal{F}^{u}_{w_1,w_2}(V_1^{u}))\leq e^{-\frac{1-\epsilon}{s^2(\tilde{x}))}}d_{C^0}(V_1^{u},V_2^{u})$.
    \item $d_{C^1}(\mathcal{F}^u_{w_1,w_2}(V_1^{u}),\mathcal{F}^u_{w_1,w_2}(V_1^{u}))\leq e^{-\frac{1-\epsilon}{s^2(\tilde{x}))}-\frac{1-\epsilon}{u^2(\hat{f}(\tilde{x}))}}\big((1+\epsilon(p^s\land p^u)^{\frac{\beta}{2}})\|d_{\cdot}G_1-d_{\cdot}G_2\|+10\epsilon\|G_1-G_2\|^{\frac{\beta}{2}}\big)$.
\end{enumerate}
 \end{prop}
 \begin{pf}
     By Theorem \ref{TF}, $\mathcal{F}^s_{w_1,w_2}(V_1^s)=\psi_{\tilde{x}}[\{(v,\tilde{F_i}(v):v\in B_{p^s}^{d_s(\tilde{x})}(0)\}]$, where $\tilde{F}_i=[D_u^{-1}(F_i(\cdot))+h_u^-(\cdot,F_i(\cdot)))]\circ \phi_i^{-1}$, $\phi_i(v):=D_s^{-1}(v)+h_s^-(v,F_i(v))$ on $B_{q^s}^{d_s(\tilde{y})}(0)$.

     (1) \begin{equation}\label{phi2}
         \begin{aligned}
             \|\phi_1^{-1}-\phi_2^{-1}\|&\leq \|\phi_1^{-1}\circ\phi_1-\phi_2^{-1}\circ\phi_1\|=\|Id-\phi_2^{-1}\circ\phi_1\|\\
             &=\|\phi_2^{-1}\circ\phi_2-\phi_2^{-1}\circ\phi_1\|
             \leq \|d_{\cdot}\phi_2^{-1}\|\cdot\|\phi_1-\phi_2\|\\
             &\leq \|d_{\cdot}\phi_2^{-1}\|\cdot\|h_s^-(\cdot,F_1(\cdot))-h_s^-(\cdot,F_2(\cdot))\|\\
             &\leq \|d_{\cdot}\phi_2^{-1}\|\cdot\|d_{\cdot}h_s^-\|\cdot\|F_1-F_2\|,
         \end{aligned}
     \end{equation}
     the first inequality holds by $\phi_i(B_{q^s}^{d_s(\tilde{y})}(0))\supset B_{p^s}^{d_s(\tilde{x})}(0)$. Then $\frac{\|\phi_1^{-1}-\phi_2^{-1}\|}{\|F_1-F_2\|}\leq \sqrt{\epsilon}\eta^\frac{\beta}{2}$ since $\|d_{\cdot}\phi_2^{-1}\|\leq 1$ (by claim \ref{phi}\ref{a}, $\eta=q^s\land q^u$). And
     \begin{equation*}
         \begin{aligned}
         \|F_1\circ\phi_1^{-1}-F_2\circ \phi_2^{-1}\|&\leq \|F_1\circ\phi_1^{-1}-F_1\circ \phi_2^{-1}\|+\|F_1\circ\phi_2^{-1}-F_2\circ \phi_2^{-1}\|\\
         &\leq Lip(F_1)\|\phi_1^{-1}-\phi_2^{-1}\|+\|F_1-F_2\|,
     \end{aligned}
     \end{equation*}
  since $\phi_i$ is expanding on $B_{q^s}^{d_s(\tilde{y})}(0)$, then $\frac{\|F_1\circ\phi_1^{-1}-F_2\circ \phi_2^{-1}\|}{\|F_1-F_2\|}\leq 1+\epsilon^{\frac{3}{2}}\eta^{\frac{\beta}{2}}$ by inequality  (\ref{G1}). Hence, we can estimate
     \begin{equation*}
         \begin{aligned}
             d_{C^0}(\mathcal{F}^s_{w_1,w_2}(V_1^s),^s(V^s_2))
             &=\|[D_u^{-1}(F_1(\cdot))+h_u^-(\cdot,F_1(\cdot)))]\circ \phi_1^{-1}-[D_u^{-1}(F_2(\cdot))+h_u^-(\cdot,F_2(\cdot)))]\circ \phi_2^{-1}\|\\
             &\leq \|D_u^{-1}\|\cdot\|F_1\circ\phi_1^{-1}-F_2\circ \phi_2^{-1}\|+\|d_{\cdot}h_u^-\|(\|\phi_1^{-1}-\phi_2^{-1}\|+\|F_1\circ\phi_1^{-1}-F_2\circ \phi_2^{-1}\|)\\
             &\leq \|D_u^{-1}\|[(1+\epsilon^{\frac{3}{2}}\eta^{\frac{\beta}{2}})+\sqrt{\epsilon}\eta^{\frac{\beta}{2}}(\sqrt{\epsilon}\eta^{\frac{\beta}{2}}+1+\epsilon^{\frac{3}{2}}\eta^{\frac{\beta}{2}})]\|F_1-F_2\|\\
             &\leq e^{-\frac{1}{u^2(\tilde{y})}+2\sqrt{\epsilon}\eta^{\frac{\beta}{2}}}\|F_1-F_2\|.
         \end{aligned}
     \end{equation*}
     For small enough $\epsilon$, $\frac{\epsilon}{u^2(\tilde{y})}\geq 2\sqrt{\epsilon}\eta^{\frac{\beta}{2}}$, thus part(1) is proved.

     (2) For $\tilde{v}\in B_{e^{\frac{1}{s^2(\hat{f}^{-1}(\tilde{y}))}-\frac{1}{2}\eta^{\frac{\beta}{2}}}q^s}^{d_s(\tilde{y})}(0)$, and let $\tilde{v}=\phi_1(v)$:
\begin{equation*}
    \begin{aligned}
        \|d_{\tilde{v}}\phi_1^{-1}-d_{\tilde{v}}\phi_2^{-1}\|&\leq \|Id-d_{\tilde{v}}\phi_2^{-1}\cdot d_v\phi_1\|= \|d_{\tilde{v}}\phi_2^{-1}\cdot d_v\phi_2-d_{\tilde{v}}\phi_2^{-1}\cdot d_v\phi_1\|\\
        &\leq \|d_v\phi_1-d_v\phi_2\|\cdot \|d_{\tilde{v}}\phi_2^{-1}\|\\
        &\leq \Big\|d_{(v,F_1(v))}h_s^-\begin{pmatrix}
            I\\
            d_vF_1
        \end{pmatrix}-d_{(v,F_2(v))}h_s^-\begin{pmatrix}
            I\\
            d_vF_2
        \end{pmatrix}\Big\|\cdot \|d_{\tilde{v}}\phi_2^{-1}\|\\
        &\leq \Big[\Big\|d_{(v,F_1(v))}h_s^-\begin{pmatrix}
            I\\
            d_vF_1
        \end{pmatrix} -d_{(v,F_1(v))}h_s^-\begin{pmatrix}
            I\\
            d_vF_2
        \end{pmatrix}\Big\|\\
        &+ \Big\|d_{(v,F_1(v))}h_s^-\begin{pmatrix}
            I\\
            d_vF_2
        \end{pmatrix}-d_{(v,F_2(v))}h_s^-\begin{pmatrix}
            I\\
            d_vF_2
        \end{pmatrix}\Big\|\Big]\cdot \|d_{\tilde{v}}\phi_2^{-1}\|\\
        &\leq [\|d_{\cdot}h_s^-\|\cdot\|d_{\cdot}F_1-d_{\cdot}F_2\|+Hol_{\frac{\beta}{2}}(d_{\cdot}h_s^-)\|F_1-F_2\|^{\frac{\beta}{2}}(1+\|d_{\cdot}F_2\|)]\cdot \|d_{\tilde{v}}\phi_2^{-1}\|,
    \end{aligned}
\end{equation*}
 then
 \begin{equation}\label{phi3}
     \|d_{\tilde{v}}\phi_1^{-1}-d_{\tilde{v}}\phi_2^{-1}\|\leq e^{(q^s)^{\frac{\beta}{2}}-\frac{1}{s^2(\hat{f}^{-1}(\tilde{y}))}}(\sqrt{\epsilon}\eta^{\frac{\beta}{2}}\|d_{\cdot}F_1-d_{\cdot}F_2\|+2\sqrt{\epsilon}\|F_1-F_2\|^{\frac{\beta}{2}}).
 \end{equation}
 Denote $v_1:=\phi_1^{-1}(\tilde{v})$ and $v_2:=\phi_2^{-1}(\tilde{v})$, we estimate
 \begin{align*}
    &\quad \|d_{(v_1,F_1(v_1))}h_u^--d_{(v_2,F_2(v_2))}h_u^-\|\\
    &\leq \|d_{(v_1,F_1(v_1))}h_u^--d_{(v_1,F_2(v_1))}h_u^-\|+\|d_{(v_1,F_2(v_1))}h_u^--d_{(v_1,F_2(v_2))}h_u^-\|+\|d_{(v_1,F_2(v_2))}h_u^--d_{(v_2,F_2(v_2))}h_u^-\|\\
    &\leq \sqrt{\epsilon}\|F_1-F_2\|^{\frac{\beta}{2}}+\sqrt{\epsilon}|v_1-v_2|^{\frac{\beta}{2}}+\sqrt{\epsilon}|v_1-v_2|^{\frac{\beta}{2}}\leq3\sqrt{\epsilon}\|F_1-F_2\|^{\frac{\beta}{2}},
    \end{align*}
since $|v_1-v_2|=|\phi_1^{-1}(\tilde{v})-\phi_2^{-1}(\tilde{v})|\leq \|F_1-F_2\|$ by part (1). Then
 \begin{equation*}
     \begin{aligned}
         &\quad \Big\|d_{(v_1,F_1(v_1))}h_u^-\begin{pmatrix}
            I\\
            d_{v_1}F_1
        \end{pmatrix}-d_{(v_2,F_2(v_2))}h_u^-\begin{pmatrix}
            I\\
            d_{v_2}F_2
        \end{pmatrix}\Big\|\\
 &\leq \|d_{(v_1,F_1(v_1))}h_u^--d_{(v_2,F_2(v_2))}h_u^-\|\cdot\Big\|\begin{pmatrix}
            I\\
            d_{v_1}F_1
        \end{pmatrix}\Big\|+ \|d_{\cdot}h_u^-\|\cdot\Big\|\begin{pmatrix}
            I\\
            d_{v_1}F_1
        \end{pmatrix}-\begin{pmatrix}
            I\\
            d_{v_2}F_2
        \end{pmatrix}\Big\|\\
        &\leq Hol_{\frac{\beta}{2}}(h_u^-)3\sqrt{\epsilon}\|F_1-F_2\|^{\frac{\beta}{2}}(1+\|d_{\cdot}F_1\|)+\|d_{\cdot}h_u^-\|\cdot\|d_{v_1}F_1-d_{v_2}F_2\|\\
        &\leq 6\epsilon\|F_1-F_2\|^{\frac{\beta}{2}}+\sqrt{\epsilon}\eta^{\frac{\beta}{2}}(\|d_{v_1}F_1-d_{v_2}F_1\|+\|d_{v_2}F_1-d_{v_2}F_2\|)\\
        &\leq  6\epsilon\|F_1-F_2\|^{\frac{\beta}{2}}+\sqrt{\epsilon}\eta^{\frac{\beta}{2}}\Big(\frac{1}{2}|v_1-v_2|^{\frac{\beta}{2}}+\|d_{\cdot}F_1-d_{\cdot}F_2\|\Big)\\
        &\leq 7\epsilon\|F_1-F_2\|^{\frac{\beta}{2}}+\sqrt{\epsilon}\eta^{\frac{\beta}{2}}\|d_{\cdot}F_1-d_{\cdot}F_2\|.
     \end{aligned}
 \end{equation*}
It follows that
       \begin{equation*}
           \begin{aligned}
               &\quad\Big\|d_{(v_1,F_1(v_1))}h_u^-\begin{pmatrix}
            I\\
            d_{v_1}F_1
        \end{pmatrix}d_{\tilde{v}}\phi_1^{-1}-d_{(v_2,F_2(v_2))}h_u^-\begin{pmatrix}
            I\\
            d_{v_2}F_2
        \end{pmatrix}d_{\tilde{v}}\phi_2^{-1}\Big\|\\
        &\leq \Big\|d_{(v_1,F_1(v_1))}h_u^-\begin{pmatrix}
            I\\
            d_{v_1}F_1
        \end{pmatrix}-d_{(v_2,F_2(v_2))}h_u^-\begin{pmatrix}
            I\\
            d_{v_2}F_2
        \end{pmatrix}\Big\|\cdot\|d_{\tilde{v}}\phi_1^{-1}\|\\
        &+\Big\|d_{(v_2,F_2(v_2))}h_u^-\begin{pmatrix}
            I\\
            d_{v_2}F_2
        \end{pmatrix}\Big\|\cdot\|d_{\tilde{v}}\phi_1^{-1}-d_{\tilde{v}}\phi_2^{-1}\|\\
        &\leq e^{(q^s)^{\frac{\beta}{2}}-\frac{1}{s^2(\hat{f}^{-1}(\tilde{y}))}}( 7\epsilon\|F_1-F_2\|^{\frac{\beta}{2}}+\sqrt{\epsilon}\eta^{\frac{\beta}{2}}\|d_{\cdot}F_1-d_{\cdot}F_2\|)\\
        &+ 2\sqrt{\epsilon}\eta^{\frac{\beta}{2}}e^{(q^s)^{\frac{\beta}{2}}-\frac{1}{s^2(\hat{f}^{-1}(\tilde{y}))}}(\sqrt{\epsilon}\eta^{\frac{\beta}{2}}\|d_{\cdot}F_1-d_{\cdot}F_2\|+2\sqrt{\epsilon}\|F_1-F_2\|^{\frac{\beta}{2}})\\
        &\leq e^{(q^s)^{\frac{\beta}{2}}-\frac{1}{s^2(\hat{f}^{-1}(\tilde{y}))}}\big[(\sqrt{\epsilon}\eta^{\frac{\beta}{2}}+2\epsilon\eta^\beta)\|d_{\cdot}F_1-d_{\cdot}F_2\|+(7\epsilon+4\epsilon\eta^{\frac{\beta}{2}})\|F_1-F_2\|^{\frac{\beta}{2}}\big].
           \end{aligned}
       \end{equation*}
And by claim \ref{phi}\ref{a} and inequality (\ref{phi3}) , we have
\begin{equation*}
    \begin{aligned}
        &\quad \|d_{v_1}F_1\cdot d_{\tilde{v}}\phi_1^{-1}-d_{v_2}F_2\cdot d_{\tilde{v}}\phi_2^{-1}\|\leq \|d_{\cdot}F_1-d_{\cdot}F_2\|\cdot\|d_{\cdot}\phi_1^{-1}\|+\|d_{\cdot}F_2\|\cdot\|d_{\cdot}\phi_1^{-1}-d_{\cdot}\phi_2^{-1}\|\\
        &\leq e^{(q^s)^{\frac{\beta}{2}}-\frac{1}{s^2(\hat{f}^{-1}(\tilde{y}))}}\|d_{\cdot}F_1-d_{\cdot}F_2\|+\epsilon e^{(q^s)^{\frac{\beta}{2}}-\frac{1}{s^2(\hat{f}^{-1}(\tilde{y}))}}(\sqrt{\epsilon}\eta^{\frac{\beta}{2}}\|d_{\cdot}F_1-d_{\cdot}F_2\|+2\sqrt{\epsilon}\|F_1-F_2\|^{\frac{\beta}{2}})\\
        &\leq e^{(q^s)^{\frac{\beta}{2}}-\frac{1}{s^2(\hat{f}^{-1}(\tilde{y}))}}[(1+\epsilon\eta^{\frac{\beta}{2}})\|d_{\cdot}F_1-d_{\cdot}F_2\|+2\epsilon\|F_1-F_2\|^{\frac{\beta}{2}}].
    \end{aligned}
\end{equation*}
Hence,
\begin{equation*}
         \begin{aligned}
             \|d_{\tilde{v}}\tilde{F}_1-d_{\tilde{v}}\tilde{F}_2\|&=\Big\|D_u^{-1}\cdot d_{v_1}F_1\cdot d_{\tilde{v}}\phi_1^{-1}+d_{(v_1,F_1(v_1))}h_u^-\begin{pmatrix}
            I\\
            d_{v_1}F_1
        \end{pmatrix}d_{\tilde{v}}\phi_1^{-1}\\
        &- D_u^{-1}\cdot d_{v_2}F_2\cdot d_{\tilde{v}}\phi_2^{-1}-d_{(v_2,F_2(v_2))}h_u^-\begin{pmatrix}
            I\\
            d_{v_2}F_2
        \end{pmatrix}d_{\tilde{v}}\phi_2^{-1}\Big\|\\
         &\leq \|D_u^{-1}\|e^{(q^s)^{\frac{\beta}{2}}-\frac{1}{s^2(\hat{f}^{-1}(\tilde{y}))}}\big[(1+\epsilon\eta^{\frac{\beta}{2}})\|d_{\cdot}F_1-d_{\cdot}F_2\|+2\epsilon\|F_1-F_2\|^{\frac{\beta}{2}}\big]\\
         &+e^{(q^s)^{\frac{\beta}{2}}-\frac{1}{s^2(\hat{f}^{-1}(\tilde{y}))}}\big[(\sqrt{\epsilon}\eta^{\frac{\beta}{2}}+2\epsilon\eta^\beta)\|d_{\cdot}F_1-d_{\cdot}F_2\|+  (7\epsilon+4\epsilon\eta^{\frac{\beta}{2}})\|F_1-F_2\|^{\frac{\beta}{2}}\big]\\
         &\leq e^{(q^s)^{\frac{\beta}{2}}-\frac{1}{s^2(\hat{f}^{-1}(\tilde{y}))}-\frac{1}{u^2(\tilde{y})}}\big((1+\epsilon\eta^{\frac{\beta}{2}})\|d_{\cdot}F_1-d_{\cdot}F_2\|+10\epsilon\|F_1-F_2\|^{\frac{\beta}{2}}\big).
         \end{aligned}
     \end{equation*}
    For small enough $\epsilon$, $\frac{\epsilon}{s^2(\hat{f}^{-1}(\tilde{y}))}+\frac{\epsilon}{u^2(\tilde{y})}\geq 2(q^s)^{\frac{\beta}{2}}$, thus $\|d_{\cdot}\tilde{F}_1-d_{\cdot}\tilde{F}_2\|\leq e^{-\frac{1-\epsilon}{s^2(\hat{f}^{-1}(\tilde{y}))}-\frac{1-\epsilon}{u^2(\tilde{y})}}\big((1+\epsilon\eta^{\frac{\beta}{2}})\|d_{\cdot}F_1-d_{\cdot}F_2\|+10\epsilon\|F_1-F_2\|^{\frac{\beta}{2}}\big)$.

    Similarly, we can get the $u$-case.
 \end{pf}

In this context, we study the orbits which may have $0$ Lyapunov exponents, so the coding map may not be $H\ddot{o}lder$ continuous when coding points on these orbits.  Ovadia introduced \textit{leaf cores} by shrinking the size of the admissible manifold, then the coding map has summable variations about the leaf cores.
\begin{mdef}
    Given a double chart $\psi_{\tilde{x}}^{p^s,p^u}$. The \textit{leaf cores} of $s/u$-admissible manifold $V^{s/u}$ is
    $$\hat{V}^{s/u}:=\psi_{\tilde{x}}\circ Graph(F^{s/u})\text{ on }B_{\eta^2}^{d_{s/u}(\tilde{x})}(0),$$
    where $\eta:=p^s\land p^u$, and $F^{s/u}$ is the representing function of the $s/u$-admissible manifold $V^{s/u}$.
\end{mdef}

Notice that $\hat{V}^{s/u}\subset V^{s/u}$ if $\hat{V}^{s/u}$ is the leaf cores of $V^{s/u}$.
 \begin{cor}
      Suppose $w_1\rightarrow w_2$, where $w_1=\psi_{\tilde{x}}^{p^s,p^u}, w_2=\psi_{\tilde{y}}^{q^s,q^u}$. For small enough $\epsilon$, and $s$-admissible manifolds $V_1^{s}, V_2^{s}$ in $\psi_{\tilde{y}}^{q^s,q^u}$ with the representing functions $F_1,F_2$, respectively:
 \begin{enumerate}[label=(\arabic*), series=myenum]
\item $d_{C^0}(\mathcal{F}^{s}_{w_1,w_2}(\hat{V}_1^{s}),\mathcal{F}^{s}_{w_1,w_2}(\hat{V}_1^{s}))\leq e^{-\frac{1-\epsilon}{u^2(\tilde{y}))}}d_{C^0}(\hat{V}_1^{s},\hat{V}_2^{s})$.
\item $d_{C^1}(\mathcal{F}^{s}_{w_1,w_2}(\hat{V}_1^{s}),\mathcal{F}^{s}_{w_1,w_2}(\hat{V}_1^{s}))\leq e^{-\frac{1-\epsilon}{s^2(\hat{f}^{-1}(\tilde{y}))}-\frac{1-\epsilon}{u^2(\tilde{y})}}\big((1+\epsilon(q^s\land q^u)^{\beta})\|d_{\cdot}F_1-d_{\cdot}F_2\|+10\epsilon\|F_1-F_2\|^{\frac{\beta}{2}}\big)$.
\end{enumerate}
For $u$-admissible manifolds $V_1^{u}, V_2^{u}$ in $\psi_{\tilde{x}}^{p^s,p^u}$ with the representing functions $G_1,G_2$, respectively:
\begin{enumerate}[label=(\arabic*), resume=myenum]
    \item $d_{C^0}(\mathcal{F}^{u}_{w_1,w_2}(\hat{V}_1^{u}),\mathcal{F}^{u}_{w_1,w_2}(\hat{V}_1^{u}))\leq e^{-\frac{1-\epsilon}{s^2(\tilde{x}))}}d_{C^0}(\hat{V}_1^{u},\hat{V}_2^{u})$.
    \item $d_{C^1}(\mathcal{F}^u_{w_1,w_2}(\hat{V}_1^{u}),\mathcal{F}^u_{w_1,w_2}(\hat{V}_1^{u}))\leq e^{-\frac{1-\epsilon}{s^2(\tilde{x}))}-\frac{1-\epsilon}{u^2(\hat{f}(\tilde{x}))}}\big((1+\epsilon(p^s\land p^u)^{\beta})\|d_{\cdot}G_1-d_{\cdot}G_2\|+10\epsilon\|G_1-G_2\|^{\frac{\beta}{2}}\big)$.
\end{enumerate}
 \end{cor}
 \begin{pf}
    We restrict the proof of Theorem \ref{TF} on $B_{\eta^2}^{d_s(\tilde{x})}(0)$ where $\eta=q^s\land q^u$, as the calculations of  inequality (\ref{qs}), we have
    $e^{\frac{1}{u^2(\tilde{y})}-\frac{1}{2}(\eta^2)^{\frac{\beta}{2}}}\eta^2\geq e^{2\Gamma(\eta)^{\frac{1}{\gamma}}}\eta^2=I(\eta)^{2}\geq (p^s\land p^u)^2$. Hence,
    $$\mathcal{F}^{s}_{w_1,w_2}(\hat{V}_1^s)=\psi_{\tilde{x}}[\{(v,\tilde{F_i}(v)):v\in B_{(p^s\land p^u)^2}^{d_s(\tilde{x})}(0)\}],$$
    where $\tilde{F}_i=[D_u^{-1}(F_i(\cdot))+h_u^-(\cdot,F_i(\cdot))]\circ \phi_i^{-1}$, $\phi_i(v):=D_s^{-1}(v)+h_s^-(v,F_i(v))$ on $B_{\eta^2}^{d_s(\tilde{y})}(0)$. Then we get the corollary by restricting the proof of Proposition \ref{Fs} on $ B_{\eta^2}^{d_s(\tilde{x})}(0)$.
 \end{pf}
 \begin{mdef}
    Given a double chart $\psi_{\tilde{x}}^{p^s,p^u}$,  and let $\{V_n\}_{n\in\mathbb{N}}$ be a sequence of $s/u$-admissible manifolds with the representing function $F_n$ in $\psi_{\tilde{x}}^{p^s,p^u}$. We say that the sequence $\{V_n\}_{n\in\mathbb{N}}$ uniformly converges to $s/u$-admissible manifold $V$ if the representing functions $F_n$ uniformly converge to the representing function of $V$.
\end{mdef}

If we restrict the domain of the representing functions $F_n$ to $B_{(p^s\land p^u)^2}^{d_{s/u}(\tilde{x})}$, we get the definition of the leaf core $\hat{V}_n$ uniformly converging to the leaf core $\hat{V}$. And notice that the admissible manifolds $\{V_n\}$ uniformly converges to the admissible manifold $V$ implies the corresponding leaf cores $\{\hat{V}_n\}$ uniformly converges to the leaf core $\hat{V}$.
\begin{prop}\label{shadow}
    For small enough $\epsilon>0$  and a chain $(v_i)_{i\in\mathbb{Z}}$, where $v_i=\psi_{\tilde{x}_i}^{p_i^s,p^u_i}$. Choose an  arbitrary sequence $u$-admissible manifolds $V_{-n}^u$ in $v_{-n}$ and a sequence $s$-admissible manifolds $V_{n}^s$ in $v_{n}$, $n\geq 0$. The following statements hold:
 \begin{enumerate}[label=(\arabic*)]
\item  Denote $$V^u((v_i)_{i\leq0}):= \lim_{n\rightarrow\infty}(\mathcal{F}^u)^n(V^u_{-n}),\ V^s((v_i)_{i\geq0}):= \lim_{n\rightarrow\infty}(\mathcal{F}^s)^n(V^s_{n}),$$
    where $(\mathcal{F}^u)^n= \mathcal{F}_{v_{-1},v_0}^u\circ \cdots\circ\mathcal{F}_{v_{-n+1},v_{-n+2}}^u\circ \mathcal{F}_{v_{-n},v_{-n+1}}^u$ and $(\mathcal{F}^s)^n= \mathcal{F}_{v_{0},v_1}^u\circ \cdots\circ\mathcal{F}_{v_{n-2},v_{n-1}}^u\circ \mathcal{F}_{v_{n-1},v_{n}}^u$. $V^u((v_i)_{i\leq0})$, $V^s((v_i)_{i\geq0})$ are well-defined, and are independent of the choice of $V_{-n}^u\in\mathcal{M}^u(v_{-n})$, $V_n^s\in\mathcal{M}^s(v_n),n\geq 0$.
\item $V^u((v_i)_{i\leq0})/V^s((v_i)_{i\geq0})$ is a $u/s$-admissible manifold in $v_0$.
\item (Invariance) $f[V^s((v_i)_{i\geq0})]\subset V^s((v_{i+1})_{i\geq0})$ and $f^{-1}_{\tilde{x}_{-1}}[V^u((v_i)_{i\leq0})]\subset V^u((v_{i-1})_{i\leq0})$.
\item (Shadowing)
    $$V^s((v_i)_{i\geq0})=\{y\in \psi_{\tilde{x}_0}[B_{p^s_0}(0)]: f^n(y)\in\psi_{\tilde{x}_n}[B_{10Q(\tilde{x}_n)}(0)],\forall n\geq 0\};$$
    and
    $$V^u((v_i)_{i\leq0})=\{y\in \psi_{\tilde{x}_0}[B_{p^u_0}(0)]: f^{-1}_{\tilde{x}_{-n}}\circ \cdots \circ f^{-1}_{\tilde{x}_{-1}}(y)\in\psi_{\tilde{x}_{-n}}[B_{10Q(\tilde{x}_{-n})}(0)],\forall n\geq 0\}.$$
\item (Summable variations) For all $N\geq 0$, suppose the chain $(w_i)_{i\in\mathbb{Z}}$ of double charts satisfy $v_i=w_i$ for $|i|\leq N$, then there exist constants $C, \alpha(\Gamma,\gamma,\beta)>0$ such that
    $$d_{C^1}(\hat{V}^u((v_i)_{i\leq0}),\hat{V}^u((w_i)_{i\leq0})\leq C(I^{-N}(p_{0}^s\land p_{0}^u))^{\frac{1}{\gamma}+\alpha}.$$
    A similar statement for the map $(v_i)_{i\geq0}\mapsto \hat{V}^s((v_i)_{i\geq 0})$.
\item (Hyperbolicity) If $y,z\in V^s((v_i)_{i\geq 0})$, then $d(f^n(y),f^n (z))\leq 8I^{-n}(p_0^s)$ for all $n\geq 0$; and if $y,z\in V^u((v_i)_{i\leq 0})$, then $d( f^{-1}_{\tilde{x}_n}\circ \cdots \circ f^{-1}_{\tilde{x}_{-1}}(y),f^{-1}_{\tilde{x}_n}\circ \cdots \circ f^{-1}_{\tilde{x}_{-1}}(z))\leq 8I^{n}(p_0^u)$ for all $n\leq 0$.
\end{enumerate}
\end{prop}
\begin{pf}
We prove the results for the $u$-case; a similar prove applies to the $s$-case.

      (1) $\{(\mathcal{F}^u)^n(V^u_{-n})\}_{n\in\mathbb{N}}$ is a Cauchy sequence: let $W^u_{-n}$ be an arbitrary $u$-admissible manifold in $v_{-n}$, for all $n> 0$. By Proposition \ref{Fs}, we get
      \begin{equation}\label{BO}
          \begin{aligned}
              d_{C^0}((\mathcal{F}^u)^n(V^u_{-n}),(\mathcal{F}^u)^n(W^u_{-n})&\leq e^{-\sum_{k=1}^{n}\frac{1-\epsilon}{s^2(\tilde{x}_{-k})}}d(V^u_{-n},W^u_{-n})\\
              &\leq e^{-(1-\epsilon)(\epsilon^{\frac{20}{\beta}})^{-\frac{1}{\gamma}}\sum_{k=1}^{n}(p_{-k+1}^u)^{\frac{1}{\gamma}}}\cdot (2p_{-n}^u)\text{ (by inequality (\ref{s2}))}\\
              &\leq 2 e^{-(1-\epsilon)(\epsilon^{\frac{20}{\beta}})^{-\frac{1}{\gamma}}\sum_{k=0}^{n-1}(p_{-k}^u)^{\frac{1}{\gamma}}}p^u_{0} e^{\Gamma\sum_{i=0}^{n-1}(p^u_{-i})^{\frac{1}{\gamma}}}\\
              &= 2 e^{-[((1-\epsilon)(\epsilon^{\frac{20}{\beta}})^{-\frac{1}{\gamma}}-\Gamma)\cdot\frac{1}{\Gamma}]\Gamma\sum_{k=0}^{n-1}(p_{-k}^u)^{\frac{1}{\gamma}}}p^u_{0}\\
              &\leq 2 e^{-\Gamma\sum_{k=0}^{n-1}(I^{-k}(p^u_{0}))^{\frac{1}{\gamma}}}p^u_{0}\\
              &\leq 2I^{-n}(p^u_{0}).
          \end{aligned}
      \end{equation}
      The second-to-last inequality holds since $\big((1-\epsilon)(\epsilon^{\frac{20}{\beta}})^{-\frac{1}{\gamma}}-\Gamma\big)\cdot\frac{1}{\Gamma}>1$ for small enough $\epsilon$ and $p_{k+1}^u=\min\{I(p_k^u),Q(\tilde{x}_{k+1})\}$. Thus $\forall n,m\in\mathbb{N}(n\geq m)$, we have
      $$d_{C^0}((\mathcal{F}^u)^n(V^u_{-n}),(\mathcal{F}^u)^m(W^u_{-m}))=d_{C^0}((\mathcal{F}^u)^m[(\mathcal{F}^u)^{n-m}(V^u_{-n})],(\mathcal{F}^u)^m(W^u_{-m}))\leq 2I^{-m}(p^u_{0}).$$
      Therefore, $\{(\mathcal{F}^u)^n(V^u_{-n})\}_{n\in\mathbb{N}}$ is a Cauchy sequence by Proposition \ref{I}(6).

      (2) It is equivalent to prove that the representing functions of $V^u((v_i)_{i\leq0})/V^s((v_i)_{i\geq0})$ satisfy \ref{AM1}-\ref{AM3}. And the proof is   the same as the proof of proposition 3.14(2) in \cite{Ovadia2}.

      (3) By definition and Theorem \ref{TF}, we have
      \begin{equation*}
          \begin{aligned}
          &\quad V^u((v_i)_{i\leq0})= \lim_{n\rightarrow\infty}(\mathcal{F}^u)^n(V^u_{-n})=\lim_{n\rightarrow\infty}\mathcal{F}_{v_{-1},v_0}^u\circ \mathcal{F}_{v_{-2,-1}}^u\cdots\circ \mathcal{F}_{v_{-n},v_{-n+1}}^u(V^u_{-n})\\
          &= \mathcal{F}_{v_{-1},v_0}^u(\lim_{n\rightarrow\infty}\mathcal{F}_{v_{-2,-1}}^u\cdots\circ \mathcal{F}_{v_{-n},v_{-n+1}}^u(V^u_{-n}))= \mathcal{F}_{v_{-1},v_0}^u(V^u((v_{i-1})_{i\leq0})\subset f(V^u((v_{i-1})_{i\leq0}).
          \end{aligned}
      \end{equation*}
      Hence, $f^{-1}_{\tilde{x}_{-1}}[V^u((v_i)_{i\leq0})]\subset V^u((v_{i-1})_{i\leq0})$.

      (4) The proof mainly uses the properties of $F_{\tilde{x}_n,\tilde{x}_{n+1}}^{\pm1}$ (Proposition \ref{F1}); for details, see \cite{Ovadia2}  .

      (5) Denote $V_k^u:=V^u((v_{i-k})_{i\leq 0})$ and $W_k^u:=V^u((w_{i-k})_{i\leq 0})$, the representing functions are $F_k$ and $G_k \forall 0\leq k\leq N$, respectively. According to the estimate regarding the inequality (\ref{BO}), we have a similar calculation for the leaf cores $\hat{V}^u_k$ and $\hat{W}^u_k$:
\begin{equation}\label{FV}
\begin{aligned}
d_{C^0}(\hat{V}^u_k,\hat{W}^u_k)
&=d_{C^0}((\mathcal{F}^u)^{N-k}(\hat{V}^u_{-N}),(\mathcal{F}^u)^{N-k}(\hat{W}^u_{-N})
\leq e^{-\sum_{i=k+1}^{N}\frac{1-\epsilon}{s^2(\tilde{x}_{-i})}}d(\hat{V}^u_{-N},\hat{W}^u_{-N})\\
&\leq e^{-(1-\epsilon)(\epsilon^{\frac{20}{\beta}})^{-\frac{1}{\gamma}}\sum_{i=k+1}^{N}(p_{-i+1}^u\land p_{-i+1}^s)^{\frac{1}{\gamma}}}\cdot 2(p_{-N}^u\land p_{-N}^s)^2\text{ (by inequality (\ref{s2}))}\\    &\leq 2 e^{-(1-\epsilon)(\epsilon^{\frac{20}{\beta}})^{-\frac{1}{\gamma}}\sum_{i=k}^{N-1}(p_{-i}^u\land p_{-i}^s)^{\frac{1}{\gamma}}}\cdot(p^u_{-k}\land p^s_{-k})^2 e^{2\Gamma\sum_{i=k}^{N-1}(p^u_{-i}\land p^s_{-i})^{\frac{1}{\gamma}}}\\            &= 2 e^{-[((1-\epsilon)(\epsilon^{\frac{20}{\beta}})^{-\frac{1}{\gamma}}-2\Gamma)\frac{1}{2\Gamma}]2\Gamma\sum_{i=k}^{N-1}(p_{-i}^u\land p_{-i}^s)^{\frac{1}{\gamma}}}\cdot(p^u_{-k}\land p^s_{-k})^2 \\
&\leq 2 e^{-2\Gamma\sum_{i=k}^{N-1}(p_{-i}^u\land p_{-i}^s)^{\frac{1}{\gamma}}}(p^u_{-k}\land p^s_{-k})^2 \text{ (for small enough }\epsilon) \\
&\leq 2 \big(e^{-\Gamma\sum_{i=0}^{N-k-1}[I^{-i}(p_{-k}^u\land p_{-k}^s)]^{\frac{1}{\gamma}}}(p^u_{-k}\land p^s_{-k})\big)^2.
\end{aligned}
\end{equation}
For any $0\leq k\leq N$ and $t\in B_{(p^s_{-k}\land p^u_{-k})^2}^{d_u(\tilde{x}_{-k})}(0)$, $\|d_{t}F_k\|\leq (p^s_{-k}\land p^u_{-k})^{\beta}$ and $\|d_{t}G_k\| \leq (p^s_{-k}\land p^u_{-k})^{\beta}$ by inequality (\ref{G1}). According to the Proposition \ref{Fs}(4), we have
\begin{equation*}
    \begin{aligned}
      \|d_{\cdot}F_0-d_{\cdot}G_0\|&\leq 2e^{-\sum_{k=1}^{N}(\frac{1-\epsilon}{s^2(\tilde{x}_{-k})}+\frac{1-\epsilon}{u^2(\tilde{x}_{-k+1})})}\|d_{\cdot}F_N-d_{\cdot}G_N\|+10\epsilon\sum_{k=1}^Ne^{-\sum_{i=1}^{k}(\frac{1-\epsilon}{s^2(\tilde{x}_{-i})}+\frac{1-\epsilon}{u^2(\tilde{x}_{-i+1})})}\|F_k-G_k\|^{\frac{\beta}{2}}\\
      &:=\uppercase\expandafter{\romannumeral1}' + \uppercase\expandafter{\romannumeral2}',
    \end{aligned}
\end{equation*}
where
\begin{equation*}
\begin{aligned}
\uppercase\expandafter{\romannumeral1}'&\leq 4(p^s_{-N}\land p^u_{-N})^{\beta} e^{-\sum_{k=1}^{N}(\frac{1-\epsilon}{s^2(\tilde{x}_{-k})}+\frac{1-\epsilon}{u^2(\tilde{x}_{-k+1})})}\\
&\leq 4(p_{-k}^s\land p_{-k}^u)^{\beta}\Big[e^{\beta\sum_{k=0}^{N-1}\Gamma(p_{0}^s\land p^u_{0})^{\frac{1}{\gamma}}}e^{-\sum_{k=1}^{N}\frac{1-\epsilon}{u^2(\tilde{x}_{-k+1})}}\Big]e^{-\sum_{k=1}^{N}\frac{1-\epsilon}{s^2(\tilde{x}_{-k})}}\\
&\leq 4(p_{0}^s\land p_{0}^u)^{\beta}e^{-\beta\sum_{k=0}^{N-1}\Gamma(p_{-k}^s\land p^u_{-k})^{\frac{1}{\gamma}}}\text{ (by inequality (\ref{s2}))}\\
&\leq 4(p_{0}^s\land p_{0}^u)^{\beta}e^{-\beta\sum_{k=0}^{N-1}\Gamma[I^{-k}(p_{0}^s\land p^u_{0})]^{\frac{1}{\gamma}}}\\
&\leq 4(I^{-N}(p_{0}^s\land p_{0}^u))^{\beta}.
\end{aligned}
\end{equation*}
and
\begin{equation*}
    \begin{aligned}
\uppercase\expandafter{\romannumeral2}'
&\overset{(a)}{\leq} 2^{\frac{\beta}{2}}\cdot 10\epsilon\sum_{k=1}^Ne^{-\sum_{i=1}^{k}(\frac{1-\epsilon}{s^2(\tilde{x}_{-i})}+\frac{1-\epsilon}{u^2(\tilde{x}_{-i+1})})}\cdot \Big[e^{-\Gamma\sum_{i=0}^{N-k-1}[I^{-i}(p_{-k}^u\land p_{-k}^s)]^{\frac{1}{\gamma}}}(p^u_{-k}\land p^s_{-k})\Big]^\beta\\
&\leq 2^{\frac{\beta}{2}}\cdot 10\epsilon\sum_{k=1}^N(p_{0}^s\land p_{0}^u)^\beta\Big[e^{\beta\sum_{i=0}^{k-1}\Gamma(p_{-i}^s\land p^u_{-i})^{\frac{1}{\gamma}}}e^{-\sum_{i=1}^{k}\frac{1-\epsilon}{u^2(\tilde{x}_{-i+1})}}\Big]\\
&\quad \cdot e^{-\sum_{i=1}^{k}\frac{1-\epsilon}{s^2(\tilde{x}_{-i})}}\cdot e^{-\beta\Gamma\sum_{i=0}^{N-k-1}[I^{-i}(p_{-k}^u\land p_{-k)}^s)]^{\frac{1}{\gamma}}}\\
&\leq 2^{\frac{\beta}{2}}\cdot 10\epsilon\sum_{k=1}^N(p_{0}^s\land p_{0}^u)^{\beta}e^{-\beta\sum_{i=0}^{k-1}\Gamma(p_{-i}^s\land p^u_{-i})^{\frac{1}{\gamma}}}\cdot e^{-\beta\Gamma\sum_{i=0}^{N-k-1}[I^{-i}(p_{-k}^u\land p_{-k}^s)]^{\frac{1}{\gamma}}}\\
&\leq 2^{\frac{\beta}{2}}\cdot 10\epsilon\sum_{k=1}^N(p_{0}^s\land p_{0}^u)^{\beta}e^{-\beta\sum_{i=0}^{N-1 }\Gamma[I^{-i}(p_{0}^s\land p^u_{0})]^{\frac{1}{\gamma}}}\\
&\leq 2^{\frac{\beta}{2}}\cdot 10\epsilon\sum_{k=1}^N (I^{-N}(p_{0}^s\land p_{0}^u))^\beta\leq 2N(I^{-N}(p_{0}^s\land p_{0}^u))^\beta,
\end{aligned}
\end{equation*}
The inequality (a) holds by the inequality (\ref{FV}). Let $\alpha=\frac{\beta}{2}-\frac{1}{\gamma}>0\ (\gamma>\frac{20}{\beta})$ and $C=\sup_{N\geq 1}2(N+1)(I^{-N}(1))^{\frac{1}{\gamma}+\alpha}$ ($C$ is bounded since $(I^{-N}(1)^{\frac{1}{\gamma}}=O(\frac{1}{N+\delta})$ for some $\delta>0$), then
$$\|d_{\cdot}F_0-d_{\cdot}G_0\|\leq 2(N+1)(I^{-N}(1))^{\frac{1}{\gamma}+\alpha}(I^{-N}(p_{0}^s\land p_{0}^u))^{\frac{1}{\gamma}+\alpha}\leq C(I^{-N}(p_{0}^s\land p_{0}^u))^{\frac{1}{\gamma}+\alpha}.$$

(6) By Theorem \ref{shadow}(3), there are $u$-admissible manifolds $V^u((v_{i+n})_{i\leq0})$ at $v_n$ with representing function $G_n$ s.t. $f^{-1}_{\tilde{x}_n}\circ \cdots \circ f^{-1}_{\tilde{x}_{-1}}(V^u((v_i)_{i\leq 0}))\subset V^u((v_{i+n})_{i\leq0})$ ($n\leq 0$). Thus $f^{-1}_{\tilde{x}_n}\circ \cdots \circ f^{-1}_{\tilde{x}_{-1}}(y), f^{-1}_{\tilde{x}_n}\circ \cdots \circ f^{-1}_{\tilde{x}_{-1}}(z)$ can be written as $\psi_{\tilde{x}_n}(y_n'),\psi_{\tilde{x}_n}(z_n')$ where $y_n'=(G_n(y_n),y_n)$ and $z_n'=(G_n(z_n),z_n)$. For all $n\leq0$, $y_{n-1}'=F_{\tilde{x}_{n-1},\tilde{x}_{n}}^{-1}(y_n')$ and $z_{n-1}'=F_{\tilde{x}_{n-1},\tilde{x}_{n}}^{-1}(z_n')$, then
\begin{equation}\label{s}
    \begin{aligned}
        |y_{n-1}-z_{n-1}|&\leq \|D_{u,n-1}^{-1}\|\cdot|y_n-z_n|+|h_{u,n}(G_n(y_n),y_n)-h_{u,n}(G_n(z_n),z_n)|\\
        &\leq e^{-\frac{1}{u^2(\tilde{x}_{n})}}|y_n-z_n|+\sqrt{\epsilon}\eta_n^{\frac{\beta}{2}}(1+\epsilon)|y_n-z_n|\\
        &\leq e^{-\frac{1-\epsilon}{u^2(\tilde{x}_{n})}}|y_n-z_n|\leq \cdots \leq e^{-\sum_{i=0}^{n}\frac{1-\epsilon}{u^2(\tilde{x}_{-k})}}|y_0-z_0|\\
        &\leq e^{-[(1-\epsilon)(\epsilon^{\frac{20}{\beta}})^{\frac{1}{\gamma}}\frac{1}{\Gamma}]\Gamma\sum_{k=0}^{n}(p^u_{-k})^{-\frac{1}{\gamma}}}|y_0-z_0|\\
        &\leq 2e^{-\Gamma\sum_{k=0}^n(I^{k}(p^u_{0}))^{\frac{1}{\gamma}}}p^u_{0}\text{ ( for small enough } \epsilon) \\
        &\leq 2I^{(n-1)}(p_{0}^u).
    \end{aligned}
\end{equation}
Thus
$$|y_{n-1}'-z_{n-1}'|\leq (1+\epsilon)|y_{n-1}-z_{n-1}|\leq 4I^{(n-1)}(p_0^u) \text{ for all } n\leq 0,$$
and so $|f^n(y)-f^n(z)|\leq \|\psi_{\tilde{x}_n}\|\cdot |y_n'-z_n'|\leq 8I^{-n}(p_0^u)$.
\end{pf}

 Our purpose is to construct Markov partition for    the subset $RST\subset M^f$. According to the idea of Bowen to construct Markov partition for Axiom A \cite{Bowen}, in the following content, we will give the corresponding definitions and properties for $(\hat{f},M^{f})$. Given an $I$-chain $(v_i)_{i\in\mathbb{Z}}$, where $v_i=\psi_{\tilde{x}_i}^{p_i^s,p_i^u}$.
\begin{mdef}
    The \textit{stable set} of $(v_i)_{i\geq 0}$ is
$$\tilde{V}^s((v_i)_{i\geq 0}):=\{\tilde{y}=(y_n)_{n\in\mathbb{Z}}\in M^f:y_0\in V^s((v_i)_{i\geq 0})\}.$$
The \textit{unstable set} of $(v_i)_{i\leq 0}$ is
$$\tilde{V}^u((v_i)_{i\leq 0}):=\{\tilde{y}=(y_n)_{n\in\mathbb{Z}}\in M^f:y_0\in V^u((v_i)_{i\leq 0}) \text{ and } y_n=f^{-1}_{\tilde{x}_n}(y_{n+1})\text{ for all }n<0\}.$$
\end{mdef}

By Proposition \ref{shadow}(4), there is an equivalent definition of stable/unstable set:
$$\tilde{V}^s((v_i)_{i\geq 0})=\{\tilde{y}=(y_n)_{n\in\mathbb{Z}}\in M^f:y_0\in \psi_{\tilde{x}_0}[B_{p_0^s}(0)]\text{ and }y_n\in \psi_{\tilde{x}_n}[B_{10Q(\tilde{x}_n)}(0)], \forall n\geq 0\},$$
and
$$\tilde{V}^u((v_i)_{i\leq 0})=\{\tilde{y}=(y_n)_{n\in\mathbb{Z}}\in M^f:y_0\in \psi_{\tilde{x}_0}[B_{p_0^u}(0)] \text{ and } y_n=f^{-1}_{\tilde{x}_n}(y_{n+1}) \in \psi_{\tilde{x}_n}[B_{10Q(\tilde{x}_n)}(0)], \forall n<0\}.$$

Given a double chart $v$, if a stable/unstable set $\tilde{V}^s((v_i)_{i\geq0})$/$\tilde{V}^u((v_i)_{i\leq0})$ with $v_0=v$, they are called a \textit{stable/unstable set at $v$}, respectively. Denote $\tilde{M}^{s/u}(v)$ as the set of all stable/unstable sets at $v$. Then the transform graph for stable/unstable sets can be defined by
$$\tilde{\mathcal{F}}_{v_i,v_{i+1}}^s:\tilde{M}^s(v_{i+1})\rightarrow\tilde{M}^s(v_i), \tilde{V}^s(v_{i+1+k})_{k\geq 0})\mapsto \tilde{V}^s(v_{i+k})_{k\geq 0});$$
and
$$\tilde{\mathcal{F}}_{v_{i-1},v_{i}}^s:\tilde{M}^s(v_{i-1})\rightarrow\tilde{M}^s(v_i), \tilde{V}^u(v_{i-1+k})_{k\leq 0})\mapsto \tilde{V}^u(v_{i+k})_{k\leq 0}).$$

Similar to the properties for $s/u$-admissible manifolds, there are some similar properties for stable/unstable sets:
\begin{prop}\label{tilde}
    For small enough $\epsilon>0$,
\begin{enumerate}[label=(\arabic*)]
    \item For any stable set $\tilde{V}^s\in \tilde{M}^s(v)$ and unstable set $\tilde{V}^u\in\tilde{M}^u(v)$ where $v$ is a double chart, there is a unique point $\tilde{p}\in M^f$ s.t. $\tilde{V}^s\bigcap\tilde{V}^u=\{\tilde{p}\}$.
    \item Given double charts $v$, $w$ s.t. $v\rightarrow w$. If $\tilde{V}^s\in\tilde{M}^s(w)$ and $\tilde{V}^u\in\tilde{M}^u(v)$, then $\tilde{\mathcal{F}}^s_{v,w}[\tilde{V}^s]\subset \hat{f}^{-1}(\tilde{V}^s)$, $\tilde{\mathcal{F}}^u_{v,w}[\tilde{V}^u]\subset \hat{f}(\tilde{V}^u)$.
    \item $\hat{f}[\tilde{V}^s((v_i)_{i\geq 0})]\subset \tilde{V}^s((v_i)_{i\geq1})$ and $\hat{f}^{-1}[\tilde{V}^u((v_i)_{i\leq0})]\subset \tilde{V}^u((v_i)_{i\leq-1})$.
    \item If $\tilde{y}=(y_k)_{k\in\mathbb{Z}}$, $\tilde{z}=(z_k)_{k\in\mathbb{Z}}\in \tilde{V}^s((v_i)_{i\geq0})$ where $v_i=\psi_{\tilde{x}_i}^{p^s_i,p^u_i},\ \forall i\geq0$, then $d(\hat{f}^n(\tilde{y}),\hat{f}^n(\tilde{z}))\leq 8I^{-n}(p_0^s)$ for all $n\geq 0$; if $\tilde{y}=(y_k)_{k\in\mathbb{Z}}$, $\tilde{z}=(z_k)_{k\in\mathbb{Z}}\in \tilde{V}^u((v_i)_{i\leq0})$ where $v_i=\psi_{\tilde{x}_i}^{p^s_i,p^u_i},\ \forall i\leq0$, then $d(\hat{f}^n(\tilde{y}),\hat{f}^n(\tilde{z}))\leq 8I^{n}(p_0^u)$ for all $n\leq 0$.
    \item For small enough $\epsilon>0$, given stable/unstable sets $\tilde{V}^{s/u}$, $\tilde{W}^{s/u}$ at $\psi_{\tilde{x}}^{p^s,p^u}$, $\psi_{\tilde{y}}^{q^s,q^u}$, respectively. If $\tilde{x}=\tilde{y}$, then $\tilde{V}^{s/u}$ and $\tilde{W}^{s/u}$ are disjoint or one contains the other.
    \item Given double charts $v$, $w$ s.t. $v\rightarrow w$. For unstable set $\tilde{V}^u\in\tilde{M}^u(v)$, $\hat{f}(\tilde{V}^u)$ intersects any stable set in $\tilde{M}^s(w)$ at a unique point; for any stable set $\tilde{W}^s\in\tilde{M}^s(w)$, $\hat{f}^{-1}(\tilde{W}^s)$ intersects any unstable set in $\tilde{M}^u(v)$ at a unique point.
\end{enumerate}
\end{prop}
\begin{pf}
    According to Proposition \ref{ad}(1), Theorem \ref{TF}(1), Proposition \ref{Fs}(2) and the relationship between $s/u$-admissible manifolds and stable/unstable sets, the parts (1)-(3) can be proved easily.

    (4) Recall that the distance $d(\tilde{x},\tilde{y})=\sup\{2^{i}d(x_i,y_i):i\leq 0\}$ for $\tilde{x}=(x_i)_{i\in\mathbb{Z}},\tilde{y}=(y_i)_{i\in\mathbb{Z}}\in M^f$. If $\tilde{y}=(y_k)_{k\in\mathbb{Z}}$, $\tilde{z}=(z_k)_{k\in\mathbb{Z}}\in \tilde{V}^s((v_i)_{i\geq0})$, for every $k\leq 0 $ and $n>0$:
\begin{enumerate}[label=(\roman*)]
    \item $0\leq n\leq -k\ (k+n\leq 0)$: $2^kd(y_{k+n},z_{k+n})\leq 2^k\leq 2^{-n}$;
    \item $n> -k\ (k+n> 0)$: by Proposition \ref{I}, we have
\end{enumerate}
\begin{equation}\label{symbol}
   \begin{aligned}
        I^{-n}(p_0^s)&=I^{-1}(I^{-n+1}(p_0^s))\geq I^{-n+1}(p_0^s)e^{-\Gamma(I^{-n+1}(p_0^s))^{\frac{1}{\gamma}}}\\
        &\geq \cdots \geq I^{-n-k}(p_0^s)e^{-\sum_{i=n+k}^{n-1}\Gamma(I^{-i}(p_0^s))^{\frac{1}{\gamma}}}\geq 2^kI^{-n-k}(p_0^s),
   \end{aligned}
\end{equation}
since $e^{-\Gamma(I^{-i}(p_0^s))^{\frac{1}{\gamma}}}\geq e^{-\Gamma (p_0^s)^{\frac{1}{\gamma}}}>\frac{1}{2}\forall i\geq 0 $ holds for small enough $\epsilon$. Then
$2^kd(y_{k+n},z_{k+n})\leq 2^k\cdot 8I^{-(n+k)}(p_0^s)\leq 8I^{-n}(p_0^s)$ by Proposition \ref{shadow}(6).
Hence,
   \begin{equation*}
       \begin{aligned} d(\hat{f}^n(\tilde{y}),\hat{f}^n(\tilde{z}))&=d((y_{k+n})_{k\in\mathbb{Z}},(z_{k+n})_{k\in\mathbb{Z}})=\sup\{2^{k}d(y_{k+n},z_{k+n}):k\leq0\}\\
    &\leq \max\{2^{-n},8I^{-n}(p_0^s)\}\leq 8I^{-n}(p_0^s),
       \end{aligned}
   \end{equation*}
since $8I^{-n}(p_0^s)\geq e^{-\Gamma\sum_{i=0}^{n-1}(I^{-i}(p_0^s))^{\frac{1}{\gamma}}}$ and $e^{-\Gamma(I^{-i}(p_0^s))^{\frac{1}{\gamma}}}\geq e^{-\Gamma (p_0^s)^{\frac{1}{\gamma}}}>\frac{1}{2}\forall i\geq 0$ hold for small enough $\epsilon$.

  If $\tilde{y}=(y_k)_{k\in\mathbb{Z}}$, $\tilde{z}=(z_k)_{k\in\mathbb{Z}}\in \tilde{V}^u((v_i)_{i\leq0})$, by Proposition \ref{shadow}(6), for every $k\leq 0$ and $n\leq 0$:
  \begin{equation*}
      \begin{aligned}
          d(\hat{f}^{n}(\tilde{y}),\hat{f}^{n}(\tilde{z}))&=d((y_{k+n})_{k\in\mathbb{Z}},(z_{k+n})_{k\in\mathbb{Z}})=\sup\{2^{k}d(y_{k+n},z_{k+n}):k\leq0\}\\
          &\leq \sup\{2^{k}\cdot 8I^{k+n}(p_0^u):k\leq0\} \leq 8I^{n}(p_0^u).
      \end{aligned}
  \end{equation*}

  (5) We prove the result for the $u$-case; a similar   proof applies to the $s$-case.

Let $\tilde{V}^u=\tilde{V}^u((v_i)_{i\leq 0})=\{\tilde{z}=(z_n)_{n\in\mathbb{Z}}\in M^f:z_0\in V^u((v_i)_{i\leq 0}) \text{ and } z_n=f^{-1}_{\tilde{x}_n}(z_{n+1})\text{ for all }n<0\}$  and $\tilde{W}^u=\tilde{W}^u((w_i)_{i\leq 0})=\{\tilde{z}=(z_n)_{n\in\mathbb{Z}}\in M^f:z_0\in W^u((w_i)_{i\leq 0}) \text{ and } z_n=f^{-1}_{\tilde{y}_n}(z_{n+1})\text{ for all }n<0\}$ where $v_i=\psi_{\tilde{x}_i}^{p_i^s,p_i^u}$, $w_i=\psi_{\tilde{y}_i}^{q_i^s,q_i^u}$ and $\tilde{x}_0=\tilde{x}$, $\tilde{y}_0=\tilde{y}$. Assume $\tilde{z}\in\tilde{V}^u\bigcap\tilde{W}^u$ and $p^u_0\leq q^u_0$, then we want to prove $\tilde{V}^u\subset \tilde{W}^u$. We will prove this by showing $f^{-1}_{\tilde{x}_n}$, $f^{-1}_{\tilde{y}_n}$ coincide in the domains of interest for all $n<0$ and $f^{-1}_{\tilde{x}_n}\circ \cdots \circ f^{-1}_{\tilde{x}_{-1}}(V^u)\subset f^{-1}_{\tilde{y}_n}\circ \cdots \circ f^{-1}_{\tilde{y}_{-1}}(W^u)$ for all $|n|$ large enough ($n<0$).

    \begin{claim}
        $f^{-1}_{\tilde{x}_n}\circ \cdots \circ f^{-1}_{\tilde{x}_{-1}}(V^u)\subset \mathfrak{E}_{\tilde{y}_{n-1}}$ and $f^{-1}_{\tilde{x}_{n-1}}\circ \cdots \circ f^{-1}_{\tilde{x}_{-1}}(V^u)\subset f^{-1}_{\tilde{y}_{n-1}}(\mathfrak{E}_{\tilde{y}_{n-1}})$ for all $n\leq 0$
    \end{claim}
    \begin{pf}
Let $z_n\in f^{-1}_{\tilde{x}_n}\circ \cdots \circ f^{-1}_{\tilde{x}_{-1}}(V^u)\bigcap f^{-1}_{\tilde{y}_n}\circ \cdots \circ f^{-1}_{\tilde{y}_{-1}}(W^u)$, for any $t_n\in f^{-1}_{\tilde{x}_n}\circ \cdots \circ f^{-1}_{\tilde{x}_{-1}}(V^u)$, we have $d(z_n,t_n)\leq 8I^{n}(p_0^u)\leq 8I^{n}(q_0^u)\leq 8q^u_{n}$ by Proposition \ref{shadow}(6) and the map $I^{-1}$ is strictly increasing on $(0,e^\Gamma)$. And $z_n\in B(p(\tilde{y}_n),20 Q(\tilde{y}_n)) $ by Proposition \ref{shadow}(4), so
\begin{equation}\label{t}
    t_n\in B(p(\tilde{y}_n),20Q(\tilde{y}_n)+8q^u_{n})\subset B(p_1(\tilde{y}_{n-1}),21Q(\tilde{y}_n)+8q^u_{n}),
\end{equation}
since $d(p_0(\tilde{y}_n),p_1(\tilde{y}_{n-1}))\ll q_n^s\land q_n^u<Q(\tilde{y}_n)$. According to the calculations of inequality (\ref{tao}), we get $Q(\tilde{y}_{n})\leq \epsilon\tau(\tilde{y}_{n-1})$ for all $n\leq 0$. Hence, $f^{-1}_{\tilde{x}_n}\circ \cdots \circ f^{-1}_{\tilde{x}_{-1}}(V^u)\subset B(p_1(\tilde{y}_{n-1}),30 Q(\tilde{y}_n))\subset B(p_1(\tilde{y}_{n-1}),30 \epsilon\tau(\tilde{y}_{n-1}))\subset \mathfrak{E}_{\tilde{y}_{n-1}}$.

For $t_{n-1}\in f^{-1}_{\tilde{x}_{n-1}}\circ f^{-1}_{\tilde{x}_n}\circ \cdots \circ f^{-1}_{\tilde{x}_{-1}}(V^u)$, we get $t_{n-1}\in B(p(\tilde{y}_{n-1}),30Q(\tilde{y}_{n-1}))$ by inequality (\ref{t}). By assumption \ref{A2}, $f^{-1}_{\tilde{y}_{n-1}}(\mathfrak{E}_{\tilde{y}_{n-1}})\supset B(p(\tilde{y}_{n-1}),2d(p(\tilde{y}_{n-1}),\Gamma)^a\tau(\tilde{y}_{n-1}))\supset B(p(\tilde{y}_{n-1}),2\rho(\tilde{y}_{n-1})^{2a})\supset B(p(\tilde{y}_{n-1}),30Q(\tilde{y}_{n-1}))$.
    \end{pf}

 Denote $V^u_n:=V^u((v_{i+n})_{i\leq 0})$ with the representing function $F_n$ and $W^u_n:=W^u((w_{i+n})_{i\leq 0})$ with the representing function $G_n$ ($n\leq 0$). Given $t\in(0,1)$, $I^{n}(t)\rightarrow 0$ as $n\rightarrow \infty$ (by Lemma \ref{I}(6)), then there exists $n_0\leq 0$ s.t.  $f^{-1}_{\tilde{x}_n}\circ \cdots \circ f^{-1}_{\tilde{x}_{-1}}(V^u)\subset \psi_{\tilde{x}_{n}}[B_{\frac{1}{2}Q(\tilde{x}_{n})}(0)]$ for all $n\leq n_0$.
Assume $z_n\in f^{-1}_{\tilde{x}_n}\circ \cdots \circ f^{-1}_{\tilde{x}_{-1}}(V^u)\bigcap f^{-1}_{\tilde{y}_n}\circ \cdots \circ f^{-1}_{\tilde{y}_{-1}}(W^u)$ and $w_n\in f^{-1}_{\tilde{y}_n}\circ \cdots \circ f^{-1}_{\tilde{y}_{-1}}(W^u)$. Notice that $f^{-1}_{\tilde{x}_n}$ and $f^{-1}_{\tilde{y}_n}$ coincide in the domains of interest for all $n<0$. By the proof of Proposition \ref{shadow}(6), $ d(\psi_{\tilde{x}_n}^{-1}(z_n),\psi_{\tilde{x}_n}^{-1}(w_n))\leq 4I^{n}(q_0^u)\leq \frac{1}{2}Q(\tilde{x}_n)$ for large enough $|n|$. Hence, $$|\psi_{\tilde{x}_n}^{-1}(w_n)|\leq |\psi_{\tilde{x}_n}^{-1}(z_n)|+\frac{1}{2}Q(\tilde{x}_n)\leq Q(\tilde{x}_n). $$
i.e. there exists $n_0'<0$ s.t.  $f^{-1}_{\tilde{x}_n}\circ \cdots \circ f^{-1}_{\tilde{x}_{-1}}(W^u)\subset \psi_{\tilde{x}_{n}}[B_{Q(\tilde{x}_{n})}(0)]$ for all $n\leq n_0'$ and small enough $\epsilon$.

Next, we will show that $\tilde{V}^u\subset \tilde{W}^u$ if $p_0^u\leq q_0^u$ and $\tilde{x}_0=\tilde{y}_0$. Since $f^{-1}_{\tilde{x}_n}\circ \cdots \circ f^{-1}_{\tilde{x}_{-1}}(V^u),f^{-1}_{\tilde{x}_n}\circ \cdots \circ f^{-1}_{\tilde{x}_{-1}}(W^u)\subset \psi_{\tilde{x}_{n}}[B_{Q(\tilde{x}_{n})}(0)]$ for all $|n|\geq \max\{|n_0|,|n_0'|\}$, then $f^{-1}_{\tilde{x}_n}\circ \cdots \circ f^{-1}_{\tilde{x}_{-1}}(V^u),f^{-1}_{\tilde{x}_n}\circ \cdots \circ f^{-1}_{\tilde{x}_{-1}}(W^u)\subset V^u_n$ by Proposition \ref{shadow}(4). And they are two connected subsets of $graph(F_n)$ and so there are two closed subsets $S_1,S_2\subset \mathbb{R}^{d_u(\tilde{x}_n)}$ s.t.
     $$f^{-1}_{\tilde{x}_n}\circ \cdots \circ f^{-1}_{\tilde{x}_{-1}}(V^u)=\psi_{\tilde{x}_n}\{(F_n(v),v):v\in S_1\},$$
     and
     $$f^{-1}_{\tilde{x}_n}\circ \cdots \circ f^{-1}_{\tilde{x}_{-1}}(W^u)=\psi_{\tilde{x}_n}\{(G_n(v),v):v\in S_2\}.$$
 According to the assumptions, $S_1\bigcap S_2\neq \emptyset$ and $S_1, S_2$ are homeomorphic to the closed unit ball in $\mathbb{R}^{d_u(\tilde{x}_0)}$. Hence, if $S_1$ is not contained in $S_2$, there exists an element $z\in \partial S_2\bigcap S_1$.
\begin{enumerate}[label=(\alph*)]
    \item If $p_0^u<q_0^u$: $f^n\circ \psi_{\tilde{x}_n}(F_n(z),z)\in\partial(W^u)\bigcap V^u$ since $f^n\circ \psi_{\tilde{x}_n}$ is continuous and $f(\partial A)\subset \partial(f(A))$ ($A$ is any subset of $M$). Then $q_0^u\leq p_0^u<q_0^u$, this is a contraction. Hence, $S_1\subset S_2$.
    \item If $p_0^u=q_0^u$: given $\delta_0=\frac{p_0^u}{2}$ and $\delta\in(0,\delta_0)$, consider $V^u_\delta:=\psi_{\tilde{x}_0}\{(F_0(v),v):v\in B_{p_0^u-\delta}(0)\}$. Then $V_{\delta}^u\subset W^u$ by the previous arguments. Hence, $\overline{\bigcup_{\delta\in(0,\delta_0)} V^u_\delta}\subset\overline{W^u}$ and so $V^u\subset W^u$ .
\end{enumerate}

Similarly, the same argument holds for the $u$-case.

  (6) For any stable set $\tilde{W}^s\in\tilde{M}^s(w)$, by part(1), assume $\tilde{z}\in\hat{f}(\tilde{V}^u)\bigcap \tilde{W}^s$. By Theorem \ref{TF}, the element $p(\tilde{z})$ is unique, and there is an unstable set $\tilde{W}^u\in\tilde{M}^u(w)\subset \hat{f}(\tilde{V}^u)$. Furthermore, $\tilde{z}\in \tilde{W}^s\bigcap \tilde{W}^u$, then the negative positions of $\tilde{z}$ are uniquely determined by $\tilde{W}^u$.
\end{pf}

 Given an $I$-chain $(v_i)_{i\in\mathbb{Z}}$ where $v_i=\psi_{\tilde{x}_i}^{p_i^s,p_i^u}$. Let $\tilde{V}^s:=\tilde{V}^s((v_i)_{i\geq 0})$, $\tilde{V}^u:=\tilde{V}^u((v_i)_{i\leq 0})$ be a stable/unstable set, respectively. Denote $T\tilde{V}^{s/u}$ as the tangent bundle $TV^{s/u}\subset TM^f$, i.e. $\forall \tilde{y}\in \tilde{V}^{s/u}$, $T_{\tilde{y}}\tilde{V}^{s/u}=T_{p(\tilde{y})}V^{s/u}\subset T_{\tilde{y}}M^f$ where $V^s:=V^s((v_i)_{i\geq 0}),V^u:=V^u((v_i)_{i\leq 0})$.
\begin{prop}\label{tilde1}
   For small enough $\epsilon$: if $\tilde{y}=(y_i)_{i\in\mathbb{Z}}\in \tilde{V}^s((v_i)_{i\geq 0})$ and $\xi$ belong to the unit ball $T_{\tilde{y}}\tilde{V}^s(1)=T_{p(\tilde{y})}V^s(1)$, then $|d_{\tilde{y}}\hat{f}^n(\xi)|\leq 4I^{-n}(p_0^s)\cdot \frac{\|C_0(\tilde{x}_0)^{-1}\|}{p_0^s}$ for all $n\geq 0$; and if $\tilde{y}\in \tilde{V}^u((v_i)_{i\leq 0})$ and $\xi$ belong to the unit ball $T_{\tilde{y}}\tilde{V}^u(1)=T_{p(\tilde{y})}V^u(1)$, then $|d_{\tilde{y}}\hat{f}^{n}(\xi)|\leq 4I^{n}(p_0^u)\cdot \frac{\|C_0(\tilde{x}_0)^{-1}\|}{p_0^u}$ for all $n\leq 0$.
\end{prop}
\begin{pf}
 Assume $p(\tilde{y})=y$, $\tilde{V}^u:=
\tilde{V}^u((v_i)_{i\leq 0})$, and $V^u:=V^u((v_i)_{i\leq 0})$. Let $V^u_n:=V^u((v_{i+n})_{i\leq 0})$ with the representing function $G_n,\forall n\leq 0$. According to the assumption $\tilde{y}\in \tilde{V}^u$, thus $y_{n}=f^{-1}_{y_{n}}\circ\cdots f^{-1}_{y_{-1}}(y)=f^{-1}_{\tilde{x}_{n}}\circ \cdots \circ f^{-1}_{\tilde{x}_{-1}} (y)=\psi_{\tilde{x}_n}(G_n(v_n), v_n)$, $v_n\in B_{Q(\tilde{x}_n)}(0)$. And $(G_n(v_{n-1}), v_{n-1})=F_{\tilde{x}_{n-1},\tilde{x}_{n}}^{-1}(G_n(v_{n}), v_{n})=(D_0(\hat{f}^{-1}(\tilde{x}_n)^{-1})+\bar{H}^-_n)(G_n(v_{n}), v_{n})$ by Theorem \ref{F1}.  Write $\xi$ as
$$d_{(G_0(v_{0}), v_{0})}\psi_{\tilde{x}_0}\begin{pmatrix}
  d_{v_0}G_0(w_0) \\
  w_0
\end{pmatrix},\text{ for some }w_0\in\mathbb{R}^{d_u(\tilde{x}_0)}.$$
For any $n\leq 0$, define
$$\begin{pmatrix}
  d_{v_{n-1}}G_{n-1}(w_{n-1}) \\
  w_{n-1}
\end{pmatrix}=\Bigg[\begin{pmatrix}
  D_{s,n}^{-1} & 0 \\
  0 & D_{u,n}^{-1}
\end{pmatrix}+\begin{pmatrix}
  d_{(G_n(v_{n}), v_{n})}H_{s,n}^- \\
 d_{(G_n(v_{n}), v_{n})}H_{u,n}^-
\end{pmatrix}^{T}\Bigg]\begin{pmatrix}
  d_{v_n}G_n(w_n) \\
  w_n
\end{pmatrix},$$
where $D_0(\hat{f}^{-1}(\tilde{x}_n))=\begin{pmatrix}
  D_{s,n} & 0 \\
  0 & D_{u,n}
\end{pmatrix}$ and $H_n^-=H_{s,n}^-+H_{u,n}^-$. The following calculations are similar to the estimate for inequality (\ref{s}):
\begin{equation*}
    \begin{aligned}
        |w_{n-1}|&\leq \|D_{u,n}^{-1}\|\cdot|w_n|+\|d_{(G_n(v_{n}), v_{n})}H_{s,n}^-\|\cdot|(d_{v_n}G_n(w_n),w_n)|\\
        &\leq (e^{-\frac{1}{u^2(\tilde{x}_n)}}+2\sqrt{\epsilon}\eta_n^{\frac{\beta}{2}})|w_n|\\
        &\leq e^{-\frac{1-\epsilon}{u^2(\tilde{x}_n)}}|w_n|
        \leq \cdots \leq e^{-\sum_{i=0}^n\frac{1-\epsilon}{u^2(\tilde{x}_i)}}|w_0|,
    \end{aligned}
\end{equation*}
where $\eta_n=p^s_n\land p^u_n$. And $|w_0|\leq \frac{1}{2}\|d_{(G_0(v_{0}), v_{0})}\psi_{\tilde{x}_0}\|^{-1}\cdot|\xi|\leq \|C_0(\tilde{x}_0)^{-1}\|$. Hence,
\begin{equation*}
    \begin{aligned}
    |d_{\tilde{y}}\hat{f}^{n}(\xi)|&=\|d_{(G_n(v_{n}), v_{n})}\psi_{\tilde{x}_i}\|\cdot\bigg|\begin{pmatrix}
  d_{v_n}G_n(w_n) \\
  w_n
\end{pmatrix}\bigg|\\
&\leq 4e^{-\sum_{i=0}^{n+1}\frac{1-\epsilon}{u^2(\tilde{x}_n)}}|w_0|\\
&\leq 4I^{n}(p_0^u)\cdot \frac{\|C_0(\tilde{x}_0)\|^{-1}}{p_0^u}.
    \end{aligned}
\end{equation*}
Apply the properties of $F_{\tilde{x}_{n},\tilde{x}_{n+1}}$ to the $s$-case ($n\geq 0$); the similar arguments hold for the $s$-case.
\end{pf}

Recall that $\Sigma:=\Sigma(\mathcal{G})=\big\{\{\psi_{\tilde{x}_i}^{p_i^s,p_i^u}\}_{i\in\mathbb{Z}}:(\psi_{\tilde{x}_i}^{p_i^s,p_i^u},\psi_{\tilde{x}_{i+1}}^{p_{i+1}^s,p_{i+1}^u})\in\mathcal{E},\forall i\in \mathbb{Z\big\}}$. Define a map $\pi:\Sigma(\mathcal{G})\rightarrow M^f$ by $\pi((v_i)_{i\in\mathbb{Z}}):= \tilde{V}^s((v_i)_{i\geq 0})\bigcap\tilde{V}^u((v_i)_{i\leq 0})$. By Proposition \ref{tilde}(1), the map $\pi$ is well-defined.
\begin{thm}\label{main1}
For small enough $\epsilon$, given $I$-chain $(v_i)_{i\in\mathbb{Z}}$ where $v_i=\psi_{\tilde{x}_{i+1}}^{p_i^s,p_i^u}$. Then
\begin{enumerate}[label=(\arabic*)]
    \item $\pi\circ \sigma=\hat{f}\circ \pi$.
    \item $\pi:\Sigma\rightarrow M^f$ is uniformly continuous.
    \item $\pi(\Sigma)\supset \pi(\Sigma^{\#})\supset RST$.
\end{enumerate}
\end{thm}
\begin{pf}
    (1) Assume $\pi((v_i)_{i\in\mathbb{Z}})=\tilde{y}\in \tilde{V}^s((v_i)_{i\geq 0})\bigcap\tilde{V}^u((v_i)_{i\leq 0})$ and $\tilde{y}=(y_k)_{k\in\mathbb{Z}}$. By definition, $y_k=f^k(y)$ for all $k\geq 0$; $y_k=f^{-1}_{\tilde{x}_k}(y_{k+1})$ for all $k< 0$. Then $y_{k+1}=f(y_k)\in V^s((v_{k+1+i})_{i\geq0})$ for all $k\geq 0$; and $y_{k+1}=f^{-1}_{\tilde{x}_{k+1}}(y_{k+2})=\cdots =f^{-1}_{\tilde{x}_{k+1}}\circ\cdots \circ f^{-1}_{\tilde{x}_{-1}}(y_{0})\in V^u((v_{k+1+i})_{i\leq0})$ for all $k< 0$ by Proposition \ref{tilde}(3). Hence,
    $\hat{f}\circ \pi((v_i)_{i\in\mathbb{Z}}) =(f(y_k))_{k\in\mathbb{Z}}=(y_{k+1})_{k\in\mathbb{Z}}=\tilde{V}^s((v_{i+1})_{i\geq 0})\bigcap\tilde{V}^u((v_{i+1})_{i\leq 0}) =\pi\circ \sigma((v_i)_{i\in\mathbb{Z}})$.

    (2) For all $N\geq 0$, suppose the chain $(w_i)_{i\in\mathbb{Z}}$ of double charts satisfies $v_i=w_i$ for $|i|\leq N$, let $\tilde{p}^1=(p_k^1)_{k\in\mathbb{Z}}=\tilde{V}^s_1((v_i)_{i\geq 0})\bigcap\tilde{V}^u_1((v_i)_{i\leq 0})$ and let $\tilde{p}^2=(p_k^2)_{k\in\mathbb{Z}}=\tilde{V}^s_2((w_i)_{i\geq 0})\bigcap\tilde{V}^u_2((w_i)_{i\leq 0})$. By Proposition \ref{ad}(2) and part (1), for $-N\leq n\leq0$,
    \begin{equation*}
        \begin{aligned}
            d(p_n^1,p_n^2)&\leq 3[d_{C^0}(\hat{V}^s_1((v_{i+n})_{i\geq0}),\hat{V}^s_2((v_{i+n})_{i\geq0}))+d_{C^0}(\hat{V}^u_1((v_{i+n})_{i\leq0}),\hat{V}^u_2((v_{i+n})_{i\leq0}))]\\
            &\leq 6C[I^{-(N+n)}(p_{0}^u\land p_{0}^u)]^{\frac{1}{\gamma}+\alpha} \text{( by Proposition \ref{shadow}(5))}.
        \end{aligned}
    \end{equation*}
    Note that  Proposition \ref{ad}(2) is not a conclusion for leaf cores, but it is easy to obtain the same arguments hold for leaf cores. Similar to the calculations of inequality (\ref{symbol}),
    \begin{equation*}
        \begin{aligned}
            d(\tilde{p}^1,\tilde{p}^2)&\leq \max\{2^{-N-1},6\cdot2^nC[I^{-(N+n)}(p_{0}^u\land p_{0}^u)]^{\frac{1}{\gamma}+\alpha}\}\\
            &\leq \max\{2^{-N-1},6C[I^{-N}(p_{0}^s\land p_{0}^u)]^{\frac{1}{\gamma}+\alpha}\}\leq 6C[I^{-N}(p_{0}^s\land p_{0}^u)]^{\frac{1}{\gamma}+\alpha}.
        \end{aligned}
    \end{equation*}

    (3) For any $\tilde{x}\in RST$, there exists a chain $\{\psi_{\tilde{x}_k}^{p_k^s,p_k^u}\}_{k\in\mathbb{Z}}\in\Sigma(\mathcal{G})$ such that $\psi_{\tilde{x}_k}^{p_k^s,p_k^u}$ $I$-overlaps $\psi_{\hat{f}^k(\tilde{x})}^{p_k^s,p_k^u}$ and $p_k^s\land p_k^u\geq q_k$ $\forall k\in\mathbb{Z}$ by Proposition \ref{exist}, where $q_k\in \mathcal{I}\bigcap [I^{-\frac{1}{4}}(q(\hat{f}^n(\tilde{x}))),I^{\frac{1}{4}}(q(\hat{f}^n(\tilde{x})))]$ and the function $q$ is   the same as the definition of RST. Denote $v_k=\psi_{\tilde{x}_k}^{p_k^s,p_k^u}$ for all $k\in\mathbb{Z}$, then $\tilde{x}=\pi((v_k)_{k\in\mathbb{Z}})$. By the definition of RST, $\limsup_{n \to \pm\infty} q(\hat{f}^n(\tilde{x})) > 0$, thus there exist infinitely many $k_i\leq 0, k_j\geq 0$ s.t. $p_{k_i}^s\land p_{k_i}^s\geq q_{k_i}>0$ and $p_{k_j}^s\land p_{k_j}^s\geq q_{k_j}>0$. By the discreteness of Proposition \ref{graph A}(1), there exists $(w_i)_{i\in\mathbb{Z}}\in \Sigma^{\#}$ s.t. $\pi((w_i))_{i\in\mathbb{Z}}=\tilde{x}$.
\end{pf}
\subsection{\texorpdfstring{Inverse theorem}{Inverse theorem}}
In this subsection, we always assume that $(v_i)_{i\in\mathbb{Z}}=(\psi_{\tilde{x}_i}^{p_i^s,p_i^u})_{i\in\mathbb{Z}}$, $(w_i)_{i\in\mathbb{Z}}=(\psi_{\tilde{y}_i}^{q_i^s,q_i^u})_{i\in\mathbb{Z}}\in\Sigma^{\#}$ s.t. $\pi((v_i)_{i\in\mathbb{Z}})=\pi((w_i)_{i\in\mathbb{Z}})=\tilde{z}$, $p(\tilde{z})=\psi_{\tilde{x}_0}(v)$ and $p(\tilde{z})=\psi_{\tilde{y}_0}(w)$, $v,w\in\mathbb{R}^d$. We will show that the two $I$-chains are almost same i.e. the respective parameters are very "close".
\begin{enumerate}[series=myenum]
    \item Estimate the distance $\tilde{x}_i$ and $\tilde{y}_i$ and compare $\rho(\tilde{x}_i)$ with $\rho(\tilde{y}_i)$:\label{1.}
\end{enumerate}
By Proposition \ref{ad}(2), $|v|\leq 10^{-2}(p_0^s\land p_0^u)^2$. Then \begin{equation}\label{x}
    d(p(\tilde{x}_0),p(\tilde{z}))=d(\psi_{\tilde{x}_0}(0),\psi_{\tilde{x}_0}(v))\leq 50^{-1}(p_0^s\land p_0^u)^2.
\end{equation}
Similarly, $d(p(\tilde{y}_0),p(\tilde{z}))\leq 50^{-1}(q_0^s\land q_0^u)^2$. Hence, $d(p(\tilde{x}_0),p(\tilde{y}_0))\leq 25^{-1}\max\{(p_0^s\land p_0^u)^2,(q_0^s\land q_0^u)^2\}$.

By inequality (\ref{x}), $d(p(\tilde{x}_0),p(\tilde{z}))\leq \epsilon^2\rho(\tilde{x}_0)\leq \epsilon^2d(p(\tilde{x}_0),\Gamma_\infty)$, so $p(\tilde{z})\in D_{\tilde{x}_0}$ and $p(\tilde{z})\in \mathcal{E}_{\hat{f}^{-1}(\tilde{x}_0)}$ for small enough $\epsilon$. Then
$$d(p(\tilde{z}),\Gamma_\infty)=d(p(\tilde{x}_0),\Gamma_\infty)\pm d(p(\tilde{x}_0),p(\tilde{z}))=e^{\pm\epsilon}d(p(\tilde{x}_0),\Gamma_\infty).$$
According to assumption \ref{A2} and inequality (\ref{x}), $d(p_{\pm1}(\tilde{x}_0),p_{\pm1}(\tilde{z}))\leq d(p(\tilde{x}_0),\Gamma_\infty)^{-a}d(p(\tilde{x}_0),p(\tilde{z}))\leq \epsilon^2\rho(\tilde{x}_0)\leq\epsilon^2d(p_{\pm1}(\tilde{x}_0),\Gamma_{\infty})$. Thus $d(p_{\pm1}(\tilde{z}),\Gamma_\infty)=e^{\pm\epsilon}d(p_{\pm1}(\tilde{x}_0),\Gamma_\infty).$ Then $\frac{\rho(\tilde{x}_0)}{\rho(\tilde{z})}=e^{\pm\epsilon}$. Similarly, $\frac{\rho(\tilde{y}_0)}{\rho(\tilde{z})}=e^{\pm\epsilon}$. Therefore, $\frac{\rho(\tilde{x}_0)}{\rho(\tilde{y}_0})=e^{\pm2\epsilon}$.

Applying this to $(\psi_{\tilde{x}_{i+k}}^{p_{i+k}^s,p_{i+k}^u})_{k\in\mathbb{Z}}, (\psi_{\tilde{y}_{i+k}}^{q_{i+k}^s,p_{i+k}^u})_{k\in\mathbb{Z}}$, the same arguments hold for other $i\neq 0$.
\begin{enumerate}[resume=myenum]
    \item Compare $\|C_0(\tilde{x}_i)^{-1}\|$ with $\|C_0(\tilde{y}_i)^{-1}\|$:
\end{enumerate}
Assume $\tilde{z}$ belongs to the stable set $\tilde{V}^s((v_i)_{i\geq0})$, then $\tilde{z}\in 0$-summ:
\begin{equation}\label{S}
    \begin{aligned}
        s^2(\tilde{z})&\leq 2\sum_{n=0}^{\infty} \Big(4I^{-n}(p_0^s)\cdot \frac{\|C_0(\tilde{x}_0)^{-1}\|}{p_0^s}\Big)^2\text{ (by Proposition \ref{tilde1})}\\
        &\leq 32\Big(\frac{\|C_0(\tilde{x}_0)^{-1}\|}{p_0^s}\Big)^2\sum_{n=0}^{\infty}(I^{-n}(1))^2 \ (I^-(x)\text{ is  an increasing function}).
    \end{aligned}
\end{equation}
Since $\tilde{x}_0$ is 0-summ, then $s^2(\tilde{x}_0)<\infty$ and $u^2(\tilde{x}_0)<\infty$, and so  $\|C_0(\tilde{x}_0)^{-1}\|^2<\infty$ by Proposition \ref{C}(1). And $(\sum_{n=0}^{\infty}I^{-n}(1))^2=\sum_{n=0}^{\infty}(I^{-n}(1))^{\frac{1}{\gamma}+(2-\frac{1}{\gamma})}<\infty$ by Proposition \ref{I}(4). Thus, $s^2(\tilde{z})<\infty$ by inequality (\ref{S}). Similarly, $u^2(\tilde{z})<\infty$.

Let $\tilde{V}^s_n:=\tilde{V}^s((v_{i+n})_{i\geq0})$,$\tilde{V}^u_n:=\tilde{V}^u((v_{i+n})_{i\leq0}$), and $V^s_n:=V^s((v_{i+n})_{i\geq0})$ with the representing function $G_n$, $V^u_n:=V^u((v_{i+n})_{i\leq0})$ with the representing function $F_n$ at $\psi_{\tilde{x}_n}^{p_n^s,p_n^u}$, for all $n\in\mathbb{Z}$. For any $\tilde{x}\in \tilde{V}_n^s$, define map $\pi_{\tilde{x}_n}^s: T_{\tilde{x}}\tilde{V}_n^s\rightarrow E_{\tilde{x}_n}^s$ by $$\pi_{\tilde{x}_n}^s(\xi):=d_0\psi_{\tilde{x}_n}\begin{pmatrix}
     w\\
     0
 \end{pmatrix}\ \forall \xi\in T_{\tilde{x}}\tilde{V}^s_n,$$
 where $\xi=d_z\psi_{\tilde{x}_n}\begin{pmatrix}
     w\\
     d_{v}G_n(w)
 \end{pmatrix}$ and $z=(v,G_n(v))=\psi_{\tilde{x}_n}^{-1}(p(\tilde{x}))$. Similarly, define map $\pi_{\tilde{x}_n}^u: T_{\tilde{x}}\tilde{V}_n^u\rightarrow E_{\tilde{x}_n}^u$ by $$\pi_{\tilde{x}_n}^u(\zeta):=d_0\psi_{\tilde{x}_n}\begin{pmatrix}
     0\\
     w
 \end{pmatrix}\ \forall \zeta\in T_{\tilde{x}}\tilde{V}^u_n,$$
 where $\zeta=d_z\psi_{\tilde{x}_n}\begin{pmatrix}
     d_{v}F_n(w)\\
     w
 \end{pmatrix}$ and $z=(F_n(v),v)=\psi_{\tilde{x}_n}^{-1}(p(\tilde{x}))$.
\begin{lem}\label{pi}
    Assume that $\xi$ satisfies the preceding conditions, then  $\frac{|\pi_{\tilde{x}_n}^s(\xi)|}{|\xi|}=e^{\pm Q(\tilde{x}_n)^{\frac{\beta}{2}-\frac{1}{2\gamma}}},\ \forall \xi\in T_{\tilde{x}}\tilde{V}^s_n$.
\end{lem}
 \begin{pf}
     For small enough $\epsilon$, there exists a set $D\in\mathcal{D}$ such that $p(\tilde{x}_n), p(\tilde{x})$  belong to $D$, and let $\Theta:=\Theta_{D}$. Then we can estimate the distortion of $\pi_{\tilde{x}_n}^s$:
assume $p(\tilde{x})=\psi_{\tilde{x}_n}(v,G_n(v))$, then \begin{equation}\label{dp}
    \big|\exp^{-1}_{p(\tilde{x}_n)}(p(\tilde{x}))\big|=d(p(\tilde{x}),p(\tilde{x}_n))\leq |\psi_{\tilde{x}_n}(v,G_n(v))-\psi_{\tilde{x}_n}(0,0)|\leq 4p_n^s.
\end{equation}
Hence,
\begin{equation*}
    \begin{aligned}
        |\Theta(\xi)-\Theta(\pi_{\tilde{x}_n}^s(\xi))|
        &=\Big|\Theta d_z\psi_{\tilde{x}_n}\begin{pmatrix}
     w\\
     d_{v}G_n(w)
 \end{pmatrix}-\Theta d_0\psi_{\tilde{x}_n}\begin{pmatrix}
     w\\
     0
 \end{pmatrix}\Big|\\
 &\leq \Big| \Theta d_z\psi_{\tilde{x}_n}\begin{pmatrix}
     w\\
     d_{v}G_n(w)
 \end{pmatrix}-\Theta d_z\psi_{\tilde{x}_n}\begin{pmatrix}
     w\\
     0
 \end{pmatrix}\Big|+\Big|\Theta d_z\psi_{\tilde{x}_n}\begin{pmatrix}
     w\\
     0
 \end{pmatrix}-\Theta d_0\psi_{\tilde{x}_n}\begin{pmatrix}
     w\\
     0
 \end{pmatrix}\Big|\\
 &\leq 2\|C_0(\tilde{x}_n)\|\cdot\Big|\begin{pmatrix}
     w\\
     d_{v}G_n(w)
 \end{pmatrix}-\begin{pmatrix}
     w\\
     0
 \end{pmatrix}\Big|+L_3|C_0(\tilde{x}_n)(z)| \\
 &\leq 2\|d_{\cdot}G_n\|\cdot|w| +L_3\exp^{-1}_{p(\tilde{x}_n)}(p(\tilde{x}))\text{ (since }z=\psi_{\tilde{x}_n}^{-1}(p(\tilde{x})))\\
 &\leq 4(p_n^s)^{\frac{\beta}{2}}\cdot \|C_0(\tilde{x}_n)^{-1}\|+4 L_3p_n^s \text{ (by inequalities (\ref{G1}) and (\ref{dp})})\\
 &\leq Q(\tilde{x}_n)^{\frac{\beta}{2}-\frac{1}{2\gamma}}.
    \end{aligned}
\end{equation*}
 \end{pf}
  Similarly, $\frac{|\pi_{\tilde{x}_n}^u(\zeta)|}{|\zeta|}=e^{\pm Q(\tilde{x}_n)^{\frac{\beta}{2}-\frac{1}{2\gamma}}},\ \forall \zeta\in T_{\tilde{x}}\tilde{V}^u_n$.
 \begin{lem}\label{S1}
    Let $\tilde{x}\in \hat{f}(\tilde{\mathcal{F}}_{v_n,v_{n+1}}(\tilde{V}^s_{n+1}))=\hat{f}(\tilde{V}^s_{n})\subset \tilde{V}^s_{n+1}$, $\xi\in T_{\hat{f}^{-1}(\tilde{x})}\tilde{V}^s_{n}$. For all $\rho_0\geq e^{\Gamma Q^{\frac{\beta}{8}-\frac{1}{4\gamma}}(\tilde{x}_n)}$, if $\frac{S(\tilde{x},d_{\hat{f}^{-1}}(\tilde{x})\hat{f}(\xi))}{S(\tilde{x}_{n+1},\pi^s_{\tilde{x}_{n+1}}(d_{\hat{f}^{-1}(\tilde{x})}\hat{f}(\xi)))}\in[\rho_0^{-1},\rho_0]$, then $\frac{S(\hat{f}^{-1}(\tilde{x}),\xi)}{S(\tilde{x}_n,\pi_{\tilde{x}_n}^s(\xi))}
    \in\Big[e^{\frac{1}{2}\Gamma Q^{\frac{\beta}{8}+\frac{3}{4\gamma}}(\tilde{x}_n)}\rho_0^{-1}, e^{-\frac{1}{2} \Gamma Q^{\frac{\beta}{8}+\frac{3}{4\gamma}}(\tilde{x}_n)}\rho_0\Big]$.

    A similar statement holds for unstable sets.
 \end{lem}
 \begin{pf}
Without loss of generality,  we assume $|\xi|=1$, then $|\pi_{\tilde{x}_{n}}^s(\xi)|\geq e^{-Q(\tilde{x}_n)^{\frac{\beta}{2}-\frac{1}{2\gamma}}}\geq \frac{1}{2}$ for small enough $\epsilon$, by Lemma \ref{pi}. Decompose
      \begin{equation}\label{decompose}
         \frac{S(\hat{f}^{-1}(\tilde{x}),\xi)}{S(\tilde{x}_n,\pi_{\tilde{x}_n}^s(\xi))}
         =\frac{S(\hat{f}^{-1}(\tilde{x}_{n+1}),\zeta)}{S(\tilde{x}_n,\pi_{\tilde{x}_n}^s(\xi))}
          \cdot\frac{S(\hat{f}^{-1}(\tilde{x}),\xi)}{S(\hat{f}^{-1}(\tilde{x}_{n+1}),\zeta)},
      \end{equation}
      where $\zeta=d_{\tilde{x}_{n+1}}\hat{f}^{-1}(\pi_{\tilde{x}_{n+1}}^s(d_{\hat{f}^{-1}(\tilde{x})}\hat{f}(\xi)))\in E^s(\hat{f}^{-1}(\tilde{x}_{n+1}))$.

      Firstly, we estimate the left factor of equality (\ref{decompose}):
\begin{equation*}
    \begin{aligned}
        \bigg|\frac{S(\hat{f}^{-1}(\tilde{x}_{n+1}),\zeta)}{S(\tilde{x}_n,\pi_{\tilde{x}_n}^s(\xi))}-1\bigg|&=\bigg|\frac{|C_0^{-1}(\hat{f}^{-1}(\tilde{x}_{n+1}))\zeta|}{|C_0^{-1}(\tilde{x}_n)\pi_{\tilde{x}_n}^s(\xi)|}-1\bigg|\\
        &\leq 2\big[|C_{n+1}^{-1}\circ \Theta\zeta-C_n^{-1}\Theta\zeta|+|C_n^{-1}\Theta\zeta-C_n^{-1}\circ\Theta\pi_{\tilde{x}_n}^s(\xi)|\big]\\
        &\leq 2\big[\|C_{n+1}^{-1}-C_n^{-1}\|\cdot|\zeta|+\|C_n^{-1}\|\cdot|\Theta\zeta-\Theta\pi_{\tilde{x}_n}^s(\xi)|\big],
    \end{aligned}
\end{equation*}
where $C_n:=\Theta\circ C_0(\tilde{x}_n)$ and $C_{n+1}:=\Theta\circ C_0(\hat{f}^{-1}(\tilde{x}_{n+1})),\Theta:=\Theta_D$ ($D\in\mathcal{D}$ contains $p(\tilde{x}_n),p(\hat{f}^{-1}(\tilde{x}_{n+1}))$ and $p(\hat{f}^{-1}(\tilde{x}))$. By Proposition \ref{overlap1}(2), $\|C_{n+1}^{-1}-C_n^{-1}\|\leq \eta_n^3\eta_{n+1}^3$ where $\eta_n:=p_n^s\land p_n^u$. Notice that $p(\tilde{x}_{n+1})\in\mathfrak{E}_{p(\tilde{x}_n)}$ and $p(\hat{f}^{-1}(\tilde{x}))\in D_{p(\tilde{x}_n)}$ for small enough $\epsilon$. Thus, by assumption  \ref{A2} and Lemma \ref{pi},
$$|\zeta|=|d_{\tilde{x}_{n+1}}\hat{f}^{-1}(\pi_{\tilde{x}_{n+1}}^s(d_{\hat{f}^{-1}(\tilde{x})}\hat{f}(\xi)))|\leq d(p(\tilde{x}_n),\Gamma_\infty)^{-2a}e^{Q(\tilde{x}_{n+1})^{\frac{\beta}{2}-\frac{1}{2\gamma}}}\leq 2d(p(\tilde{x}_n),\Gamma_\infty)^{-2a}.$$
Hence,
\begin{equation*}
\begin{aligned}
   \bigg|\frac{S(\hat{f}^{-1}(\tilde{x}_{n+1}),\zeta)}{S(\tilde{x}_n,\pi_{\tilde{x}_n}^s(\xi))}-1\bigg|
    &\leq 4d(p(\tilde{x}_n),\Gamma_\infty)^{-2a}\eta_n^3\eta_{n+1}^3+2\|C_0(\tilde{x}_n)^{-1}\|\cdot\big(|\Theta\zeta-\Theta\xi|+|\Theta\xi-\Theta\pi_{\tilde{x}_n}^s(\xi)|\big).
\end{aligned}
\end{equation*}
 Next, we estimate $|\Theta\zeta-\Theta\xi|$. Assume $D'\in\mathcal{D}$ contains points $p(\tilde{x}_{n+1}),p(\tilde{x})$, and $\Theta':=\Theta_{D'}$. Since $\tilde{x}\in\hat{f}(\tilde{V}^s_{n})$, then $\hat{f}^{-1}(\tilde{x})\in \tilde{V}^s_{n}$. Assume $p(\hat{f}^{-1}(\tilde{x}))=\psi_{\tilde{x}_{n}}(v,G_n(v)),|v|\leq p_n^s$ and $p(\tilde{x})=\psi_{\tilde{x}_{n+1}}(w,G_{n+1}(w))$, then $|w|\leq (e^{-\frac{1}{s^2(\tilde{x}_n)}}+\epsilon (p_n^s)^3+\sqrt{\epsilon}(p^s_n)^{\frac{\beta}{2}})\cdot |v|\leq p_n^s$, so $d(p(\tilde{x}_{n+1}), p(\tilde{x}))\leq |\psi_{\tilde{x}_{n+1}}(w,G_{n+1}(w))-\psi_{\tilde{x}_{n+1}}(0,0)|\leq 4p_n^s$,
 \begin{equation*}
     \begin{aligned}
&|\Theta'\pi_{\tilde{x}_{n+1}}^s(d_{\hat{f}^{-1}(\tilde{x})}\hat{f}(\xi))-\Theta'd_{\hat{f}^{-1}(\tilde{x})}\hat{f}(\xi)|\leq 4(p_n^s)^{\frac{\beta}{2}}\cdot \|C_0(\tilde{x}_{n+1})^{-1}\|+4 L_3p_n^s\\
&\leq Q(\tilde{x}_n)^{\frac{\beta}{2}-\frac{1}{2\gamma}}\cdot [(I(p_{n+1}^s))^{\frac{1}{2\gamma}}\cdot \|C_0(\tilde{x}_{n+1})^{-1}\|]+4 L_3Q(\tilde{x}_n)\leq Q(\tilde{x}_n)^{\frac{\beta}{2}-\frac{1}{2\gamma}},
     \end{aligned}
 \end{equation*}
 and
 \begin{equation}
     |\pi_{\tilde{x}_{n+1}}^s(d_{\hat{f}^{-1}(\tilde{x})}\hat{f}(\xi))|\leq e^{Q(\tilde{x}_n)^{\frac{\beta}{2}-\frac{1}{2\gamma}}}\cdot|d_{\hat{f}^{-1}(\tilde{x})}\hat{f}(\xi)|\leq d(p(\tilde{x}_n),\Gamma_\infty)^{-a}e^{Q(\tilde{x}_n)^{\frac{\beta}{2}-\frac{1}{2\gamma}}}
 \end{equation}by the proof of the Lemma \ref{pi}.  Thus, by assumptions \ref{A2}, \ref{A3},
 \begin{equation}\label{zeta1}
     \begin{aligned}
         |\Theta\zeta-\Theta\xi|&=|\Theta d_{\tilde{x}_{n+1}}\hat{f}^{-1}(\pi_{\tilde{x}_{n+1}}^s(d_{\hat{f}^{-1}(\tilde{x})}\hat{f}(\xi)))-\Theta d_{\tilde{x}}\hat{f}^{-1}(d_{\hat{f}^{-1}(\tilde{x})}\hat{f}(\xi))|\\
         &\leq |\Theta d_{\tilde{x}_{n+1}}\hat{f}^{-1}(\pi_{\tilde{x}_{n+1}}^s(d_{\hat{f}^{-1}(\tilde{x})}\hat{f}(\xi)))-\Theta d_{\tilde{x}}\hat{f}^{-1}\nu_{\tilde{x}}\Theta'(\pi_{\tilde{x}_{n+1}}^s(d_{\hat{f}^{-1}(\tilde{x})}\hat{f}(\xi)))|\\
         &+|\Theta d_{\tilde{x}}\hat{f}^{-1}\nu_{\tilde{x}}\Theta'(\pi_{\tilde{x}_{n+1}}^s(d_{\hat{f}^{-1}(\tilde{x})}\hat{f}(\xi)))-\Theta d_{\tilde{x}}\hat{f}^{-1}(d_{\hat{f}^{-1}(\tilde{x})}\hat{f}(\xi))|\\
         &\leq \|\Theta d_{\tilde{x}_{n+1}}\hat{f}^{-1}\nu_{\tilde{x}_{n+1}}-\Theta d_{\tilde{x}}\hat{f}^{-1}\nu_{\tilde{x}}\|\cdot|\pi_{\tilde{x}_{n+1}}^s(d_{\hat{f}^{-1}(\tilde{x})}\hat{f}(\xi))|\\         &+\|d_{\tilde{x}}\hat{f}^{-1}\|\cdot|\Theta'\pi_{\tilde{x}_{n+1}}^s(d_{\hat{f}^{-1}(\tilde{x})}\hat{f}(\xi))-\Theta'd_{\hat{f}^{-1}(\tilde{x})}\hat{f}(\xi)|\\
          &\leq Kd(p(\tilde{x}_{n+1}), p(\tilde{x}))^\beta d(p(\tilde{x}_{n}),\Gamma_\infty)^{-a}e^{Q(\tilde{x}_n)^{\frac{\beta}{2}-\frac{1}{2\gamma}}}+d(p(\tilde{x}_n),\Gamma_\infty)^{-a}Q(\tilde{x}_n)^{\frac{\beta}{2}-\frac{1}{2\gamma}}\\
         &\leq 2^{1+2\beta}Kd(p(\tilde{x}_n),\Gamma_\infty)^{-a}[(p_n^s)^{\beta}+Q(\tilde{x}_n)^{\frac{\beta}{2}-\frac{1}{2\gamma}}]\\
         &\leq Q(\tilde{x}_n)^{\frac{\beta}{4}-\frac{1}{2\gamma}}.
     \end{aligned}
 \end{equation}
Therefore,
 \begin{equation*}
\begin{aligned}
    \bigg|\frac{S(\hat{f}^{-1}(\tilde{x}_{n+1}),\zeta)}{S(\tilde{x}_{n},\pi_{\tilde{x}_{n}}^s(\xi))}-1\bigg|
    &\leq 4d(p(\tilde{x}_{n}),\Gamma_\infty)^{-2a}Q(\tilde{x}_{n})^3Q(\hat{f}^{-1}(\tilde{x}_{n+1}))^3+2  \|C_0(\tilde{x}_{n})^{-1}\|\cdot\big(Q(\tilde{x}_n)^{\frac{\beta}{4}-\frac{1}{2\gamma}}+Q(\tilde{x}_{n})^{\frac{\beta}{4}}\big)\\
    &\leq Q(\tilde{x}_n)^{\frac{\beta}{4}-\frac{1}{\gamma}}.
\end{aligned}
\end{equation*}
This implies that
\begin{equation}\label{right}
    \frac{S(\hat{f}^{-1}(\tilde{x}_{n+1}),\zeta)}{S(\tilde{x}_n,\pi_{\tilde{x}_n}^s(\xi))}=e^{\pm Q(\tilde{x}_n)^{\frac{\beta}{4}-\frac{1}{\gamma}}}.
\end{equation}

For the right factor of equality (\ref{decompose}) $\frac{S(\hat{f}^{-1}(\tilde{x}),\xi)}{S(\hat{f}^{-1}(\tilde{x}_{n+1}),\zeta)}$: by definition, we have
$$S^2(\hat{f}^{-1}(\tilde{x}),\xi)=2\sum_{m=0}^{\infty}|d_{\hat{f}^{-1}(\tilde{x})}\hat{f}^m(\xi)|^2= 2|\xi|^2+S^2(\tilde{x},d_{\hat{f}^{-1}(\tilde{x})}\hat{f}(\xi)),$$
and
$$S^2(\hat{f}^{-1}(\tilde{x}_{n+1}),\zeta)=2\sum_{m=0}^{\infty}|d_{\hat{f}^{-1}(\tilde{x}_{n+1})}\hat{f}^m(\zeta)|^2= 2|\zeta|^2+S^2(\tilde{x}_{n+1},d_{\hat{f}^{-1}(\tilde{x}_{n+1})}\hat{f}(\zeta)).$$
By inequality (\ref{zeta1}), \begin{equation}\label{zeta}
     e^{-Q(\tilde{x}_{n})^{\frac{\beta}{4}-\frac{1}{2\gamma}}}\leq |\zeta|\leq e^{Q(\tilde{x}_{n})^{\frac{\beta}{4}-\frac{1}{2\gamma}}}.
 \end{equation}
Then
\begin{equation*}
    \begin{aligned}
        \frac{S^2(\hat{f}^{-1}(\tilde{x}),\xi)}{S^2(\hat{f}^{-1}(\tilde{x}_{n+1}),\zeta)}&=\frac{2|\xi|^2+S^2(\tilde{x},d_{\hat{f}^{-1}(\tilde{x})}\hat{f}(\xi))}{2|\zeta|^2+S^2(\tilde{x}_{n+1},d_{\hat{f}^{-1}(\tilde{x}_{n+1})}\hat{f}(\zeta))}
        \overset{\text{by assumption}}{\leq}\frac{2|\xi|^2+\rho_0^2S^2(\tilde{x}_{n+1},\pi^s_{\tilde{x}_{n+1}}(d_{\hat{f}^{-1}}(\tilde{x})\hat{f}(\xi)))}{2|\zeta|^2+S^2(\tilde{x}_{n+1},d_{\hat{f}^{-1}(\tilde{x}_{n+1})}\hat{f}(\zeta))}\\
        &= \frac{2|\xi|^2+\rho_0^2S^2(\tilde{x}_{n+1},d_{\hat{f}^{-1}(\tilde{x}_{n+1})}\hat{f}(\zeta))}{2|\zeta|^2+S^2(\tilde{x}_{n+1},d_{\hat{f}^{-1}(\tilde{x}_{n+1})}\hat{f}(\zeta))}
        =\rho_0^2-\frac{2(\rho^2_0-1)|\zeta|^2+2(|\zeta|^2-1)}{S^2(\hat{f}^{-1}(\tilde{x}_{n+1}),\zeta)}\\
        &\leq \rho_0^2-2(\rho^2_0-1)\frac{|\zeta|^2+\frac{|\zeta|^2-1}{\rho^2_0-1}}{S^2(\tilde{x}_n,\pi_{\tilde{x}_n}^s(\xi))}\cdot e^{-Q(\tilde{x}_n)^{\frac{\beta}{4}-\frac{1}{\gamma}}} \text{ (by inequality (\ref{right}))}\\
         &\leq \rho_0^2-2(\rho^2_0-1)\frac{e^{-2Q(\tilde{x}_n)^{\frac{\beta}{4}-\frac{1}{2\gamma}}}+\frac{e^{-2Q(\tilde{x}_n)^{\frac{\beta}{4}-\frac{1}{2\gamma}}}-1}{e^{2\Gamma Q(\tilde{x}_n)^{\frac{\beta}{8}-\frac{1}{4\gamma}}}-1}}{S^2(\tilde{x}_n,\pi_{\tilde{x}_n}^s(\xi))}\cdot e^{-Q(\tilde{x}_n)^{\frac{\beta}{4}-\frac{1}{\gamma}}}\\
        &\leq \rho_0^2-2(\rho^2_0-1)\frac{e^{-2Q(\tilde{x}_n)^{\frac{\beta}{4}-\frac{1}{2\gamma}}}-\frac{1}{\Gamma}Q(\tilde{x}_n)^{\frac{\beta}{8}-\frac{1}{4\gamma}}}{\|C_0(\tilde{x}_n)^{-1}\|^2\cdot |\pi_{\tilde{x}_n}^s(\xi)|^2}\cdot e^{-Q(\tilde{x}_n)^{\frac{\beta}{4}-\frac{1}{\gamma}}}\\
        &\leq \rho_0^2-2(\rho^2_0-1)\frac{1-2Q(\tilde{x}_n)^{\frac{\beta}{4}-\frac{1}{2\gamma}}-\frac{1}{\Gamma}Q(\tilde{x}_n)^{\frac{\beta}{8}-\frac{1}{4\gamma}}}{\|C_0(\tilde{x}_n)^{-1}\|^2}\cdot e^{-Q(\tilde{x}_n)^{\frac{\beta}{4}-\frac{1}{\gamma}}-2Q(\tilde{x}_n)^{\frac{\beta}{2}-\frac{1}{2 \gamma}}}\\
        &\leq \rho_0^2\Big(1-\frac{2(1-\frac{1}{\rho_0^2})e^{-\frac{3}{\Gamma}Q(\tilde{x}_n)^{\frac{\beta}{8}-\frac{1}{4\gamma}}}}{\|C_0(\tilde{x}_n)^{-1}\|^2}\cdot e^{-3Q(\tilde{x}_n)^{\frac{\beta}{4}-\frac{1}{\gamma}}}\Big)\\
        &\leq \rho_0^2\Big(1-\frac{\Gamma Q(\tilde{x}_n)^{\frac{\beta}{8}-\frac{1}{4\gamma}}}{\|C_0(\tilde{x}_n)^{-1}\|^2}\cdot e^{-\frac{4}{\Gamma}Q(\tilde{x}_n)^{\frac{\beta}{8}-\frac{1}{4\gamma}}}\Big)\\
        &\leq \rho_0^2\Big(1-\frac{1}{2}\Gamma Q(\tilde{x}_n)^{\frac{\beta}{8}-\frac{1}{4\gamma}+\frac{1}{\gamma}}\Big) \text{ (since }Q(\tilde{x}_n)^{\frac{1}{\gamma}}\leq \|C_0(\tilde{x}_n)^{-1}\|^{-2}).
    \end{aligned}
\end{equation*}
Hence, $\frac{S^2(\hat{f}^{-1}(\tilde{x}),\xi)}{S^2(\hat{f}^{-1}(\tilde{x}_{n+1}),\zeta)}\leq \rho_0^2e^{-\frac{1}{2}\Gamma Q(\tilde{x}_n)^{\frac{\beta}{8}+\frac{3}{4\gamma}}}$. Similarly, $\frac{S^2(\hat{f}^{-1}(\tilde{x}),\xi)}{S^2(\hat{f}^{-1}(\tilde{x}_{n+1}),\zeta)}\geq \rho_0^{-2}e^{\frac{1}{2}\Gamma Q(\tilde{x}_n)^{\frac{\beta}{8}+\frac{3}{4\gamma}}}$.
 Combine inequality (\ref{decompose}) with inequality (\ref{right}),
\begin{equation*}
    \begin{aligned}
        &\quad \frac{S^2(\hat{f}^{-1}(\tilde{x}),\xi)}{S^2(\tilde{x}_n,\pi_{\tilde{x}_n}^s(\xi))}\leq \rho_0^2e^{Q(\tilde{x}_n)^{\frac{\beta}{4}-\frac{1}{\gamma}}-\frac{1}{2}\Gamma Q(\tilde{x}_n)^{\frac{\beta}{8}+\frac{3}{4\gamma}}}\\
        &\leq \rho_0^2e^{-\frac{1}{4}\Gamma Q(\tilde{x}_n)^{\frac{\beta}{8}+\frac{3}{4\gamma}}+(Q(\tilde{x}_n)^{\frac{\beta}{4}-\frac{1}{\gamma}}-\frac{1}{4}\Gamma Q(\tilde{x}_n)^{\frac{\beta}{8}+\frac{3}{4\gamma}})}\leq \rho_0^2e^{-\frac{1}{4}\Gamma Q(\tilde{x}_n)^{\frac{\beta}{8}+\frac{3}{4\gamma}}}.
    \end{aligned}
\end{equation*}
Similarly, $ \frac{S^2(\hat{f}^{-1}(\tilde{x}),\xi)}{S^2(\tilde{x}_n,\pi_{\tilde{x}_n}^s(\xi))}\geq \rho_0^{-2}e^{\frac{1}{4}\Gamma Q(\tilde{x}_n)^{\frac{\beta}{8}+\frac{3}{4\gamma}}}$.

Notice that $f^{-1}_{\tilde{x}_{n}}, f^{-1}_{\hat{f}^{-1}(\tilde{x}_{n+1})}$ coincide if $\psi_{\tilde{x}_{n}}^{p_0^s,p^u_0}\rightarrow\psi_{\tilde{x}_{n+1}}^{p^s_1,p^u_1}$. Analogously, the statement holds for unstable sets.
 \end{pf}
 \begin{rmk}
     In our context, the bound of the ratio $\frac{Q(\hat{f}(\tilde{x}))}{Q(\tilde{x})}$ is not sure since $\rho(\hat{f}(\tilde{x}))$ and $\rho(\tilde{x})$ are hard to compare. Therefore, we can not get the estimate for $|\Theta\zeta-\Theta\xi|$ by Lemma \ref{pi} directly.
 \end{rmk}

\begin{lem}\label{s1}
For any $I$-chains $(v_i)_{i\in\mathbb{Z}}=(\psi_{\tilde{x}_i}^{p_i^s,p_i^u})_{i\in\mathbb{Z}}$, $(w_i)_{i\in\mathbb{Z}}=(\psi_{\tilde{y}_i}^{q_i^s,q_i^u})_{i\in\mathbb{Z}}\in\Sigma^{\#}$ such that $\pi((v_i)_{i\in\mathbb{Z}})=\pi((w_i)_{i\in\mathbb{Z}})=\tilde{z}$, then $\forall i\in\mathbb{Z}$, and $\xi\in T_{\hat{f}^i(\tilde{z})}\tilde{V}^s_i$,
$$\frac{S(\tilde{x}_i,\pi_{\tilde{x}_i}^s(\xi))}{S(\tilde{y}_i,\pi_{\tilde{y}_i}^s(\xi))}=e^{\pm2\Gamma Q(\tilde{x}_i)^{\frac{\beta}{8}-\frac{1}{4\gamma}}}.$$
\end{lem}
\begin{pf}
    Decompose $\frac{S(\tilde{x}_i,\pi_{\tilde{x}_i}^s(\xi))}{S(\tilde{y}_i,\pi_{\tilde{y}_i}^s(\xi))}=\frac{S(\tilde{x}_i,\pi_{\tilde{x}_i}^s(\xi))}{S(\hat{f}^i(\tilde{z}),\xi)}\cdot \frac{S(\hat{f}^i(\tilde{z}),\xi)}{S(\tilde{y}_i,\pi_{\tilde{y}_i}^s(\xi))}$. Next, we estimate $\frac{S(\hat{f}^i(\tilde{z}),\xi)}{S(\tilde{x}_i,\pi_{\tilde{x}_i}^s(\xi))}$: w.l.o.g., assume $v_{n_k}=v$ for infinitely many $n_k>0$, where $v=\psi_{\tilde{x}_{j_0}}^{p_{j_0}^s,p_{j_0}^u}$. By inequality (\ref{S}), there is a constant $C_v>1$ which depends only on $v$ s.t. $s(\hat{f}^{n_k}(\tilde{z}))<C_v$. Let $k_v:=\Big[\frac{log C_v}{\frac{1}{2}\Gamma Q(\tilde{x}_{j_0})^{\frac{\beta}{8}+\frac{3}{4\gamma}}}\Big]+1$, and $\xi\in T_{\hat{f}^{n_{k_v}}(\tilde{z})}\tilde{V}^s_{n_{k_v}}(1)$. By Lemma \ref{S1}, there exists $j\in\{0,\cdots,n_{k_v}\}$ s.t.
    $$\frac{S(\hat{f}^j(\tilde{z}),d_{\hat{f}^{n_{k_v}}(\tilde{z})}\hat{f}^{-(n_{k_v}-j)}(\xi))}{S^(\tilde{x}_j,\pi_{\tilde{x}_j}^s(d_{\hat{f}^{n_{k_v}}(\tilde{z})}\hat{f}^{-(n_{k_v}-j)}(\xi))}=e^{\pm\Gamma Q(\tilde{x}_j)^{\frac{\beta}{8}-\frac{1}{4\gamma}}}.$$
    Otherwise, the ratio is improved by $e^{k_v\cdot \frac{1}{2}\Gamma Q(\tilde{x}_{j_0})^{\frac{\beta}{8}+\frac{3}{4\gamma}}}\geq C_v$, and so the ratio is greater than $C_v$; this is with the initial ratio smaller than $C_v$. Pulling back $j$ to 0, the ratio is improved or remained by $\epsilon e^{\pm Q(\tilde{x}_j)^{\frac{\beta}{4}-\frac{1}{\gamma}}}$ by the proof of the inequality (\ref{right}). For all $\xi'\in T_{\tilde{z}}\tilde{V}^s_0(1)$, $\exists \xi\in T_{\hat{f}^{n_{k_v}}(\tilde{z})}\tilde{V}^s_{n_{k_v}}$ such that $\xi=d_{\tilde{z}}\hat{f}^{n_{k_v}}(\xi')$, then
    $$\frac{S(\tilde{z},\xi')}{S(\tilde{x}_0,\pi_{\tilde{x}_0}^s\xi')}
    =e^{\pm\max\big\{\epsilon Q(\tilde{x}_0)^{\frac{\beta}{4}-\frac{1}{\gamma}},\Gamma Q(\tilde{x}_0)^{\frac{\beta}{8}-\frac{1}{4\gamma}}\big\}}
    =e^{\pm\Gamma Q(\tilde{x}_0)^{\frac{\beta}{8}-\frac{1}{4\gamma}}}.$$
    Similarly, $\frac{S(\tilde{z},\xi')}{S(\tilde{y}_0,\pi_{\tilde{y}_0}^s\xi')}=e^{\pm\Gamma Q(\tilde{x}_0)^{\frac{\beta}{8}-\frac{1}{4\gamma}}}$. Hence, the conclusion holds for $i=0$. Applying the process for $I$-chains $(v_{i+k})_{i\in\mathbb{Z}}$ and $(w_{i+k})_{i\in\mathbb{Z}}$, we obtain the same statements for $k\in\mathbb{Z}\backslash\{0\}$.
\end{pf}

Define $P_i:=P_i^s+P_i^u:=\pi_{\tilde{y}_i}^s\circ (\pi_{\tilde{x}_i}^s)^{-1}+\pi_{\tilde{y}_i}^u\circ (\pi_{\tilde{x}_i}^u)^{-1}: T_{\tilde{x}_i}M^f\rightarrow T_{\tilde{y}_i}M^f$, and it satisfies the following lemma,   for the proof see \cite{Ovadia2}.

\begin{lem}\label{P}
    $\|P_i\|, \|P_i^{-1}\| \leq e^{Q(\tilde{x}_i)^{\frac{\beta}{4}-\frac{1}{\gamma}}}$, for all $i\in\mathbb{Z}$.
\end{lem}

\begin{prop}\label{C3}
For any $I$-chains $(v_i)_{i\in\mathbb{Z}}=(\psi_{\tilde{x}_i}^{p_i^s,p_i^u})_{i\in\mathbb{Z}}$, $(w_i)_{i\in\mathbb{Z}}=(\psi_{\tilde{y}_i}^{q_i^s,q_i^u})_{i\in\mathbb{Z}}\in\Sigma^{\#}$ such that $\pi((v_i)_{i\in\mathbb{Z}})=\pi((w_i)_{i\in\mathbb{Z}})=\tilde{z}$, then $\forall i\in\mathbb{Z}$,
$$\frac{\|C_0^{-1}(\tilde{x}_i)\|}{\|C_0^{-1}(\tilde{y}_i)\|}=e^{\pm(2\Gamma+1) Q(\tilde{x}_i)^{\min\{\frac{\beta}{2}-\frac{1}{\gamma},\frac{\beta}{4}-\frac{1}{4\gamma}\}}}.$$
\end{prop}
\begin{pf}
According to Lemma \ref{C}(1), $\frac{|C_0^{-1}(\tilde{x}_i)(\xi)|}{|C_0^{-1}(\tilde{y}_i)(P_i(\xi))|}=\sqrt{\frac{S^2(\tilde{x}_i,\xi_s)+U^2(\tilde{x}_i,\xi_u)}{S^2(\tilde{y}_i,P_i(\xi_s))+U^2(\tilde{y}_i,P_i(\xi_u))}}$ where $\xi:=\xi^s+\xi^u\in E_{\tilde{x}_i}^s\oplus E_{\tilde{x}_i}^u$. And $S^2(\tilde{x}_i,\xi_s)=e^{\pm4\Gamma Q(\tilde{x}_i)^{\frac{\beta}{8}-\frac{1}{4\gamma}}}S^2(\tilde{y}_i,P_i(\xi_s))$, $U^2(\tilde{x}_i,\xi_u)=\ e^{\pm4\Gamma Q(\tilde{x}_i)^{\frac{\beta}{8}-\frac{1}{4\gamma}}}$
$U^2(\tilde{y}_i,P_i(\xi_u))$ by Lemma \ref{s1}. Then $\frac{|C_0^{-1}(\tilde{x}_i)(\xi)|}{|C_0^{-1}(\tilde{y}_i)(P_i(\xi))|}=e^{\pm2\Gamma Q(\tilde{x}_i)^{\frac{\beta}{8}-\frac{1}{4\gamma}}}$. By Lemma \ref{P}, we obtain
$$\frac{\|C_0^{-1}(\tilde{x}_i)\|}{\|C_0^{-1}(\tilde{y}_i)\|}=e^{\pm(2\Gamma+1) Q(\tilde{x}_i)^{\frac{\beta}{8}-\frac{1}{4\gamma}}}.$$
\end{pf}
\begin{enumerate}[resume=myenum]
    \item Compare $Q(\tilde{x}_i)$ with $Q(\tilde{y}_i)$:
\end{enumerate}
Recall that $Q_{\epsilon}(\tilde{x}) := \max\{q \in \mathcal{I} \,|\, q \leq \tilde{Q}_{\epsilon}(\tilde{x})\}$, where
$\tilde{Q}_{\epsilon}(\tilde{x}) := \epsilon^{\frac{20}{\beta}}\rho(\tilde{x})^{\frac{8a}{\beta}}\cdot\|C_0(\tilde{x})^{-1} \|^{-2\gamma}$.  By Proposition \ref{C3}, $\frac{\|C_0^{-1}(\tilde{x}_i)\|^{-2\gamma}}{\|C_0^{-1}(\tilde{y}_i)\|^{-2\gamma}}=e^{\pm 2\gamma(2\Gamma+1) Q(\tilde{x}_i)^{\frac{\beta}{8}-\frac{1}{4\gamma}}}$. And  $\frac{\rho(\tilde{x}_i)^{\frac{8a}{\beta}}}{\rho(\tilde{y}_i)^{\frac{8a}{\beta}}}==e^{\pm\frac{16a\epsilon}{\beta}}$ by the first part \ref{1.} of this subsection.
Then $\forall i\in\mathbb{Z}$,
   $$\frac{\tilde{Q}(\tilde{x}_i)}{\tilde{Q }(\tilde{y}_i)}=e^{\pm\big( 2\gamma(2\Gamma+1) Q(\tilde{x}_i)^{\frac{\beta}{8}-\frac{1}{4\gamma}}+\frac{16a\epsilon}{\beta}\big)},$$
and so $\frac{Q(\tilde{x}_i)}{Q(\tilde{y}_i)}\in\Big[I^{-\frac{1}{4}}(1)e^{-\big( 2\gamma(2\Gamma+1) Q(\tilde{x}_i)^{\frac{\beta}{8}-\frac{1}{4\gamma}}+\frac{16a\epsilon}{\beta}\big)},I^{\frac{1}{4}}(1)e^{ \big( 2\gamma(2\Gamma+1) Q(\tilde{x}_i)^{\frac{\beta}{8}-\frac{1}{4\gamma}}+\frac{16a\epsilon}{\beta}\big)}\Big]$ since $I^{-\frac{1}{4}}(1)\tilde{Q}(\tilde{x})\leq Q(\tilde{x})\leq Q(\tilde{x})$. Furthermore, we have $Q(\tilde{y}_i)=I^{\pm1}(Q(\tilde{x}_i))$ for small enough $\epsilon$ (since $\frac{\beta}{8}-\frac{1}{4\gamma}-\frac{1}{\gamma}>0)$.

\begin{claim}\label{=}
    $\pi(\Sigma^{\#})=RST$.
\end{claim}
\begin{pf}
    By Theorem \ref{main1}(3), $\pi(\Sigma^{\#})\supset RST$. Next, we will prove $\pi(\Sigma^{\#})\subset RST$: For any $(u_i)_{i\in\mathbb{Z}}\in\Sigma^{\#}$, where $u_i=\psi_{\tilde{x}_i}^{p^s_i,p^u_i}$, then $\limsup_{i\rightarrow \pm\infty}(p^s_i\land p^u_i)>0$. Let $\tilde{z}:=\pi((u_i)_{i\in\mathbb{Z}})$,
    then $Q(\hat{f}^i(\tilde{z}))=I^{\pm1}(Q(\tilde{x}_i))$ by the proof of Lemma \ref{s1} and Proposition \ref{C3}. And so $I^{-1}(p_i^s\land p^u_i)\leq I^{-1}(Q(\tilde{x}_i))\leq Q(\hat{f}^i(\tilde{z}))$. Set the map $q(\cdot)$ of point $\tilde{z}$ by $q(\hat{f}^i(\tilde{z}))=I^{-1}(p_i^s\land p^u_i)$, then
\begin{enumerate}[label=(\roman*)]
    \item  $q\circ\hat{f}=I^{\pm1}(q)$: $\forall i\in \mathbb{Z}$, $q\circ\hat{f}(\hat{f}^i(\tilde{z}))=q(\hat{f}^{i+1}(\tilde{z}))=I^{-1}(p^s_{i+1}\land p^u_{i+1})=I^{-1}(I^{\pm1}(p_i^s\land p^u_i))=I^{\pm1}(I^{-1}(p_i^s\land p^u_i))=I^{\pm1}(q(\hat{f}^i(\tilde{z})))$;
    \item $q(\hat{f}^i(\tilde{z}))=I^{-1}(p_i^s\land p^u_i)\leq I^{-1}(Q(\tilde{x}_i))\leq Q(\hat{f}^i(\tilde{z}))$;
    \item $\limsup_{i\rightarrow \pm\infty} I^{-1}(p^s_i\land p^u_i)\geq \limsup_{i\rightarrow \pm\infty}(p^s_{i}\land p^u_{i})e^{-\Gamma(p^s_{i}\land p^u_{i})^{\frac{1}{\gamma}}}>0$.
\end{enumerate}
Hence, the point $\tilde{z}\in RST$.
\end{pf}
\begin{enumerate}[resume=myenum]
    \item Compare $p^{s/u}(\tilde{x}_i)$ with $q^{s/u}(\tilde{y}_i)$:\label{4.}
\end{enumerate}
 Since $(v_i)_{i\in\mathbb{Z}}\in \Sigma^{\#}$, $\limsup_{i\rightarrow \pm\infty}(p_i^s\land p^u_i)>0$. Then there exist $l_1>0$ and $l_2<0$ such that $p^s_{l_1}=Q(\tilde{x}_{l_1})$ and $p^u_{l_2}=Q(\tilde{x}_{l_2})$ by Lemma \ref{limsup}. Since $Q(\tilde{x}_{l_1})\geq I^{-1}(Q(\tilde{y}_{l_1}))\geq I^{-1}(q_{l_1}^s)$ and $Q(\tilde{x}_{l_2})\geq I^{-1}(Q(\tilde{y}_{l_2}))\geq I^{-1}(q_{l_2}^u)$, then $p^s_{l_1}\geq I^{-1}(q_{l_1}^s)$ and $p^u_{l_2}\geq I^{-1}(q_{l_2}^u)$. By definition
\begin{equation*}
     \begin{aligned}
         p_{l_1-1}^s&=\min\big\{I(p_{l_1}^s),Q(\tilde{x}_{l_1-1})\big\}\geq \min\big\{I(I^{-1}(q_{l_1}^s)),I^{-1}(Q(\tilde{y}_{l_1-1}))\big\}\\
         &=I^{-1}\min\big\{I(q_{l_1}^s),Q(\tilde{y}_{l_1-1})\big\}=I^{-1}(q^s_{l_1-1}),
     \end{aligned}
 \end{equation*}
 and
 \begin{equation*}
     \begin{aligned}
         p_{l_2+1}^u&=\min\big\{I(p_{l_2}^u),Q(\tilde{x}_{l_2+1})\big\}\geq \min\big\{I(I^{-1}(q_{l_2}^u)),I^{-1}(Q(\tilde{y}_{l_2+1}))\big\}\\
         &=I^{-1}\min\big\{I(q_{l_2}^u),Q(\tilde{y}_{l_2+1})\big\}=I^{-1}(q^u_{l_2+1}).
     \end{aligned}
 \end{equation*}
By induction, we have $p_0^s\geq I^{-1}(q_0^s)$ and $p_0^u\geq I^{-1}(q_0^u)$. Symmetrically, there is the same statement for $(w_i)_{i\in\mathbb{Z}}$: $q_0^s\geq I^{-1}(p_0^s)$ and $q_0^u\geq I^{-1}(p_0^u)$. Hence, $p_0^s=I^{\pm}(q_0^s)$ and $p_0^u=I^{\pm}(q_0^u)$. Applying this to $(\psi_{\tilde{x}_{i+k}}^{p_{i+k}^s,p_{i+k}^u})_{k\in\mathbb{Z}}, (\psi_{\tilde{y}_{i+k}}^{q_{i+k}^s,p_{i+k}^u})_{k\in\mathbb{Z}}$, we have $p_i^s=I^{\pm1}(q_i^s)$, $p_i^u=I^{\pm1}(q_i^u)\ \forall i\in\mathbb{Z}$.

The following theorem summarizes the conclusions of this subsection. The proof of the last item is similar to that in \cite{Ovadia1}.
\begin{thm}\label{inverse}
For any $I$-chains $(v_i)_{i\in\mathbb{Z}}=(\psi_{\tilde{x}_i}^{p_i^s,p_i^u})_{i\in\mathbb{Z}}$, $(w_i)_{i\in\mathbb{Z}}=(\psi_{\tilde{y}_i}^{q_i^s,q_i^u})_{i\in\mathbb{Z}}\in\Sigma^{\#}$ such that $\pi((v_i)_{i\in\mathbb{Z}})=\pi((w_i)_{i\in\mathbb{Z}})=\tilde{z}$, then $\forall i\in\mathbb{Z}$,
\begin{enumerate}[label=(\arabic*)]
    \item $d(p(\tilde{x}_i),p(\tilde{y}_i))\leq 25^{-1}\max\big\{(p_i^s\land p_i^u)^2,(q_i^s\land q_i^u)^2\big\}$.
    \item $\frac{\|C_0^{-1}(\tilde{x}_i)\|}{\|C_0^{-1}(\tilde{y}_i)\|}=e^{\pm(2\Gamma+1) Q(\tilde{x}_i)^{\frac{\beta}{8}-\frac{1}{4\gamma}}}$.
    \item $Q(\tilde{y}_i)=I^{\pm1}(Q(\tilde{x}_i))$.
    \item $p_i^s=I^{\pm1}(q_i^s)$, $p_i^u=I^{\pm1}(q_i^u)$.
    \item  $(\psi_{\tilde{y}_n}^{-1}\circ \psi_{\tilde{x}_n})(u)= O_n(u)+a_n+\Delta_n(u)\ \forall u\in B_{\epsilon}(0)$, where $a_n$ is a constant vector in $\mathbb{R}^d$ satisfies $|a_n|\leq 10^{-1}(q_n^s\land q^u_n)$, $O_n$ is an orthogonal matrix preserving $\mathbb{R}^{d_s(\tilde{x}_0)}$ and $\mathbb{R}^{d-d_s(\tilde{x}_0)}$ , and $\Delta_n: B_{\epsilon}(0)\rightarrow \mathbb{R}^d$ satisfies $\Delta_n(0)=0$, $\|d_u\Delta_n\|\leq \epsilon^{\frac{1}{3}}$.
\end{enumerate}
\end{thm}
\subsection{\texorpdfstring{Locally finite Markov Partition}{Locally finite Markov Partition}}
Let
$$Z(v)=\{\pi((v_i)_{i\in\mathbb{Z}}):(v_i)_{i\in\mathbb{Z}}\in\Sigma^{\#}\text{ and }v_0=v\},$$
then $\mathcal{Z}:=\{Z(v):v\in\mathcal{V}\}$ is a cover of the set $RST$ since $\pi(\Sigma^{\#})=RST$. $\mathcal{Z}$ is called \textit{Markov cover}. For any $\tilde{x}\in Z:=Z(v)$, there exists $(v_i)_{i\in\mathbb{Z}}\in\Sigma^{\#}$ s.t. $v_0=v$ and $\pi((v_i)_{i\in\mathbb{Z}})=\tilde{x}$. Denote $\tilde{V}^s(\tilde{x},Z):=\tilde{V}^s((v_i)_{i\geq 0})$, $\tilde{V}^u(\tilde{x},Z):=\tilde{V}^u((v_i)_{i\leq 0})$; and $\tilde{W}^s(\tilde{x},Z):=\tilde{V}^s(\tilde{x},Z)\bigcap Z$, $\tilde{W}^u(\tilde{x},Z):=\tilde{V}^u(\tilde{x},Z)\bigcap Z$. By the properties of stable/unstable sets, the following proposition shows the properties of $\tilde{W}^s(\tilde{x},Z)$, $\tilde{W}^u(\tilde{x},Z)$.   This prepares for the construction of the locally finite Markov partition.
\begin{prop}\label{Z}
   For every $Z\in\mathcal{Z}$, assume $\tilde{x}=\pi((v_i)_{i\in\mathbb{Z}})\in Z$ where $(v_i)_{i\in\mathbb{Z}}\in\Sigma^{\#}$.
   The following   statements hold for small enough $\epsilon>0$:
\begin{enumerate}[label=(\arabic*)]
    \item (Locally finite) $|\{Z'\in\mathcal{Z}:Z'\bigcap Z\neq \emptyset\}|<\infty$.
    \item For every $\tilde{y}\in Z$, $\tilde{W}^s(\tilde{x},Z), \tilde{W}^s(\tilde{y},Z)$ are either equal or disjoint. Similarly for $\tilde{W}^u(\tilde{x},Z), \tilde{W}^u(\tilde{y},Z)$.
    \item For every $\tilde{y}\in Z$, there exists a unique element $\tilde{z}$ such that $\tilde{W}^s(\tilde{x},Z)\bigcap \tilde{W}^u(\tilde{y},Z)=\{\tilde{z}\}$.
    \item (Markov property) $\hat{f}(\tilde{W}^s(\tilde{x},Z(v_0)))\subset \tilde{W}^s(\hat{f}(\tilde{x}),Z(v_1))\text{ and } \hat{f}^{-1}(\tilde{W}^u(\hat{f}(\tilde{x}),Z(v_1)))\subset \tilde{W}^u(\tilde{x},Z(v_0)).$
    \item There exists a map $[\cdot,\cdot]_{Z}: \tilde{W}^s(\cdot,Z)\times \tilde{W}^u(\cdot,Z)$. If $\tilde{y}\in Z(v_0),\hat{f}(\tilde{y})\in Z(v_1)$, then $\hat{f}([\tilde{x},\tilde{y}]_{Z(v_0)})=[\hat{f}(\tilde{x}),\hat{f}(\tilde{y})]_{Z(v_1)}$.
    \item (Uniform contraction/expansion) If $\tilde{y}\in \tilde{W}^s(\tilde{x},Z), \tilde{z}\in \tilde{W}^u(\tilde{x},Z)$, then $$d(\hat{f}^{n}(\tilde{x}),\hat{f}^{n}(\tilde{y}))\leq 8I^{-n}(1),d(\hat{f}^{-n}(\tilde{x}),\hat{f}^{-n}(\tilde{z}))\leq 8I^{-n}(1)\forall n\in\mathbb{N}.$$
    \item If $Z'\in\mathcal{Z}$ with $Z\bigcap Z'\neq\emptyset$, where $Z=Z(\psi_{\tilde{x}_0}^{p_0^s,p_0^u}), Z'=Z'(\psi_{\tilde{y}_0}^{q_0^s,q_0^u})$. Then
\begin{enumerate}
    \item $p(Z)\subset \psi_{\tilde{y}_0}[B_{q_0^s\land q_0^u}(0)]$;
    \item If $\tilde{x}\in Z\bigcap Z'$, then $\tilde{W}^s(\tilde{x},Z)\subset \tilde{V}^s(\tilde{x},Z')$ and $\tilde{W}^u(\tilde{x},Z)\subset \tilde{V}^u(\tilde{x},Z')$;
    \item If $\tilde{y}\in Z'$ and $\tilde{x}\in Z\bigcap Z'$, then there exists a unique element $\tilde{o}$ such that $\tilde{V}^s(\tilde{x},Z)\bigcap \tilde{V}^u(\tilde{y},Z' )=\{\tilde{o}\}$.
\end{enumerate}
\end{enumerate}
\end{prop}
\begin{pf}
    (1) Assume $Z=Z(\psi_{\tilde{x}_0}^{p_0^s,p_0^u}), Z'=Z'(\psi_{\tilde{y}_0}^{q_0^s,q_0^u})$, and $\tilde{x}=\pi((v_i)_{i\in\mathbb{Z}})=\pi((w_i)_{i\in\mathbb{Z}})\in Z'\bigcap Z$ where $v_0=\psi_{\tilde{x}_0}^{p_0^s,p_0^u}, w_0=\psi_{\tilde{y}_0}^{q_0^s,q_0^u}$. By Theorem \ref{inverse}(4), $q_0^s\geq I^{-1}(p_0^s)$ and $q_0^u\geq I^{-1}(p_0^u)$. And so there are only finitely many double charts $\psi_{\tilde{y}_0}^{q_0^s,q_0^u}$ in $\mathcal{V}$ such that $\tilde{x}=\pi((v_i)_{i\in\mathbb{Z}})=\pi((w_i)_{i\in\mathbb{Z}})\in Z'\bigcap Z$ by discreteness.

    (2) By Proposition \ref{tilde}(5), the stable/unstable set $\tilde{V}^{s/u}(\tilde{x},Z),\tilde{V}^{s/u}(\tilde{y},Z)$ are either equal or disjoint. Then it holds for $\tilde{W}^{s/u}(\tilde{x},Z),\tilde{W}^{s/u}(\tilde{y},Z)$ by definition.

    (3) By Proposition \ref{tilde}(1), the stable set $\tilde{V}^s(\tilde{x},Z)\in\tilde{M}^s(v_0)$ and the unstable set $\tilde{V}^u(\tilde{x},Z)\in\tilde{M}^u(v_0)$ intersect at a unique point $\tilde{z}$. (The construction of the point $\tilde{z}\in Z$ can be found in \cite{Sarig}. Hence, $\tilde{W}^s(\tilde{x},Z)\bigcap \tilde{W}^u(\tilde{y},Z)=\{\tilde{z}\}$.

    (4) Using $\pi\circ \sigma=\hat{f}\circ \pi$ and Proposition \ref{tilde}(3), part (4)can be proved as the proof of proposition 10.9 in \cite{Sarig}.

    (5) The map $[\cdot,\cdot]_{Z}: \tilde{W}^s(\cdot,Z)\times \tilde{W}^u(\cdot,Z)$ is well-defined by part (3). And it is called the \textit{Smale bracket}. By Proposition \ref{tilde}(6), part (4) can be proved as the proof of lemma 10.7 in \cite{Sarig}.

    (6) Using Proposition \ref{ad}(4), part (6) holds.

    (7) (a) By Proposition \ref{ad}(1), $p(Z)\subset \psi_{\tilde{x}_0}[B_{10^{-2}(p^s_0\land p^u_0)^2}(0)]$. Then
    $$p(Z)\subset \psi_{\tilde{y}_0}\circ(\psi_{\tilde{y}_0}^{-1}\circ\psi_{\tilde{x}_0})[B_{10^{-2}(p^s_0\land p^u_0)^2}(0)]\subset \psi_{\tilde{y}_0}[B_{q_0^s\land q_0^u}(0)],$$ by Theorem \ref{inverse}(5).

    (b) The key of the proof is to prove $\tilde{V}^{s/u}(\tilde{x},Z)=\tilde{V}^{s/u}(\tilde{x},Z')$ by the equivalent definition of stable/unstable set.

    (c) Using Theorem \ref{inverse}(5) and graph transform, $V^s=p(\tilde{V}^s(\tilde{x},Z))$ can be expressed as the graph in $\psi_{\tilde{y}_0}$. Assume $V^s=\psi_{\tilde{y}_0}\{(\tilde{w},\tilde{G}(\tilde{w})):\tilde{w} \text{ in some area}\}$, $\tilde{G}$ is  contractive by Theorem \ref{inverse}(5). Then we obtain $\tilde{V}^s(\tilde{x},Z)$ and $\tilde{V}^s(\tilde{y},Z')$ intersect in a single point by the fixed point theorem and the Kirszbraun theorem.
    \end{pf}

Using the method first developed by Sina$\check{\text{\i}}$ and Bowen \cite{Bowen1, Bowen2, Bowen}, there is a locally finite Markov partition $\mathcal{R}$ defined by the following equivalence relation on RST:
$$\tilde{x}=\tilde{y}\iff \forall Z,Z'\in \mathcal{Z}:
\Bigg\{
  \begin{matrix}
    \tilde{x}\in Z \iff\tilde{y}\in Z \\
    \tilde{W}^s(\tilde{x},Z)\bigcap Z'\neq \emptyset\iff \tilde{W}^s(\tilde{y},Z)\bigcap Z'\neq \emptyset\\
    \tilde{W}^u(\tilde{x},Z)\bigcap Z'\neq \emptyset\iff \tilde{W}^u(\tilde{y},Z)\bigcap Z'\neq \emptyset.
  \end{matrix}
$$
By Proposition \ref{Z}, the partition $\mathcal{R}$ satisfies:

(1) Product structure.

(2) Markov property.

(3) $Z\in\mathcal{Z}$ contains finitely many $R\in\mathcal{R}$, $\forall Z\in\mathcal{Z}$.

Let $\tilde{\mathcal{G}}$ be a countable directed graph with vertices $\tilde{\mathcal{V}}=\mathcal{R}$ and edges $\mathcal{E}:=\{(R,S)\in\mathcal{R}\times \mathcal{R}: R\bigcap \hat{f}^{-1}(S)\neq \emptyset\}.$ Denote $\tilde{\Sigma}:=\tilde{\Sigma}(\tilde{\mathcal{G}})$ as the topological Markov shift associated with $\tilde{\mathcal{G}}$. We claim that every vertex of $\tilde{\mathcal{G}}$ has finite degree, i.e. $|\{S\in\mathcal{R}:R\rightarrow S\}|<\infty$: assume $R\rightarrow S$ and $R\subset Z(u)$, where $u\in\mathcal{V}$. And $R\bigcap \hat{f}^{-1}(S)\neq \emptyset \iff \hat{f}(R)\bigcap S\neq \empty$ since $\hat{f}$ is homeomorphism. If $\tilde{x}\in \hat{f}(R)\bigcap S$, there exists $v\in\mathcal{V}$ s.t. $u\rightarrow v$ and $S\subset Z(v)$. Notice that $Z\in\mathcal{Z}$ contains finitely many $R\in\mathcal{R}$, $\forall Z\in\mathcal{Z}$, and $\mathcal{V}$ has finite degree. Then the claim holds. Furthermore, $\tilde{\Sigma}$ is locally compact.

The following lemma can be found in \cite{Lima}.
\begin{lem}\label{Z1}
\begin{enumerate}[label=(\arabic*)]
    \item If $\uline{R}=(R_n)_{n\in\mathbb{Z}}\in\tilde{\Sigma}$ and $R_0\subset Z(v)$, there exists $(v_n)_{n\in\mathbb{Z}}\in \Sigma$ with $v_0=v$ s.t. $\bigcap_{i=m}^n\hat{f}^{-i}(R_i)\subset \bigcap_{i=m}^n\hat{f}^{-i}(Z(v_i))\ \forall m\leq n$. And if $\uline{R}\in\tilde{\Sigma}^{\#}$, then $(v_n)_{n\in\mathbb{Z}}\in \Sigma^{\#}$.
    \item (Bowen property) If $\uline{R}=(R_n)_{n\in\mathbb{Z}},\uline{S}=(S_n)_{n\in\mathbb{Z}}\in\tilde{\Sigma}$, then $\tilde{\pi}(\uline{R})=\tilde{\pi}(\uline{S})$ iff $R_n\sim S_n\ \forall n\in\mathbb{Z}$, where $R_n\sim S_n\iff \exists Z_n,Z_n'\in\mathcal{Z}$ s.t. $R_n\subset Z_n, S_n\subset Z_n'$ and $Z_n\bigcap Z_n'\neq \emptyset$.
\end{enumerate}
\end{lem}
Consider a path $R_m\rightarrow R_{m+1}\rightarrow \cdots\rightarrow R_{n}$ on $\tilde{\mathcal{G}}(m<n)$, then $\bigcap_{i=m}^n\hat{f}^{-1}(R_i)\neq \emptyset$ by the Markov property of $\mathcal{R}$, and the diameter of $\bigcap_{i=-n}^n\hat{f}^{-i}(R_i)\leq 16I^{-n}(1)$, then the diameter converging to zero. Define $\tilde{\pi}:\tilde{\Sigma}\rightarrow M^f$ by $\{\tilde{\pi}(\uline{R})\}=\bigcap_{n\geq0}\overline{\bigcap_{i=-n}^n\hat{f}^{-i}(R_i)}$.
 For $R\in\mathcal{R}$, define
$$N(R):=|\{(R',v')\in \mathcal{R}\times \mathcal{A}:R'\sim R\text{ and }R'\subset Z(v')\}|.$$

 Now we can give the Main Theorem.
\begin{thm}\label{main}
    For small enough $\epsilon>0$, the map $\tilde{\pi}$ satisfies:

    (1) $\tilde{\pi}(\tilde{\Sigma}^{\#})=RST$.

    (2) If $\tilde{x}=\tilde{\pi}(\uline{R})$, where $R_n=R$ for infinitely many $n<0$ and $R_n=S$ for infinitely many $n>0$, then $|\tilde{\pi}^{-1}(\tilde{x})\bigcap \tilde{\Sigma}^{\#}|\leq N(R)\cdot N(S)$ ($\tilde{\pi}$ is a finite-to-one map).

    (3) $\tilde{\pi}\circ \sigma=\hat{f}\circ \tilde{\pi}$.

    (4) $\tilde{\pi}$ is uniformly continuous on $\tilde{\Sigma}$.
\end{thm}
\begin{pf}
    (1) $RST\subset \tilde{\pi}(\tilde{\Sigma}^{\#})$: if $\tilde{x}\in RST$, $\exists (v_i)_{i\in\mathbb{Z}}\in \Sigma^{\#}$ s.t. $\pi((v_i)_{i\in\mathbb{Z}})=\tilde{x}$ by Theorem \ref{main1}(3). By the construction of $\mathcal{R}$, there exists $R_i\in\mathcal{R}$, s.t. $R_i\subset Z(v_i)\ \forall i\in\mathbb{Z}$ and $\tilde{\pi}((R_i)_{i\in\mathbb{Z}})=\tilde{x}$, and $(R_i)_{i\in\mathbb{Z}}\in \tilde{\Sigma}^{\#}$ since $Z\in\mathcal{Z}$ contains finitely many $R\in\mathcal{R}$, $\forall Z\in\mathcal{Z}$. And
    $\tilde{\pi}(\tilde{\Sigma}^{\#})\subset RST$: by claim \ref{=}, $RST=\pi(\Sigma^{\#})$. And $\tilde{\pi}(\tilde{\Sigma}^{\#})\subset \pi(\Sigma^{\#})$ by Lemma \ref{Z1}(1). Hence, $\tilde{\pi}(\tilde{\Sigma}^{\#})=\pi(\Sigma^{\#})=RST$.

    (2) The proof is   the same as the proof of Theorem 7.6 in \cite{Lima}.

    (3) By Lemma \ref{Z1}(1) and Theorem \ref{main1}(1), part (3) holds.

    (4) Assume $\uline{R}=(R_i)_{i\in\mathbb{Z}},\uline{S}=(S_i)_{i\in\mathbb{Z}}\in\tilde{\Sigma}^{\#}$ with $d(\uline{R},\uline{S})=e^{-n}$, then $R_i=S_i,\forall |i|\leq n$. Hence, $d(\tilde{\pi}(\uline{R}),\tilde{\pi}(\uline{S}))\leq diam(\bigcap_{i=n}^{-n}\hat{f}^{-i}(R_i))\leq 16I^{-n}(1)=16I^{\log d(\uline{R},\uline{S})}(1)$.
\end{pf}

\begin{prop}
    Assume $\tilde{x}\in \tilde{\pi}(\tilde{\Sigma}^{\#})$, then there exists a splitting $T_{\tilde{x}}M^f=E^s({\tilde{x}})\oplus E^u(\tilde{x})$, where
    $$\sup_{\substack{\xi^s \in E^s_{\tilde{x}},\ |\xi^s|=1}} \sum_{m \geq 0} |d\hat{f}^m \xi_s|^2 < \infty, \quad
\sup_{\substack{\xi^u \in E^u_{\tilde{x}},\ |\xi^u|=1}} \sum_{m \geq 0} |d\hat{f}^{-m} \xi_u|^2<\infty.$$
And the maps $\uline{R}\in\tilde{\Sigma}^{\#}\rightarrow E^{s/u}(\tilde{\pi}(\uline{R}))$ have summble variations.
\end{prop}
\begin{pf}
    By Theorem \ref{main}(1), $\tilde{x}\in\tilde{\pi}(\tilde{\Sigma}^{\#})=RST$. Then the decomposition of $T_{\tilde{x}}M^f$ exists. Since $\pi(\Sigma^{\#})=\tilde{\pi}(\tilde{\Sigma}^{\#})$, $\exists (u_i)_{i\in\mathbb{Z}}\in\Sigma^{\#}$ s.t. $\pi((u_i)_{i\in\mathbb{Z}})=\tilde{x}$. Then $E^s(\tilde{x})= T_{\tilde{x}}\tilde{V}^s((u_i)_{i\geq0})$ and $E^u(\tilde{x})= T_{\tilde{x}}\tilde{V}^u((u_i)_{i\leq0})$ by the uniqueness. By Proposition \ref{shadow}(5), the rest part holds.
\end{pf}

\section{Applications and Example}\label{example}
\subsection{\texorpdfstring{The growth rate of $P_n(f)$ and the number of M.M.E.}{The growth rate of Pn(f) and the number of M.M.E.}}
For any $f$-invariant probability measure $\mu$, we require that system $(M,f,\mu)$ satisfies the following integrability condition
$$\log d(x,\Gamma_{\infty})\in L^1(\mu).$$
If the $f$-invariant probability measure $\mu$ satisfies the integrability condition, $\mu$ is said to be $f$-adapted. By the integrability condition, $\mu(\Gamma_{\infty})=0$. By the definition, it is easy to see that $\mu$ is $f$-adapted iff $\log\rho\in L^1(\tilde{\mu})$ where $\tilde{\mu}$ is the $\hat{f}$-invariant Borel probability measure on $M^f$ such that $p\tilde{\mu}=\mu$.

Denote $P_n(f)$ as the number of $n$-period points, and $h_{top}(f,0\text{-summ})$ is the topological entropy of $f|_{0\text{-summ}}$.
\begin{thm}
    Suppose $f$ is a $C^{1+\beta}$ endomorphism with singularities and satisfies the assumptions \ref{A1}-\ref{A3}.
\begin{enumerate}[label=(\arabic*)]
\item  If $f$ has a measure of maximal entropy on $0-$summ, then there exists $p\in\mathbb{N}$ such that
$$\liminf_{n\rightarrow \infty,p|n}e^{-nh_{top}(f,0\text{-summ})}P_n(f)>0,\forall n\in \mathbb{N}.$$
\item $f|_{0\text{-summ}}$ possesses at most countably many ergodic M.M.E. where M.M.E. is the maximal entropy measure.
\end{enumerate}
\end{thm}
\begin{pf}
According to the Theorem \ref{main}, the proof is similar to   that of Theorem 1.1 and 1.2 of \cite{Sarig}.
\end{pf}
\subsection{\texorpdfstring{An example}{An Example}}
In this section, we will construct an example that satisfies the conditions of this context. Consider the linear map $A=\begin{pmatrix}
    3 & 1\\
    1 & 1
\end{pmatrix}:\mathbb{R}^2\rightarrow\mathbb{R}^2$, and denote $E^u,E^s$ as the subspace of the eigenvalues $\lambda_u=2+\sqrt{2},\lambda_s=2-\sqrt{2}$, respectively. Let $T:\mathbb{T}^2\rightarrow\mathbb{T}^2$ be an endomorphism on $\mathbb{T}^2$ that satisfies $p\circ A=T\circ p$, where $p$ is the canonical projection from $\mathbb{R}^2$ to $\mathbb{T}^2$. Obviously, $p_1:=(0,0)$ is a fixed point. Next, we will modify the endomorphism $T$ by the following steps:
\begin{enumerate}[label=Step \arabic*,  labelwidth=3em, leftmargin=!]
\item: According to the technique of "slow-down"\cite{Katok}, there exists a $C^r$ map $T_1$ satisfies ($r\geq 2)$: $\exists$ two functions $K^s,K^u:\mathbb{T}^2\rightarrow\mathbb{R}$ s.t.
\begin{enumerate}[label=(\arabic*)]
\item $K^s(x)=2-\sqrt{2}$ is a constant function, and $|dT_1|_{E^s}|\leq K^s(x)\ \forall x\in\mathbb{T}^2$;
\item $K^u(\cdot)$ is continuous on $\mathbb{T}^2$, and $|dT_1|_{E^u}|\geq K^u(x)\ \forall x\in\mathbb{T}^2$;
\item $K^u(x)>1\ \forall x\in \mathbb{T}^2\backslash\{p_1\}$ and $K^u(p_1)=1$.
\end{enumerate}\label{step1}
\item: Let $p_2\in \mathbb{T}^2\backslash\{p_1\}$ be a period point of $T$. For arbitrary small $r_1>0$, there is a $C^r$ map $T_2$ such that $T_2(D_{r_1}(p_2))=\{p_2\}$ and $T_2(x)=T_1(x)$ in $\mathbb{T}^2\backslash D_{r_1}(p_2)$ where $D_{r_1}(p_2)$ is a ball at $p_2$ with the radius $r_1$, then $\det(d_{p_2}T_2)=0$, i.e. $p_2$ is a singular point of $T_2$.
\end{enumerate}

Set $\Gamma_{\infty}:=\{p_2\}$, assume the diameter of $\mathbb{T}^2$ is smaller than one by multiply the metric
with a small constant. By the above constructions, we can check that the map $T_2$ satisfies the conditions \ref{A1}-\ref{A3}: for $x\in\mathbb{T}^2$ with $x,T_2(x)\notin \Gamma_\infty$, given constants $a>1,K<\frac{1}{2}$ s.t. $\tau(x):= K\min\{d(x,\Gamma_\infty),d(x,\Gamma_{\infty})\}\geq \min\{d(x,\Gamma_\infty)^a,d(x,\Gamma_{\infty})^a\}$. For the neighborhoods $D_{x}:=B(x,2\tau(x))$ and $\mathfrak{E}(x):=B(T_2(x),2\tau(x))$,
\begin{itemize}
    \item (A1): $T_2|_{D(x)}$ and $T_{2,x}^{-1}|_{\mathfrak{E}(x)}$ are diffeomorphisms onto the image, where $T_{2,x}^{-1}$ is the inverse branch of $T_2$ that sends $T_2(x)$ to $x$.
    \item (A2): For any $x\in \mathbb{T}^2\backslash\{p_2\}$, $d(x,\Gamma_\infty)^a\leq 1$ and $d(x,\Gamma_\infty)^{-a}\geq1$. Since $\mathbb{T}^2$ is compact and $T_2$ is a $C^r$ map, then $\|dT_2\|$ is bounded on $\mathbb{T}^2$. And $|\det d_xT_2| = \|d_xT_2\| \cdot \|(d_xT_2)^{-1}\|^{-1}\ \forall x\notin\Gamma_\infty$. Hence, there is a constant $K'<1$ s.t. $K'd(x,\Gamma_\infty)^{a}\leq \|d_yT_2\|\leq \frac{1}{K'}d(x,\Gamma_{\infty})^{-a}\ \forall y\in D(x)$ and $K'd(x,\Gamma_\infty)^{a}\leq \|d_yT_{2,x}^{-1}\|\leq \frac{1}{K'}d(x,\Gamma_{\infty})^{-a}\ \forall y\in\mathfrak{E}(x)$.
    \item (A3): The maps $T_2|_{D(x)}$ and $T_{2,x}^{-1}|_{\mathfrak{E}(x)}$ are $C^r$ maps, then they satisfy the locally $H\ddot{o}lder$ property.
\end{itemize}
Denote $({\mathbb{T}^2}^{T},\hat{T})$ as the natural extension of $(\mathbb{T}^2,T_2)$. Ovadia proved in \cite{Ovadia2} that for almost Anosov diffeomorphisms $(M,f)$, the set $M\backslash \{p\}$ is the $0$-summ set where $p$ is the unique non-hyperbolic point. Hence, $\tilde{x}=\{x_n\}_{n\in\mathbb{N}}\in {\mathbb{T}^2}^{T}\backslash\bigcup_{n\in\mathbb{Z}}\hat{T}^n\big(p^{-1}(\{p_1,p_2\})\big)$ is the 0-summ set by the properties of the natural extension.

\section*{Acknowledgments}

\end{document}